\documentclass[reqno,a4paper]{article}

\usepackage{amssymb,amsmath,epsfig,graphics,graphicx,color,subfigure}

\definecolor{lightblue}{rgb}{0.22,0.45,0.70}
\usepackage[colorlinks=true,breaklinks=true,linkcolor=lightblue,citecolor=lightblue]{hyperref}
\usepackage{float}
\usepackage{mathrsfs}
\usepackage{comment}
\usepackage{tikz}
\usepackage{pgfplots}

\newtheorem{remark}{Remark}[section]
\newtheorem{lemma}{Lemma}[section]
\newtheorem{theorem}{Theorem}[section]
\newtheorem{proposition}{Proposition}[section]

\usepackage{soul}

\renewcommand{\O}{\Omega}

\newcommand\E{{K}}

\newcommand\Vh{{\calV}_h}

\def\qon{{\quad\hbox{on}\quad}}

\newcommand{\calA}{ \mathcal{A}}

\newcommand{\calC}{ \mathfrak{C}}

\newcommand{\calE}{ \mathcal{E}}
\newcommand{\calR}{ \mathcal{R}}

\renewcommand{\P}{{\mathcal P}}  % polynomials

\newcommand\R{\mathbb{R}}
\newcommand\N{\mathbb{N}}

\def\dof{{\rm dof}}

\def\bD{{\bf D}}

\def\ub{\boldsymbol{u}}

\def\bv{{\bf v}}
\def\vb{\boldsymbol{v}}

\def\wb{\boldsymbol{w}}
\def\fb{\boldsymbol{f}}
\def\nb{\boldsymbol{n}}
\def\tb{\boldsymbol{t}}

\def\Zb{{\bf Z}}
\def\H{{\bf H}}

\def\taub{{\boldsymbol \tau}}
\def\sigb{{\boldsymbol \sigma}}

\def\0{\boldsymbol{0}}
\def\div{{\rm div} \:}
\def\calD{\mathcal D}

\def\calN{\mathfrak N}

\def\curl{{\mathbf{curl}} \:}

\def\wtalpha{\widetilde{\alpha}}

\def\CT{\mathscr{T}}
\def\CTh{\mathscr{T}_h}
\def\calW{\mathcal{W}}
\def\calV{\mathcal{M}}

\newcommand{\calZ}{ \mathcal{Z}}

\def\c{{\rm C}}
\def\nc{{\rm NC}}

\def\inte{{\rm int}}
\def\bdry{{\rm bdry}}
\def\rI{{\rm I}}
\def\Ib{{\boldsymbol \rI}}

\def\VtK{\widetilde{\calV}_{h}(\E)}
\def\VK{\calV_{h}(\E)}

\def\HncT{H^{2,\nc}(\CTh)}

\def\CRS{\:\boldsymbol{\mathcal{U}}}

\def\HdoK{{H^{2}(\E)}}

\def\VV{\mathscr{V}}
\def\VVi{\VV_h^{\inte}}
\def\VVbdry{\VV_h^{\bdry}}
\def\EE{\mathscr{E}}
\def\EEh{\EE_h}
\def\EEi{\EE_h^{\inte}}
\def\EEbdry{\EE_h^{\bdry}}

\def\PiK{\Pi_{\E}^{\bD}}

\def\PioK{\Pi^{2}_{\E}}
\def\PimoK{\Pi^{0}_{\E}}

\def\PizeroK{\boldsymbol{\Pi}^0_{\E}}
\def\PiunoKb{\boldsymbol{\Pi}^{1}_{\E}}

\def\Pivec{\boldsymbol{\Pi}_{\E}^{\nablab}}

\def\CAh{C_{A_h}}
\def\CBh{C_{B_h}}
\def\CFh{C_{F_h}}
\def\CB{C_{B}}
\def\CBD{C_{{\tt bd}}}
\def\CR{C_{{\tt reg}}}
\def\CS{C_{{\tt sob}}}

\newcommand\rot{\mathop{\mathrm{rot}}\nolimits}

\newcommand\chib{\boldsymbol{\chi}}
\newcommand\bu{\boldsymbol{u}}

\def\reg{{\rm reg}}

\def\DVu{\mathbf{D}_{\calV}{1}}
\def\DVd{\mathbf{D}_{\calV}{2}} 

\def\DU{\mathbf{D}_{\CRS}}

\newcommand\nablab{\boldsymbol{\nabla}}
\newcommand\Deltab{\boldsymbol{\Delta}}

\newcommand*{\jump}[1]{\lbrack\hspace{-1.5pt}\lbrack #1\rbrack\hspace{-1.5pt}\rbrack}

\newenvironment{proof}{\noindent{\it Proof.}}{\hfill$\square$ \\}

\usetikzlibrary{arrows,decorations.markings}
\pgfplotsset{compat=1.5}
\usepackage{pgfkeys}
    \newenvironment{customlegend}[1][]{%
        \begingroup
        \csname pgfplots@init@cleared@structures\endcsname
        \pgfplotsset{#1}%
    }{%
        \csname pgfplots@createlegend\endcsname
        \endgroup
    }%
    \def\addlegendimage{\csname pgfplots@addlegendimage\endcsname}

\begin{document}

%***********************************************************************************
\title{The Morley-type virtual element method for the Navier-Stokes equations 
		in stream-function form on general meshes}
%***********************************************************************************

\author{D. Adak\thanks{GIMNAP, Departamento de Matem\'atica,
		Universidad del B\'io-B\'io, Concepci\'on, Chile.
		E-mail: {\tt dadak@ubiobio.cl}},\quad 
	D. Mora\thanks{GIMNAP, Departamento de Matem\'atica, Universidad
		del B\'io-B\'io, Concepci\'on, Chile and
		CI$^2$MA, Universidad de Concepci\'on, Concepci\'on, Chile.
		E-mail: {\tt dmora@ubiobio.cl}},\quad
	A. Silgado\thanks{GIMNAP, Departamento de Matem\'atica,
		Universidad del B\'io-B\'io, Concepci\'on, Chile.
		E-mail: {\tt alberth.silgado1701@alumnos.ubiobio.cl}}
}

\date{}
\maketitle

%***********************************************************************************
\begin{abstract}
The nonconforming Morley-type virtual element method for 
the incompressible Navier-Stokes equations formulated  
in terms of the stream-function on  simply connected polygonal
domains (not necessarily convex) is designed. 
A rigorous analysis by using a new \textit{enriching operator}  
is developed. More precisely, by employing such operator,  
we provide novel discrete Sobolev embeddings, 
which allow to establish  the well-posedness of the discrete scheme and obtain 
optimal error estimates in broken $H^2$-, $H^1$- and  $L^2$-norms  under 
\textit{minimal regularity} condition on the weak solution. 
The velocity and vorticity fields are recovered via a postprocessing formulas. 
Furthermore, a new algorithm for pressure recovery 
based on a Stokes complex sequence is presented. Optimal error estimates
are obtained  for all the postprocessed variables.
Finally, the theoretical error bounds and the good performance of the method
are validated through several benchmark tests.
\end{abstract}



\maketitle
%************************************************************************************************
\section{Introduction}
The two dimensional steady Navier-Stokes equations in its standard velocity-pressure 
form reads as:  given a sufficiently smooth force density 
$\fb :  \O \to \R^2$, find $(\ub,p)$ such that  
\begin{equation}\label{NSE:velocity:pressure}
	\begin{split}
		- \nu \Deltab\ub + (\nablab\ub)\ub + \nabla p=\fb,
		\qquad \div\ub&= 0\qquad\textrm{in}\quad\Omega,\\
		\ub=\0 \qquad\textrm{on}\quad \Gamma:=\partial \O,
		\qquad  (p,1)_{0,\O} &=0,
	\end{split}
\end{equation}
where $\ub: \Omega \to  \R^2$ is the velocity field, $p: \Omega \to  \R$ is the pressure 
fields  and  $\nu >0$ represents the fluid viscosity.
This system model the behaviour of a viscous incompressible fluid in the domain $\O$. 
The first and second equations in~\eqref{NSE:velocity:pressure} dictates the momentum and mass 
conservation of the fluid, while the third  identity indicates non-slip boundary conditions 
for the velocity field and the last equation represents the mean value of $p$ over $\O$ vanishing, 
which is used for the uniqueness of the pressure solution. Due to the important role it 
plays in the study of viscous incompressible flows, several numerical schemes have been 
developed to efficiently approximate the Navier–Stokes system. 
In particular, we are interested in discretizing this system by using 
general polygonal decompositions and  introducing the \textit{stream-function} of the 
velocity field.

In the last years, numerical methods for PDEs on polytopal meshes have received substantial 
attention. Different approaches have been proposed 
(see for instance \cite{BE2016} and the references therein),  offering 
significant flexibility in terms of dealing with complicated domains. 
Among them, we can find the Virtual Element Method (VEM), 
which was presented for first time in~\cite{BBCMMR2013}, as an evolution of 
mimetic finite differences  and a generalization of 
the Finite Element Method (FEM). The approach of VEM allows avoid an 
explicit construction of the discrete shape functions and this fact  
implies a high flexibility of the method, which is reflected, 
for instance in the ability to construct numerical schemes of high-order 
on general polygonal meshes (including ``hanging vertexes" and nonconvex shapes). 
Moreover, in the construction of discrete spaces with high-regularity  
and of schemes with the divergence-free property (in the context of fluid problems). 
In virtue of these features, the VEM technology has enjoyed extensive 
success in numerical modeling and engineering applications, both in its conforming and 
\textit{nonconforming}  approaches (see for instance \cite{BM13,CMS2016,BBMR2016,ABSV2016,BMRR2017,ALKM2016,CKP2022-NM}). 
 In particular, many works have been devoted to solving problems 
in fluid mechanics by using the VEM.
Below are two representative list works in the conforming and 
nonconforming cases; 
~\cite{ABMV2014,GSMM3AS2018,BMV2019,BLV-NS18,AMNS2021-M2AN} and 
\cite{CGM2016,LC19,ZZMC19,LLN2020}, respectively. 
For a  current state of the art on VEM,  we refer to book~\cite{Book_VEM2022}.

In~\cite{AMV2018,ZZCM2018} the authors have introduced 
fully-nonconforming VEMs of high-order, independently and by using  different approaches
to solve biharmonic problems. In particular, the lowest-order configuration (i.e., $k=2$) 
of these  VEMs, can be consider as the extension of the popular Morley FE \cite{Morley} to general 
polygonal meshes. Since then, several schemes and analysis based on these VEMs have been developed for linear problems; see for instance~\cite{CH2020,LZHC2021,WZ2021-IMAJNA,DH2021-IMAJANA,JY2021,CKP-arxiv2022,AMS2022-CI2MA}.
In the present work we are interesting to extend the Morley-type VEM 
to solve the  nonlinear fourth-order Navier-Stokes equations in stream-function 
form  on simply connected domains (not necessarily convex) by using  general polygonal decompositions. 

Typically, the velocity-pressure formulation~\eqref{standard:weak:NS} is the most used to 
discretize the Navier-Stokes problem. 
However, the stream-function formulation has shown to be a competitive alternative
to discretize flow problems, which has been the focus of study in the last decades. 
In particular, we can highlight the following features: the system is reduced in a singular 
scalar weak formulation, with automatic satisfaction of the incompressibility constrain 
(the velocity field is equal to the curl of the stream-function),
the possibility to recover further variables of interest such as the velocity, vorticity and
pressure fields by postprocessing  from the stream-function.
Besides, for nonlinear problems, the resulting trilinear form is naturally skew-symmetric, 
allowing more direct stability and convergence arguments.
On the other hand, the stream-function approach avoid the difficulties related with the boundary 
values for the vorticity field, which are present in stream-function–vorticity formulation.
Due to the attractive features discussed above, over past decades  the stream-function formulation  
has received great attention from many researchers. In particular, in the area  of Numerical Analysis
several works have been devoted to the development and study of efficient numerical schemes to approximate
this system. For instance; conforming and nonconforming FEMs in \cite{CN86,CN89,F2003,CMN2021}, 
bivariate spline \cite{LW2000}, $hp$-version discontinuous FE \cite{MSB2016}, 
NURBS-based Isogeometric Analysis in \cite{TDQ2014}.
Moreover, in~\cite{KPK-CMAME2021} the nonconforming Morley FEM
have been used to solve the steady Quasi-Geostrophic equations, which can be seen as an extension 
(in form) of the two dimensional Navier-Stokes equations in stream-function formulation.

In the present contribution we configure the Stokes complex structure of the nonconforming VEM 
introduced  in \cite{ZZMC19} to solve the fourth-order nonlinear Navier-Stokes equations 
in stream-function form on domains not necessarily convex and employing general polygonal 
partitions of the domain, allowing  additionally the reconstruction of the pressure field. 
By using the enhancement technique, we introduce a discrete Stokes complex structure associate
to the Morley-  and  Crouzeix-Raviart-type VE spaces. Then, we 
construct suitable projections useful to build the discrete trilinear form, which mimics the 
interesting and \textit{naturally skew-symmetry} property of the continuous version.
In order to establish the well-posedness of the discrete nonconforming formulation, 
is necessary to prove  the continuity of the resulting discrete trilinear form respect to 
the natural norm in the Morley-type VE space $\Vh$. However, this fact does not follow directly, since it involves a discrete Sobolev inclusion (namely, $\Vh \subset W^{1,4}(\O)$).  The derivation of the Sobolev embeddings require 
particular attention for the nonconforming approach, which is usually considered a challenging task.
To the best of our knowledge, this is the first work where Sobolev embeddings for the Morley-type VE 
space are established. More precisely, with the aim of achieving such purpose, we introduce a novel 
\textit{enriching operator}, which is a special kind of quasi-interpolation operator that maps the elements 
of the sum space between the continuous and nonconforming spaces (namely, $\Phi + \Vh$)  
to the conforming counterpart of the nonconforming space.
Then, by using this operator and its approximation properties we provide 
new discrete Sobolev embeddings for the sum space $\Phi + \Vh$ and we prove the 
well-posedness of the discrete problem by using the fixed point Banach Theorem.

It well know that due to nonconformity of the space increases the technicalities in the demonstrations  
of error estimates in the nonconforming approach, implying in some cases high-regularity of the solution,
which are not realistic. Furthermore, for nonlinear problems these difficulties increase remarkably. 
In the present work, by employing the naturally skew-symmetry  property of 
the discrete trilinear form and  the discrete Sobolev inclusion, we write elegantly an abstract convergence result for the nonlinear VE scheme. Then, by exploiting again the enriching operator, we establish key approximation properties 
involving the bilinear and trilinear forms, together  with the consistency 
errors, allowing the derivation of an optimal error estimate in broken  $H^2$-norm under the 
\textit{minimal regularity} condition on the weak stream-function solution  (see below Theorem~\ref{regu-aditional}).
In addition, by using duality arguments and the enriching operator we also provided new optimal error estimates in the $H^1$- and $L^2$-norm under the same regularity condition on the stream-function and  the density force.

On the another hand, by exploiting the stream-function approach, we present techniques to recover further variables 
of physical interest, such as, the primitive velocity and pressure variables, along with the important vorticity field. More precisely, we recover the velocity and vorticity fields  
through a postprocess of the discrete stream-function by using adequate polynomial projections,  
which are directly computable from the degrees of freedom. 
The pressure recovery procedure require a special attention. Indeed, we approximate the 
fluid pressure by exploiting the Stokes complex sequence associate to the Morley-  and  Crouzeix-Raviart-type VE spaces,  and solving an additional  Stokes-like system with right hand side coming from the virtual 
stream-function solution and the force density $\fb$.  For all the postprocessed variables, 
we provide optimal a priori error estimates.
Furthermore the numerical method is tested with several benchmark tests, 
including the Kovasznay and cavity problems, where  the theoretical accuracy 
and the good performance of the scheme are corroborated.
Finally, we expect that the results reported here constitute a stepping-stone towards  
for the development and analysis of new numerical schemes based on the Morley-type VEM,
for solving fourth-order related problems, allowing  the derivation of optimal order 
error estimates in different broken norms, under less regularity assumptions of the solution 
in more complicated situations, such as, nonlinear coupled and/or time dependent systems.

The outline  of the remaining parts of this paper reads as follows: 
in Section~\ref{cont_problem} we introduce some preliminaries notations and the stream-function 
weak formulation of the Navier-Stokes problem~\eqref{NSE:velocity:pressure}. 
Moreover, we recall its well-posedness and regularity property. 
The Morley-type VE discretization, together with the  
Crouzeix-Raviart VE space  are described in Section~\ref{VEM-disc}. 
In Section~\ref{Discrete-Problem} we introduce the enriching operator, provide 
the discrete Sobolev embeddings and the well-posedness of the  discrete problem
by using a fixed-point strategy. 
In Section~\ref{Error:Analisis} we develop the error analysis of the scheme under minimal regularity 
condition on the weak solution. 
In Section~\ref{rec:fields} we describe the recovery techniques for the velocity, vorticity and pressure fields 
by using the discrete stream-function solution. Finally, several numerical tests on different polygonal meshes 
are reported in Section~\ref{numeric_Result}.

%-----------------------------------------------------------
\section{Preliminaries and continuous weak form}\label{cont_problem}
\setcounter{equation}{0}

\subsection{Notations}
In this subsection we introduce notations that we will use along the paper, including those already
employed above. We will follow the standard notations of Sobolev spaces and 
their respective seminorms and norms according~\cite{AF2003}. Hence, 
for every open bounded domain $\calD$, 
the   seminorms and norms  in the spaces $L^q(\calD)$ and $W^{\ell,q}(\calD)$ 
(with $\ell \geq 0$ and $q \in [1,+ \infty)$), 
are denoted by $|\cdot|_{\ell,q,\calD}$ and  $\|\cdot\|_{\ell,q,\calD}$, respectively.
%$|\cdot|_{W_p^{t}(\calD)}$ and  $\|\cdot\|_{W_p^{t}(\calD)}$, respectively.
We adopt the usual convention $W^{0,q}(\calD):=L^q(\calD)$. In particular when $q=2$, 
we  write $H^{\ell}(\calD)$  instead to  $W^{\ell,2}(\calD)$ and  the corresponding
convention for the seminorms and norms is also adopted, i.e.,   
$|\cdot|_{\ell,\calD}$  and  $\|\cdot\|_{\ell,\calD}$, respectively.

For any tensor fields $\taub=(\tau_{ij})_{i,j=1,2}$ and  $\sigb=(\sigma_{ij})_{i,j=1,2}$,
we consider the standard scalar product of $2 \times 2$-matrices: 
$\taub : \sigb = \sum_{i=1}^{2} \tau_{ij} \sigma_{ij}$ and for simplicity the scalar, 
vectorial and tensorial $L^2$-inner products will be denoted by
\begin{equation*}
	(\varphi,\phi)_{0,\calD}= \int_{\calD} \varphi \phi \qquad
	(\vb,\wb)_{0,\calD}= \int_{\calD} \vb \cdot \wb 
	\qquad  (\taub,\sigb)_{0,\calD}= \int_{\calD} \taub : \sigb. 
\end{equation*}
 
Moreover, with the usual notations, for scalar functions the symbols 
$\nabla$, $\Delta$, $\Delta^2$ and $\bD^2$  denote the gradient, 
Laplacian, Bilaplacian operators and the Hessian matrix, respectively, 
while  the bold symbols $\nablab$ and $\Deltab$ denote the gradient 
and Laplacian operators  for vector fields, respectively. 
In addition, for smooth scalar and vectorial functions $\phi$ and 
$\vb=(v_1,v_2)$, we define the curl, divergence and rotational operators, as follow:
\begin{equation*}
\begin{split}
\curl \phi:=\begin{pmatrix} 
\quad \partial_y\varphi \\ 
-\partial_x\varphi
\end{pmatrix}&,
\qquad
\div\vb:=\partial_x
v_1+\partial_y v_2, \quad 
\quad \text{and} \quad	
\rot\vb:=\partial_x
v_2-\partial_y v_1.   
\end{split}
\end{equation*}

Henceforth, $\O$ will denote a simply connected  bounded domain 
of $\R^2$ with polygonal Lipschitz boundary $\Gamma:=\partial\O$.
The symbol $\nb=(n_i)_{1\le i\le2}$ is the outward unit 
normal vector to the boundary $\Gamma$, 
while the vector $\tb=(t_i)_{i=1,2}$ is the unit tangent 
to $\Gamma$ oriented such that $t_1 =-n_2$, $t_2 = n_1$. Moreover,
$\partial_{\nb}\phi= \nabla \phi \cdot \nb$ and $\partial_{\tb}\phi= \nabla \phi  \cdot \tb$
denote the normal and tangential derivatives, respectively.
In addition, $c$ or  $C$, with or without subscripts,
will represent a generic constant, which is independent of the 
mesh parameter $h$ that might have distinct values at different places.
 %, assuming different values in different occurrences.

\paragraph{The Navier-Stokes in velocity-pressure weak form.}
The standard variational formulation of problem~\eqref{NSE:velocity:pressure} reads as:
find $(\ub,p) \in \H \times Q$, such that 
\begin{equation}\label{standard:weak:NS}
\begin{split}
\nu (\nablab \ub,\nablab \vb)_{0,\O}+((\nablab\bu)\ub,\vb)_{0,\O}
		- (p, \div \vb)_{0,\O}  &= (\fb,\vb)_{0,\O} \qquad \forall \vb \in \H,\\%\label{eq1.1}\\
		- (g, \div \ub)_{0,\O}&= 0 \qquad \qquad \quad \: \forall g \in Q, \\%\label{eq1.2}\\
	\end{split} 
\end{equation}
where  the Hilbert spaces $\H$ and $Q$ are defined by:
\begin{equation}\label{spaces:H1:L20}
\H:=\left\{\vb \in H^1(\O)^2: \vb=\0 \qon \Gamma \right\}	\qquad \text{and} \qquad
Q:=\left\{g \in L^2(\O): (g,1)_{0,\O}=0  \right\}.
\end{equation}
It is well known that problem \eqref{standard:weak:NS} admits a unique solution (see \cite{GR}) 
under smallness assumption on the data.
Moreover, several works have been devoted to develop numerical schemes to approximate this formulation. 
For instance, see \cite{BLV-NS18,FM2020-IMAJNA,LC19,ZZM2021} in the VEM context.

%In the remainder of this work,
In this work, we will study the Navier-Stokes equations with a different approach.
More precisely, under assumption that the domain is simply connected and by using 
the incompressibility condition of the velocity field (i.e., $\div \ub =0$), 
we write an equivalent variational formulation in terms of the \textit{stream-function} of the velocity field.
 \subsection{The stream-function weak form}

Since $\O\subset \R^2$ is simply connected, is well known 
that a vector function $\vb \in \Zb:= \left\{\vb \in\H: \div \vb =0 \right\}$ 
if and only if there exists a function $\varphi \in H^2(\O)$ 
(called \textit{stream-function}), such that $\vb =\curl \varphi$.

Let us consider the following Hilbert space 
$\Phi:=\left \{\varphi \in H^2(\O) : \varphi =0, \: \partial_{\nb}\varphi = 0  \qon \Gamma \right\}$,
and we endow this space  with the norm 
$\Vert\varphi\Vert_{2,\O}:=\left(  \bD^2 \varphi, \bD^2 \varphi\right)_{0,\O}^{1/2} \quad\forall\varphi\in \Phi$.
Then, we have that a variational formulation  
of problem~\eqref{NSE:velocity:pressure}, formulated in terms of  stream-function, read as (see for instance~\cite[Section 10.4]{QV}):
given $\fb \in L^2(\O)^2$, find $\psi\in \Phi$, such that
\begin{equation}\label{weak:stream:NS}
\nu A(\psi,\phi)+B(\psi;\psi,\phi)=  F(\phi) \qquad \forall \phi \in \Phi, 
\end{equation}
where the multilineal forms $A: \Phi \times \Phi \to \R$, $B:\Phi \times\Phi \times \Phi \to \R$
and $F: \Phi \to \R$ are defined by:  
\begin{align}
		A(\psi,\phi) &:=(\bD^2 \psi,\bD^2 \phi)_{0,\O}  \label{A-cont}, \\ 
	B(\zeta;\psi,\phi)& :=(\Delta \zeta \:
	\curl \psi, \nabla\phi)_{0,\O}, \label{B-cont}\\
	F(\phi) &:= (\fb,\curl\phi)_{0,\O}. \label{loadcont}	
\end{align}

From the definition of the bilinear form $A(\cdot,\cdot)$ and equivalence of norms, 
we obtain its $\Phi$-ellipticity. 
Moreover, by using the Cauchy-Schwarz  inequality is easily obtain: 
\begin{equation*}
\begin{split}
|A(\varphi,\phi)|  &\leq  \,  \|\varphi\|_{2,\O} \|\phi\|_{2,\O} \quad \quad \quad  \forall \varphi, \phi \in \Phi,\\
|F(\phi)| &\leq \,   C_F\|\fb\|_{0,\O} \|\phi\|_{2,\O}  \qquad  \forall  \phi \in \Phi,
\end{split} 
\end{equation*}
where $C_F$ is a positive constant. Now, we recall the following continuous Sobolev inclusion: 
for all  $\vb \in H^1(\O)^2$, 
there exists $\widetilde{C}_{{\rm sob}}>0$  such that
\begin{equation}\label{sob:cont}
\|\vb\|_{L^4(\O)} \leq  \widetilde{C}_{{\rm sob}}\|\vb\|_{1,\O}.
\end{equation}
Then, by using the H\"{o}lder inequality and the above inclusion, there exists 
$\CB:= \widetilde{C}_{{\rm sob}}^2>0$, such that 
\begin{equation*}
\begin{split}
|B(\zeta;\varphi,\phi)| &\leq \CB \, \|\zeta\|_{2,\O} \|\varphi\|_{2,\O} \|\phi\|_{2,\O} 
\qquad \forall \zeta, \varphi, \phi \in \Phi.
\end{split} 
\end{equation*}

From the above properties and the  fixed-point  Banach Theorem,
we can prove that problem~\eqref{weak:stream:NS} is well-posed.
More precisely, we have the following existence and uniqueness 
result (see for instance, \cite[Chapter IV, Section 2.2]{GR}).
\begin{theorem} 
	If $\CB C_{F} \nu^{-2} \|\fb\|_{0,\O} <1$,
%	\begin{equation}\label{lambda}
%		\lambda:=\CB C_{F} \nu^{-2} \|\fb\|_{0,\O} <1,
%	\end{equation}
	then there exists a unique $\psi\in \Phi$ solution to problem~\eqref{weak:stream:NS},
	which satisfies the following continuous dependence on the data
	\begin{equation*}%\label{depend:cont}
		\|\psi\|_{2,\O}\le C_{F}\nu^{-1} \|\fb\|_{0,\O}.	
	\end{equation*}
\end{theorem}

Now, we state an additional regularity result
for the solution of problem~\eqref{weak:stream:NS} (see for instance \cite{BR80}).

\begin{theorem}\label{regu-aditional}
Let $\psi \in \Phi$ be the unique solution of problem \eqref{weak:stream:NS}.
Then, there exist $\gamma\in(1/2,1]$ and $C_{\reg}>0$,
such that $\psi\in H^{2+\gamma}(\O)$ and
\begin{equation*}
\|\psi\|_{2+\gamma,\O} \leq C_{\reg}\|\fb\|_{0,\O}.
	\end{equation*}	
\end{theorem}

\section{Morley-type virtual element approximation}\label{VEM-disc}
\setcounter{equation}{0}
This section is devoted to the construction of a VEM
to solve problem~\eqref{weak:stream:NS}. We  will introduce
a Morley-type VE space by using some auxiliaries local virtual spaces and 
the enhancement technique. More precisely, the present framework is based on the  discrete Stokes 
complex sequence for the {\it Morley}- and  {\it Crouzeix-Raviart}-type VE spaces presented in~\cite{ZZMC19}.
This Stokes complex structure will allow us to approximate the main unknown in problem~\eqref{weak:stream:NS}
and as an important topic, also it will allow to compute the pressure variable of the Navier-Stokes 
system~\eqref{NSE:velocity:pressure} as a postprocess,  by solving a Stokes-like problem 
with right hand side coming form the discrete-stream function solution and  force density $\fb$ 
(cf. subsection~\ref{rec:pressure}).

We start with a subsection introducing the polygonal decompositions 
and some useful notations, these preliminaries are
following by a subsection on the local and global 
nonconforming virtual spaces, their degrees of freedom and the classical VEM local projectors.
Later on, we introduce other polynomial projections useful to  build  the discrete trilinear form. 

\subsection{The polygonal decompositions and basic setting}
\label{Notations}
Let $\{\CTh\}_{h>0}$ be a sequence of decompositions of $\O$ 
into general non-overlapping simple polygons $\E$, 
where $h:=\max_{\E\in\CTh}h_\E$ and $h_\E$ is the diameter of $\E$.  
We will denote by  $\partial \E$, $N_{\E}$ and $|\E|$ the boundary, 
the number of vertices and area of each polygon $\E$, respectively.

For each element $\E$ we denote by $\EEh^\E$ the set of its edges, while  the set of  all the 
edges in $\CTh$ will be denote by $\EE_h$. We decompose this set as
the following union: $\EE_h := \EEi \cup \EEbdry$, where $\EEi$ and $\EEbdry$ are the set of
interior and boundary edges, respectively.  For the set of all the vertices we have an analogous 
notation. More precisely,  we will denote  by  $\VV_h := \VVi \cup \VVbdry$ the set of vertices
in $\CTh$, where  $\VVi$ and $\VVbdry$ are the set of interior and boundary vertices, respectively. 
In addition, we denote by $e$ a generic edge of $\EE_h$ and by $h_e$ its length.   
 
Besides, for each $\E \in \CT_h$, we denote by $\nb_{\E}$ its unit outward normal vector 
and by $\tb_{\E}$ its tangential vector along the boundary $\partial \E$. 
Moreover, we will adopt the notation $\nb_e$ and $\tb_{e}$
for a unit normal and tangential vector of an edge $e \in \EE_h$, respectively.

For every $\ell>0$ and $q\in [1,+\infty)$, we define the following broken Sobolev spaces
\begin{equation*}
W^{\ell,q}(\CTh):= \{ \phi \in L^2(\O): \phi|_{\E} \in W^{\ell,q}(\E)  \quad \forall \E \in \CTh\},
\end{equation*}
and we endow these spaces with the following broken seminorm:
\begin{equation*}
|\phi|_{\ell,q,h}:=\Big( \: \sum_{\E\in\CTh}|\phi|_{\ell,q,\E}^{q} \Big)^{1/q},
\end{equation*}
where $|\cdot|_{\ell,q,\E}$ is the usual seminorm in $W^{\ell,q}(\E)$. 
When $q=2$, we omit $q$ and write $H^{\ell}(\CTh)$ instead $W^{\ell,2}(\CTh)$, with the corresponding
seminorm denoted by $|\cdot|_{\ell,h}$.

Next, we will define the jump operator. First, for each $\phi_h \in H^2(\CTh)$,  
we denote by $\phi_h^{\pm}$  the trace of $\phi_h|_{\E^{\pm}}$, with 
$e \subset \partial \E^{+} \cap \partial \E^{-}$. 
Then, the jump operator $\jump{\cdot}$ is defined as follows:
 \begin{equation*}
 \jump{\phi_h} 
 	:= \begin{cases}
 	\phi_h^{+}-\phi_h^{-} & \text{for every  $e \in \EEi$,}\\
 	\phi_h|_e  &  \text{for every  $e \in \EEbdry$}.
 	\end{cases}
 \end{equation*}
The same notation is adopted for vectorial fields. Let us define a subspace of $H^2(\CTh)$ with certain continuity: 
\begin{equation*}
\begin{split}
\HncT := \Big\{ \phi_h \in H^2&(\CTh) : \phi_h  \in C^0(\VVi), \quad \phi_h(\bv_i)=0 \quad  \forall \bv_i \in \VVbdry, \quad  (\jump{\partial_{\nb_e} \phi_h},1)_{0,e}=0 \quad \forall  e \in \EEh \Big\},
\end{split}
\end{equation*}
where $C^0(\VVi)$ is the set of functions continuous at internal vertexes.

Finally, for each subset $\calD \subset\R^2$ and 
every integer $\ell\geq 0$,  $\P_{\ell}(\calD)$ is the space of 
polynomials of degree up to $\ell$ defined on $\calD$. 
Furthermore, the piecewise $\ell$-order polynomial space is defined by:
\begin{equation*}
	\P_{\ell}(\CTh):= \{ \chi \in L^2(\O): \chi|_{\E} \in \P_{\ell}(\E)  \quad \forall \E \in \CTh\}.
\end{equation*}

In what follows, we will introduce some preliminary spaces, which are useful to construct
the Morley-type VE space to approximate the solution of problem~\eqref{weak:stream:NS}. 
%the discrete Stokes complex sequence.
\subsection{Some auxiliary spaces}
%\subsection{Local and global nonconforming virtual element space}
For every polygon $\E\in\CTh$, first we consider the following auxiliary finite dimensional space~\cite{AMV2018,ZZCM2018,LZHC2021}:
\begin{align*}
\VtK
:= \left\{\phi_h\in \HdoK : \Delta^2\phi_h\in\P_{2}(\E),  \: \:\phi_h|_e\in\P_{2}(e),
\: \: \Delta \phi_h|_e\in\P_{0}(e)\quad \forall e\in\partial\E\right\}.
\end{align*}

Next, for a given $\phi_h\in\VtK$, we introduce the following sets:
\begin{itemize}
	\item $\DVu$: the values of $\phi_h(\bv_i)$ for all vertex $\bv_i$ of the polygon $\E$;
	\item $\DVd$: the edge moments $( \partial_{\nb_{e}}\phi_h,1)_{0,e} \quad  \forall \, 
	\text{edge} \, \, e \in \EE_h^{\E}$.
\end{itemize}

For each polygon $\E$, we define the following projector $\PiK:\VtK  \to \P_2(\E)\subseteq\VtK$,
as the solution of the local problems:
\begin{align*}
	A^{\E}( \PiK \phi_h, \chi) & = A^{\E}(\phi_h , \chi) \quad  \forall \chi \in \P_{2}(\E), \\
	\langle \langle \: \PiK \phi_h, \chi \: \rangle \rangle_{\E}&=\langle \langle \phi_h, \chi \rangle \rangle_{\E} \quad  \forall \chi \in \P_{1}(\E), 
\end{align*}
%\end{subequations}
where $\langle \langle \varphi_h, \phi_h \rangle \rangle_{\E}$ is defined as follows:
\begin{equation*}%\label{average}
\langle \langle \varphi_h, \phi_h \rangle \rangle_{\E}:= \sum_{i=1}^{N_\E}
	\varphi_h (\bv_i)	\phi_h (\bv_i),
\end{equation*}
with $\bv_i$, $1\le i \le N_{\E}$, being the vertices of $\E$ and $A^{\E}(\cdot,\cdot)$ 
is the restriction of the  continuous form $A(\cdot,\cdot)$ (cf.~\eqref{A-cont}) on the element $\E$.

The operator $\PiK :\VtK \to \P_{2}(\E)$ is explicitly
computable for every $\phi_h \in \VtK$, using only the
information of the linear operators $\DVu-\DVd$ 
(for further details, we refer to~\cite{ZZCM2018}).

Now, we will introduce another auxiliary local spaces. Indeed, following \cite{ZZMC19} we define the spaces: 
\begin{equation*}%\label{CR1}
	\widehat{\CRS}(\E):= \Big\{\vb_h \in H^1(\E)^2 : \div \vb_h \in \P_0(\E), \: \rot \vb_h \in \P_0(\E), \: \vb_h \cdot \nb_{e} \in \P_1(e) \quad \forall e \in \EEh^\E \Big\}, 
\end{equation*} 
and
\begin{equation*}%\label{CR2}
	\widetilde{\calZ}(\E):= \Big\{ \phi \in H^2(\E) : \Delta^2 \phi_h =0,\: \phi_h|_e =0, \:\Delta \phi_h|_e \in \P_0(e) \quad \forall e \in \EE_h^{\E} \Big\}.
\end{equation*}

By adding $\widehat{\CRS}(\E)$ and $\curl$ of the functions belongs to $\widetilde{\calZ}(\E)$, we define the space
\begin{equation*}
	\CRS_0(\E):=\widehat{\CRS}(\E)+ \curl(\widetilde{\calZ}(\E)).
\end{equation*} 
Then, for each $\vb_h \in \CRS_0(\E)$  we introduce the set of vector-valued, bounded linear functional
\begin{itemize}
\item $\DU$: the edge moments $h^{-1}_e  (\vb_h, \boldsymbol{1})_{0,e} \qquad \forall e \in \EE_h^{\E}.$
\end{itemize} 
We observe that $\P_1(\E)^2 \subset\CRS_0(\E)$, and we introduce the grad projection operator $\Pivec: \CRS_0(\E) \rightarrow \P_1(\E)^2$ as the solution of the following problem:
\begin{equation}\label{proj:H1:vec}
\begin{split}
&( \nablab (\Pivec \vb_h-\vb_h), \nablab\chib)_{0,\E} =0 \quad \forall \chib \in \P_1(\E)^2, \\
& (\Pivec \vb_h-\vb_h,1)_{0,\partial \E}=0.
\end{split}
\end{equation}

By using an integration by parts, we can deduce that the polynomial $\Pivec \vb_h$ is computable for all $\vb_h \in \CRS_0(\E)$ from the set of values $\DU$ (see \cite{ZZMC19}). 

Next, by employing the grad projection operator $\Pivec$, we define the local Crouzeix-Raviart-like  VE space 
\begin{equation*}\label{local:CR:VE}
	\CRS_h(\E):=\Big\{\vb_h \in \CRS_0(\E) :  (\vb_h \cdot \nb_e - \Pivec\vb_h  \cdot \nb_e, \chi)_{0,e}  \quad \forall \chi \in \P_1(e) \setminus \P_0(e), \quad \forall e \in \EE_h^{\E}\Big\}.
\end{equation*}

Further, from \cite{ZZMC19} we have that the set $\DU$ characterize  uniquely the functions of $\CRS_h(\E)$.  
Moreover, for each  $\phi_h \in \VtK$, the function  $ \Pivec \curl\phi_h $  is  computable using the sets $\DVu$ and $\DVd$. %(see \cite[Lemma 18]{ZZMC19}).

The global Crouzeix-Raviart-like space is  defined  as follows~\cite{ZZMC19}:
\begin{equation}
	\label{CR_Glob}
	\CRS_h:= \Big\{\vb_h \in L^2(\Omega)^2 : \vb_h|_{\E} \in \CRS_h(\E) \quad \forall \E \in \CT_h, \quad (\jump{\vb_h}, \boldsymbol{1})_{0,e} =0  \quad \forall e \in \EE_h \Big\}.
\end{equation}

We have that the dimension of the space $\CRS_h$ is equal to $2N^{\EE_h}$, where $N^{\EE_h}$ is the total number of mesh edges of the discretization $\CTh$. This space will be useful in subsection~\ref{rec:pressure} to present the pressure recovery technique. 

\begin{remark}
	The nonconforming VE space defined in \eqref{CR_Glob} coincides with the Crouzeix-Raviart finite element space when the polygon $K$ is a triangle. Therefore, this space can be seen as an extension of the classical Crouzeix-Raviart space from triangle to polygonal element in the nonconforming VEM context. For further details  of this discussion, see~\cite[Remark 8]{ZZMC19}.
\end{remark}

\subsection{The Morley-type nonconforming virtual element space}
By using the auxiliary spaces defined in the above subsection, for each $\E \in \CT_h$  
we introduce the local Morley-type VE space~\cite{ZZMC19}:
\begin{equation}\label{local:space:nc}
\begin{split}
\VK := \Big\{ \phi_h \in \VtK &:  (\curl\phi_h \cdot \nb_e-\Pivec (\curl\phi_h \cdot \nb_e), \chi)_{0,e}=0 \quad \forall \chi \in \P_1(e)\setminus \P_0(e)  \quad \forall e \in\EEh^\E, \\
&\quad  (\phi_h- \PiK \phi_h,\chi)_{0,\E} =0 \quad \forall \chi \in  \P_{2}(\E)\Big\}.
\end{split}
\end{equation}
In the next result we summarize the main properties of the local Morley-type VE space.  
\begin{lemma}\label{lemma:prop:Morley}
For each polygons $\E$, the space $\VK $ defined in \eqref{local:space:nc}, we have  $\P_2(\E) \subseteq \VK$. 
Moreover, we can deduce the following properties: 
	\begin{itemize}
\item The linear operators $\DVu-\DVd$ constitutes 
	   a set of degrees of freedom for $\VK$;
\item The operator $\PiK :\VK \to \P_{2}(\E)$ is 
     computable using the sets $\DVu-\DVd$;
 \item  For each  $\phi_h \in \VK$, the function  $ \Pivec \curl\phi_h $  
     is  computable using the degrees of freedom $\DVu-\DVd$.
	\end{itemize}
\end{lemma}

With the above preliminaries we can introduce the global Morley-type VE  space to  
the numerical approximation of the problem~\eqref{weak:stream:NS}. 
Indeed, for every decomposition $\CTh$ of $\O$ into polygons $\E$, 
the global nonconforming VE space is given by:
\begin{equation}\label{global:space:fnc}
\begin{split}
\Vh := \left\{ \phi_h \in H^{2,\nc}(\CT_h): \phi_h|_{\E} \in \VK, \quad \forall  \E \in \CT_h\right\}.
\end{split}
\end{equation}

We have that $\Vh \subset H^{2,\nc}(\CT_h)$, but $\Vh \nsubseteq \Phi$. 
Moreover, we observe that the nonconforming  VE does not require 
that the $C^0$-continuity over $\O$. This space can be seen as an extension 
of the popular Morley FE \cite{Morley} to general polygonal meshes. 
For further details about this discussion, we refer to~\cite[Remark 20]{ZZMC19} 
and \cite[Remark 4.1]{ZZCM2018}.

For the continuous bilinear form $A(\cdot,\cdot)$, we adopt the following notation:
\begin{equation*}%\label{notation:AD}
	A(\varphi_h, \phi_h) := \sum_{\E \in \CT_h} A^{\E}(\varphi_h, \phi_h) 
	\qquad \forall \varphi_h, \phi_h \in \Phi+\Vh.
\end{equation*}
We also adopt the same notation by the continuous forms $B(\cdot;\cdot,\cdot)$ and $F(\cdot)$.

\subsection{Polynomial projection operators and discrete multilinear forms}
\label{projec:and:forms}
This subsection is dedicated to the presentation of other important polynomial projections, 
along with the construction of the trilinear form and the load term, by using such projections. 
Moreover, we build the bilinear discrete form.

For each $m \in \N \cup \{0\}$, we consider the usual $L^2$-projection, 
$\Pi_{\E}^{m}: L^2(\E)  \to \P_{m}(\E)$, defined by the function 
such that 
\begin{equation}\label{proymdos}
( \phi-  \Pi_{\E}^{m} \phi,\chi )_{0,\E}=0 
\qquad\forall \chi\in\P_{m}(\E).
\end{equation}
Moreover, we define its vectorial ${\bf \Pi}_\E^m$ version in an analogous way. 
For the projection previously defined  we have the following result.

We recall that there exists $\CBD>0$ such that (see \cite{BLV-NS18}):
\begin{equation}\label{bound:proj}
\| \Pi_{\E}^{m} \phi \|_{L^4(\E)} \leq \CBD \| \phi \|_{L^4(\E)} \qquad \text{and} \qquad
\| \Pi_{\E}^{m} \phi \|_{0,\E} \leq \| \phi \|_{0,\E}  \quad \forall \phi \in L^2(\E). 
\end{equation}
\begin{lemma}\label{lemm-PioK}
Let $\PioK,\PimoK$ and $\PiunoKb$ be the operators defined by relation \eqref{proymdos} and by its vectorial 
version. Then,  for each $\phi_h \in \VK$,  the polynomial functions	
$\PioK \phi_h,  \PimoK \Delta \phi_h, \PiunoKb\curl \phi_h$ and $\PiunoKb\nabla  \phi_h$
are computable using only the information of the degrees freedom $\DVu-\DVd$.
\end{lemma}
\begin{proof}
Let $\phi_h \in \VK$, the proof of the function $\PioK \phi_h$ 
follows from the definition of the space $\VK$ (cf. \eqref{local:space:nc}). 
Moreover, using integration by parts  we obtain 
\begin{equation*}
\begin{split}
(\curl \phi_h,\boldsymbol{\chi})_{0,\E} 
=\rot \boldsymbol{\chi}(\PioK \phi_h,1)_{0,\E} -(\phi_h,\chib\cdot \tb_{\E}))_{0,e} 
\qquad \forall \chib \in \P_{1}(\E)^2,
   	\end{split}
   \end{equation*}
then we also conclude that the 
$\PiunoKb\curl \phi_h$ is fully computable from the degrees of freedom.
Similarly, we prove that function $\PiunoKb\nabla \phi_h$
is computable from the degrees of freedom $\DVu-\DVd$.
   
Next, we will prove that the polynomial function 
$\PimoK \Delta \phi_h$ is also computable. Indeed,
using integration by parts, we have 
\begin{equation*}%\label{Pik-2lap}
	\PimoK \Delta \phi_h  = |\E|^{-1} (\partial_{\nb_{\E}}  \phi_h,1)_{0,\partial\E}=
	|\E|^{-1}\sum_{e\in \partial \E}  (\partial_{\nb_e}  \phi_h,1)_{0,e},
\end{equation*}
and note that the above integral is computable  using the output values of the set $\DVd$. 
\end{proof}

In this part, we will build the discrete version
of the continuous  forms defined in \eqref{A-cont}, \eqref{B-cont} and 
\eqref{loadcont} using the operators introduced previously. 
First, we consider the following discrete local bilinear form, 
$A_{h}^{\E}: \VK \times \VK \to \R$ approximating the continuous form $A(\cdot, \cdot)$:
\begin{equation}\label{disc-bili:form}
A_{h}^{\E}(\varphi_h,\phi_h):=
A^{\E} \left(\PiK \varphi_h,\PiK \phi_h\right) +S_{\bD}^{\E}\big((\rI-\PiK) \varphi_h,(\rI-\PiK) \phi_h\big) \quad  \forall  \varphi_h,\phi_h\in\VK.
\end{equation}
where $S_{\bD}^{\E}(\cdot,\cdot)$ is any symmetric positive definite bilinear form to be 
chosen as to satisfy:
\begin{equation}\label{term-stab-SK}
c_{*} A^{\E}(\phi_h,\phi_h)\leq S_{\bD}^{\E}(\phi_h,\phi_h)\leq c^{*} A^{\E}(\phi_h,\phi_h)\qquad
\forall \phi_h \in \textrm{Ker}(\PiK),
\end{equation}
with $c_{*}$ and $c^{*}$ positive constants independent of $\E$. 
More precisely, we choose the following computable representation satisfying property \eqref{term-stab-SK} 
(see~\cite[Lemma 5.1]{CKP-arxiv2022}):
\begin{equation*}%\label{forms:stab}
S_{\bD}^{\E}(\varphi_h,\phi_h):= h^{-2}_{\E}\sum_{i=1}^{N^{\E}_{\dof}} \dof_i(\varphi_h) \dof_i(\phi_h) 
\qquad \forall \varphi_h,\phi_h \in \VK,	
\end{equation*}
 where $N^{\E}_{\dof}$ denote the number of 
degrees freedom of $ \VK$ and $\dof_i(\cdot)$ is the operator that to each 
smooth enough function $\phi$ associates the $i$th local degree of freedom 
$\dof_i(\phi)$, with $1 \leq i \leq N^{\E}_{\dof}$.  

To approximate the local trilinear form $B^{\E}(\cdot; \cdot, \cdot)$, we consider the following expression:   
\begin{equation}\label{disc-tril:form}
	B_{h}^{\E}(\zeta_h; \varphi_h,\phi_h)
	:=\left( \PimoK \Delta \zeta_h \:\PiunoKb \curl \varphi_h ,  \PiunoKb \nabla \phi_h\right)_{0,\E} \qquad  \forall \zeta_h, \varphi_h,\phi_h\in\VK.
\end{equation}

Finally, for the functional \eqref{loadcont} we consider the following  local approximation:
\begin{equation*}
	F_h^{\E}(\phi_h)
	:= (\PiunoKb \fb, \curl \phi_h)_{0,\E} \equiv  (\fb,  \PiunoKb  \curl\phi_h)_{0,\E}
	\qquad \forall\phi_h\in\VK.
\end{equation*}

Thus, for all  $ \zeta_h, \varphi_h,\phi_h\in\Vh$, we define the global multilineal forms, as follows:
\begin{align}
A_h : \Vh \times  \Vh &\to \R, \quad A_h(\varphi_h, \phi_h) :=\sum_{\E \in \CT_h} A_h^{\E}(\varphi_h, \phi_h),
\label{disc-globA}\\
B_h : \Vh \times  \Vh \times \Vh &\to \R, \quad B_h(\zeta_h;\varphi_h, \phi_h) :=\sum_{\E \in \CTh} B_h^{\E}(\zeta_h;\varphi_h, \phi_h),\label{disc-globB} \\
F_h : \Vh   &\to \R, \quad	F_h(\phi_h):= \sum_{\E\in\CTh}F_h^{\E}(\phi_h). \label{load-global} 
\end{align}

We recall that all the forms defined above are computable using the degrees freedom and 
the trilinear form $B_h(\cdot;\cdot,\cdot)$ is extendable to the whole $\Phi$.

Now, we establish the classical consistency and stability VEM properties 
(see~\cite{BBCMMR2013,ALKM2016,CMS2016,ZZMC19}).
\begin{lemma}
The local bilinear forms $A^{\E}(\cdot,\cdot)$ and $A_h^{\E}(\cdot,\cdot)$ satisfy  the following properties:
\begin{itemize}
\item \textit{consistency}: for all $h > 0$ and for all $\E\in\CTh$, we have that
\begin{align}
A_h^{\E}(\chi,\phi_h)&=A^{\E}(\chi,\phi_h)\qquad \quad \forall \chi\in\P_2(\E),
\qquad\forall \phi_h\in \VK,\label{consis-A}
\end{align}
		
\item \textit{stability and boundedness}: there exist positive
constants $\alpha_1$ and $\alpha_2$, independent of $h$ and $\E$, such that:
\begin{equation}
\alpha_1 A^{\E}(\phi_h,\phi_h)
\leq A_h^{\E}(\phi_h,\phi_h)
\leq\alpha_2 A^{\E}(\phi_h,\phi_h)
\qquad\forall \phi_h\in \VK.\label{stab-a}  
		\end{equation}
	\end{itemize}
\end{lemma}

%%%%----------- 

\section{Discrete formulation and its well-posedness}\label{Discrete-Problem}
%\section{Theoretical results}\label{Discrete-Problem}
\setcounter{equation}{0}
In this section we write the nonconforming discrete VE formulation  
and we provide its well-posedness  by using a fixed-point strategy.

The nonconforming VE problem reads as:
find $\psi_h\in\Vh$, such that
\begin{equation}\label{NSE:stream:disc}
\nu  A_h(\psi_h,\phi_h)+B_h(\psi_h;\psi_h,\phi_h)=F_{h}(\phi_h) 
\qquad \forall \phi_h\in\Vh,
\end{equation}
where the multilineal forms $A_{h}(\cdot,\cdot)$, $B_{h}(\cdot; \cdot,\cdot)$ and
$F_{h}(\cdot)$ are defined in \eqref{disc-globA}, \eqref{disc-globB} and
 \eqref{load-global}, respectively.

In order to  prove that problem~\eqref{NSE:stream:disc} is well-posed, 
in next section, we will introduce an~\textit{enriching operator} $\widetilde{E}_h$, 
from the sum space $\Phi+\Vh$ into the conforming counterpart of the space $\Vh$.
Moreover, we establish some approximation properties  for this operator, 
and by using such estimates we provide novel embedding results for the 
sum space $\Phi +\Vh$, which will be useful to establish 
the well-posedness of discrete problem and the error estimates. 

We remark that the operator $\widetilde{E}_h$ constructed here can be seen as 
an extension of the enriching operator defined in \cite{JY2021} and 
the quasi-interpolation operator constructed in~\cite{CH2022}.

\subsection{A new enriching operator}\label{Enrich:operator}
With the aim of introducing the aforementioned  operator 
and establish its approximation properties, we start by assuming the classical assumptions
on the polygonal decomposition.  
There exists a uniform number $\rho>0$ independent of $\CT_h$,  such that for 
every $\E\in \CTh$ it holds \cite{BBCMMR2013}:
\begin{itemize}
	\item[${\bf A_1}:$]  $\E$ is star-shaped with
	respect to every point of a  ball
	of radius  $ \geq  \rho h_\E$;
	\item[${\bf A2}:$] the length $h_e$ of every  edge $e \in \partial \E$,
	satisfies $h_e\geq \rho h_\E$.
\end{itemize}

From reference \cite{CH2017} we have that if
the mesh $\CT_h$ fulfilling the assumptions ${\bf A_1}$ and ${\bf A_2}$, then the mesh also 
satisfy  the following property:

\begin{itemize}
\item[${\bf P_1}:$] For each  $\E\in \CTh$, there exists a virtual triangulation $\CT^{\E}_h$ of $\E$ such that  $\CT^{\E}_h$ is uniformly shape regular and quasi-uniform. 
The corresponding mesh size $h_T$ of $\CT^{\E}_h$ is proportional to $h_{\E}$. Every edge of $\E$ 
is a side of a certain triangle in $\CT^{\E}_h$. 
\end{itemize}

\begin{remark}\label{remark:A1}
From property ${\bf P_1}$, we have that the number of triangles of 
each virtual triangulation $\CT^{\E}_h$ is uniformly bounded by a 
number $L$ and the size of each triangle is comparable to that
of the polygon (for further details, see~\cite{CH2017}).
\end{remark}

Now, for the sake of completeness, we will recall the construction of the 
$H^2$-conforming virtual space~\cite{ABSV2016}.
\paragraph{Conforming virtual local and global space.}
For every polygon $\E\in\CT_h$, we introduce
the following preliminary finite dimensional space \cite{ABSV2016}:
\begin{align*}
\widetilde{\calW}_h^{\c}(\E)
:=\left\{\phi_h\in \HdoK : \Delta^2\phi_h\in\P_{2}(\E), \phi_h|_{\partial\E}\in C^0(\partial\E),
\phi_h|_e\in\P_{3}(e)\,\,\forall e\subseteq\partial\E,\right.\\
\left.\nabla \phi_h|_{\partial\E}\in C^0(\partial\E)^2,
\partial_{\nb_e} \phi_h|_e\in\P_{1}(e)\,\,\forall e\subseteq\partial\E\right\},
\end{align*}

Next, for a given $\phi_{h}\in	\widetilde{\calW}_h^{\c}(\E)$, we introduce two sets
${\bf \mathscr{D}^{\bv}_1}$ and ${\bf \mathscr{D}^{\nabla}_2}$ of linear operators from
the local virtual space $\widetilde{\calW}_h^{\c}(\E)$ into $\R$:

\begin{itemize}
	\item  $D^{\c}_{\bv}$:  the values of $\phi_h(\bv)$
	 for all vertex $\bv$ of the polygon $\E$;
	\item  $D^{\c}_{\nabla}$: the values
	of $h_{\bv_i}\nabla\phi_h(\bv)$ for all vertex $\bv$ of the polygon $\E$,
\end{itemize}
where $h_{\bv}$  is a characteristic length attached to each vertex $\bv$,
for instance  to the average  of the diameters of the elements with $\bv$ as a vertex.
  
Now, we consider the operator 
$\Pi_{\E}^{\bD,\c}: \widetilde{\calW}_h^{\c}(\E)\longrightarrow\P_2(\E)\subseteq 
\widetilde{\calW}_h^{\c}(\E)$ associated to the conforming approach, which is computable
 using the sets 
${\bf D}^{\c}_{\bv}$ and ${\bf D}^{\c}_{\nabla}$ (for further details see \cite[Lemma 2.1]{ABSV2016}).

Next, for each $\E \in \CT_h$, we consider the conforming local  virtual space given by:
\begin{align*}%\label{local:conforming:space}
	\calW_h^{\c}(\E) 
	:= \left \{ \phi_h \in \widetilde{\calW}_h^{\c}(\E): 
	(\phi_h-\Pi_{\E}^{\bD,\c} \phi_h,\chi)_{0,\E}=0 \quad \forall \chi \in \P_{2}(\E) \right \}.
\end{align*}

 For every decomposition $\CT_h$ of $\O$ into polygons $\E$, we define the 
 conforming virtual spaces $\calW^{\c}_h$:
 \begin{equation*}
 	\calW^{\c}_h:=\left\{\phi_h\in \Phi:\ \phi_h|_{\E}\in\calW_h^{\c}(\E) 
 	\qquad \forall \E \in\CT_h\right\}.
 \end{equation*} 
 We recall that the global DOFs are defined by  ${\bf D}^{\c}_{\bv}$  and 
 ${\bf D}^{\c}_{\nabla}$ excluding the DOFs on the boundary $\Gamma$. 
  
\paragraph{Construction of the Enriching operator.}
We will extend the ideas of~\cite{JY2021,CH2022}. 
First, we will introduce some additional notations. 
Indeed, for each vertex  $\bv \in \VV_h$ and for all
 $e \in\EE_h$ we define the following sets 
(\textit{patches}):
\begin{equation*}
\begin{split}
\omega(\bv)	:= \bigcup \big\{ \E \in \CT_h: \bv \in \E \big\}\quad  \text{and} 
\quad\omega(e) := \bigcup \big\{ \E \in \CT_h: e \in \partial\E \big \}.
\end{split}
\end{equation*}
Moreover, for each $\E \in \CT_h$ we define
\begin{equation*}
\begin{split}
\omega(\E) := \bigcup \big\{ \widehat{\E} \in \CT_h: \E \cap \widehat{\E}  \neq \emptyset  \big\},
	\end{split}
\end{equation*}
and for a function $\phi_h \in H^2(\CT_h)$, we defined the following broken seminorm $|\phi_h|^2_{2,\omega(\E),h}	 
:= \Big(\sum_{\widehat{\E} \in \omega(\E) } |\phi_h|^2_{2,\widehat{\E}} \Big)^{1/2}$.

We will denote by $N(\bv)$ and  by $N(e)$ the number of elements in $\omega(\bv)$ 
and $\omega(e)$, respectively. In addition, for any $\varphi_h \in \Phi +\Vh$, 
we introduce the piecewise $L^2$-projection $\Pi^2$, as $\Pi^2  \varphi_h|_{\E} = \PioK(\varphi_h|_{\E}),$
where $\PioK$  is the usual $L^2$-projection onto $\P_2(\E)$ defined in \eqref{proymdos}. 

Let $N^{\c}_{\dof}:=\dim(\calW^{\c}_{h})$, then as in \cite{JY2021,CH2022} we can relabel the degrees 
of freedom using a single subindex $j=1, \ldots, N^{\c}_{\dof}$ and will denote the degrees of freedom by $\{{\bf D}^{\c}_{j}\}^{N^{\c}_{\dof}}_{j=1}$, which are associated with the shape basis
functions $\{\zeta_j\}^{N^{\c}_{\dof}}_{j=1}$ of the space $\calW^{\c}_{h}$. 
Employing this notation the enriching operator $\widetilde{E}_h : \Phi + \Vh \to \calW^{\c}_{h}$ 
is defined by:
\begin{equation*}
\widetilde{E}_h\varphi_h(x)= \sum_{j=1}^{N^{\c}_{\dof}} {\bf D}^{\c}_{j}(\widetilde{E}_h\varphi_h)\zeta_{j}(x),
\end{equation*} 
where the degrees of freedom for $\widetilde{E}_h \varphi_h$ are determined by:  
\begin{enumerate}
\item 	${\bf D}^{\c}_{1,\bv}(\widetilde{E}_h\varphi_h)= \widetilde{E}_h\varphi_h(\bv) := \varphi_h(\bv) \qquad \forall \bv \in \VVi$;
\item  ${\bf D}^{\c}_{2,\bv}(\widetilde{E}_h\varphi_h):= \frac{1}{N(\bv)} \sum_{\widehat{\E}\in  \omega(\bv )} h_{\bv} \nabla( \Pi^2\varphi_h|_{\widehat{\E}}(\bv))  \qquad \forall \bv \in \VVi$. 
\end{enumerate}

The following result establishes approximation properties of the enriching  operator $\widetilde{E}_h$. \begin{proposition}\label{proposition-Eh}
For all $\phi_h \in \Phi+\Vh$, there exists $C>0$ independent of $h$, such that
\begin{equation*}
\sum_{j=0}^{2}h^{2j}_{\E}|\phi_h-\widetilde{E}_h\phi_h|^2_{j,\E} \leq Ch_{\E}^4|\phi_h|^2_{2,\omega(\E),h} 
\qquad \forall \E \in \CT_h.
	\end{equation*}	
\end{proposition}

\begin{proof}
First, we note that using the same arguments used in~\cite[Lemma 4.2]{JY2021} 
 and \cite[Lemma 4.1]{CH2022}  (see also \cite{AMS2022-CI2MA}), for all 
 $\phi_h \in \Phi+\Vh$, we have that 
\begin{align}\label{Eh:norm1:local}
\|\phi_h-\widetilde{E}_h\phi_h\|_{0,\E}\leq Ch^2_{\E}|\phi_h|_{2,\omega(\E),h}  \quad \text{and} \quad 
|\phi_h-\widetilde{E}_h\phi_h|_{2,\E}\leq C|\phi_h|_{2,\omega(\E),h}. 		
\end{align}
Now, by using standard inequality (see \cite[equation (3.3)]{JY2021}) and \eqref{Eh:norm1:local}, 
there exists a constant $C>0$, independent to $h_{\E}$, such that
\begin{equation}\label{Eh:pre1}
\begin{split}
|\phi_h-\widetilde{E}_h\phi_h|_{1,\E} &\leq  C(h_{\E}|\phi_h-\widetilde{E}_h\phi_h|_{2,\E} +
  h^{-1}_{\E}\|\phi_h-\widetilde{E}_h\phi_h\|_{0,\E})\\
& \leq C (h_{\E}|\phi_h|_{2,\omega(\E),h} + h^2_{\E} h^{-1}_{\E}|\phi_h|_{2,\omega(\E),h})\\
& \leq Ch_{\E}|\phi_h|_{2,\omega(\E),h}.
		\end{split}
	\end{equation}	
The desired result follows from \eqref{Eh:norm1:local} and \eqref{Eh:pre1}. 
\end{proof}

\subsection{Discrete Sobolev embeddings and properties of the discrete forms}
In this subsection we establish two important estimates, which are useful to prove the continuity 
of the discrete multilineal forms.   We start presenting  the main result of this section, 
which establishes discrete Sobolev
embeddings for the space $\Phi +\Vh$.
\begin{theorem}\label{solobev:inclusion}
For any $2 \leq q <  \infty$ there exists a positive constant $C$, independent of  $h$, such that
\begin{equation*}
|\phi_h|_{1,q,h} \leq \CS|\phi_h|_{2,h} \quad  \quad \forall\phi_h \in \Phi +\Vh.
\end{equation*}
\end{theorem}

\begin{proof}	
Let $2 \leq q <  \infty$, $\phi_h \in \Phi+ \Vh$ and $\widetilde{E}_h : \Phi+\Vh \to \calW^{\c}_{h}$ be the enriching operator  
defined in the above subsection. Then, by using 
the triangle inequality, the  embedding  of $H^2(\O)$ into $W^{1,q}(\O)$ and stability
 property in Proposition~\ref{proposition-Eh}, we have  that
\begin{equation}\label{ineq:1}
\begin{split}
|\phi_h|_{1,q,h} &\leq |\phi_h- \widetilde{E}_h \phi_h|_{1,q,h} + |\widetilde{E}_h \phi_h|_{1,q,\O} \\
&\leq |\phi_h- \widetilde{E}_h \phi_h|_{1,q,h} + C|\widetilde{E}_h \phi_h|_{2,\O} \\
&\leq |\phi_h- \widetilde{E}_h \phi_h|_{1,q,h} + C|\phi_h|_{2,h}. 
\end{split}
\end{equation}
In what follows we will estimate the term $|\phi_h- \widetilde{E}_h \phi_h|_{1,q,h}$ in the right-hand 
side of \eqref{ineq:1}. To do that,  for each $\E \in \CT_h$,  we consider the 
sub-triangulation $\CT^{\E}_h$ of property ${\bf P_1}$. 
Next, let  $\varphi := \nabla(\phi_h- \widetilde{E}_h \phi_h)|_{\E}$ and  $\widehat{\varphi}$ 
be the image of $\varphi $ under the affine transformation from $T$ to the
reference triangle $\widehat{T}$. Then, by using scaling arguments and the embedding of
$H^1(\widehat{T})$ into $L^q(\widehat{T})$, there is $C>0$  independent of $\E$, such that %. we obtain
\begin{equation*}%\label{ineq:2}
\begin{split}
|\phi_h- \widetilde{E}_h \phi_h|_{1,q,T} &= \|\varphi\|_{L^q(T)}  
\leq  C  |T|^{1/q}\|\widehat{\varphi}\|_{L^q(\widehat{T})} 
\leq C|T|^{1/q}  \|\widehat{\varphi}\|_{1,\widehat{T}}\\
& \leq C|T|^{(2-q)/2q}(\|\varphi\|^2_{0,T} + h^2_T|\varphi|^2_{1,T})^{1/2}\\
& \leq C(h^2_{T})^{(2-q)/2q}  (|\phi_h- \widetilde{E}_h \phi_h|^2_{1,T} + h^2_T|\phi_h- \widetilde{E}_h \phi_h|^2_{2,T})^{1/2}\\
& \leq C h_{\E}^{(2-q)/q} (|\phi_h- \widetilde{E}_h \phi_h|^2_{1,\E} + h^2_{\E}|\phi_h- \widetilde{E}_h \phi_h|^2_{2,\E})^{1/2},
	\end{split}
\end{equation*}
where we have used the relation $|T| \approx h^2_{T} $ and that the size of each triangle in $\CT^{\E}_h$ 
is comparable with the polygon mesh size $h_{\E}$ (see Remark~\ref{remark:A1}). 

Now, from the above estimate and  Proposition~\ref{proposition-Eh} it holds
\begin{equation}\label{ineq:2}
 \begin{split}
|\phi_h- \widetilde{E}_h \phi_h|_{1,q,T} 
& \leq C h_{\E}^{(2-q)/q}h_{\E}|\phi_h|_{2,\omega(\E),h} 
 \leq  Ch^{2/q}_{\E}|\phi_h|_{2,\omega(\E),h}.
  	\end{split}
  \end{equation}

From bound \eqref{ineq:2} and since the number of triangles 
of each virtual triangulation $\CT^{\E}_h$ is uniformly bounded by a number $L$ (see again Remark~\ref{remark:A1}), 
we obtain
\begin{equation*}%\label{ineq:iv}
\begin{split}
|\phi_h- \widetilde{E}_h \phi_h|^q_{1,q,\E} =  \sum_{T \in \E} |\phi_h- \widetilde{E}_h \phi_h|^q_{1,q,T} 
\leq C\sum_{T \in \E} h^{2}_{\E}|\phi_h|^q_{2,\omega(\E),h} \leq C L h^{2}_{\E}|\phi_h|^q_{2,\omega(\E),h}.
	\end{split}
\end{equation*}

Summing over each $\E \in \CT_h$, using the fact that $q \geq 2$ and a $\ell^{q}$-norms inequality,
 along with  $0< h \leq C < 1$, we obtain  
\begin{equation}\label{ineq:iv}
\begin{split}
|\phi_h- \widetilde{E}_h \phi_h|_{1,q,h} &= \Big(  \sum_{\E \in \CT_h}|\phi_h- \widetilde{E}_h \phi_h|^q_{1,q,\E} \Big)^{1/q} 
\leq C h^{2/q} \Big(\sum_{\E \in \CT_h}|\phi_h|^q_{2,\omega(\E),h} \Big)^{1/q}\\
&\leq C h^{2/q} \Big(\sum_{\E \in \CT_h}|\phi_h|^2_{2,\omega(\E),h} \Big)^{1/2}
%&\leq C h^{2}|\phi_h|^2_{2,h}\sum_{\E \in \CT_h}|\phi_h|^2_{2,\E} 
\leq  C h^{2/q}|\phi_h|_{2,h}  \leq C|\phi_h|_{2,h},
\end{split}
\end{equation}
where the constant $C>0$ is independent of $h$.

Finally, combining  the estimates \eqref{ineq:1} and \eqref{ineq:iv} we conclude the proof.
\end{proof}

The next result  has been established in \cite[Lemma 5.1]{ZZCM2018}
 and allows show that the application  $| \cdot|_{2,h}$ is a norm in $\Vh$.
\begin{lemma}
	For all $\phi_h \in \Vh$, is holds: 
	\begin{equation*}
		\|\phi_h\|_{0,\O}+ |\phi_h|_{1,h} \leq  C|\phi_h|_{2,h},
	\end{equation*}
	where $C>0$ is a constant independent of $h$.
\end{lemma}

The following lemma summarize other properties of the discrete forms defined in 
\eqref{disc-globA}-\eqref{load-global}, which  will be used to establish the
well-posedness of  the discrete problem. 

\begin{lemma}\label{lemma:forms}
There exist positive constants $\CAh,\wtalpha,\CBh,\CFh$, independent of $h$, 
such that for all $\zeta_h, \varphi_h, \phi_h \in \Vh$  the forms defined in \eqref{disc-globA}-\eqref{load-global} satisfies the following properties: 
\begin{align}
|A_h(\varphi_h,\phi_h)|  &\leq  \CAh|\varphi_h|_{2,h} |\phi_h|_{2,h}\qquad \text{and} \qquad	A^h(\phi_h,\phi_h)\ge \wtalpha|\phi_h|_{2,h}^2, \label{Ah-ellipt}\\
B_h(\zeta_h;\varphi_h,\phi_h)  &\leq \CBh |\zeta_h|_{2,h}|\varphi_h|_{2,h}|\phi_h|_{2,h},\label{cont-Bh}\\
B_{h}(\zeta_h; \phi_h, \phi_h)&=0,\qquad \text{and} \qquad	
B_{h}(\zeta_h; \varphi_h, \phi_h)=-B_{h}(\zeta_h; \phi_h, \varphi_h), \label{Bh-skewsym}\\
|F_{h}(\phi_h)| &\leq  \CFh\|\fb\|_{0,\O}|\phi_h|_{2,h}.\label{cont-Fh} 
\end{align}
\end{lemma}
\begin{proof}
Properties in \eqref{Ah-ellipt} are obtained  from the definition of bilinear form $A_h(\cdot,\cdot)$  and the stability~\eqref{stab-a}.
To prove property \eqref{cont-Bh}, we use the definition of trilinear form $B_h(\cdot;\cdot,\cdot)$ and H\"{o}lder inequality to obtain that
\begin{equation*}
\begin{split}
B_h(\zeta_h;\varphi_h,\phi_h) & \leq \CBD^2\Big(\sum_{\E \in \CT_h} \|\Delta \zeta_h\|^2_{0,\E} \Big)^{1/2} \Big(\sum_{\E \in \CT_h}\|\curl \varphi_h\|^4_{L^{4}(\E)} \Big)^{1/4} \Big(\sum_{\E \in \CT_h}\| \nabla \phi_h\|^4_{L^{4}(\E)} \Big)^{1/4} \\
& \leq  \CBD^2|\zeta_h|_{2,h} |\varphi_h|_{1,4,h} |\phi_h|_{1,4,h}\\
& \leq  \CBh |\zeta_h|_{2,h} |\varphi_h|_{2,h} |\phi_h|_{2,h},
	\end{split}
\end{equation*} 
where $\CBh:= (\CBD\CS)^2 >0$, and  $\CBD$, $\CS$ are the constants in
\eqref{bound:proj} and Theorem~\ref{solobev:inclusion}, respectively.

Finally, the proof of properties~\eqref{Bh-skewsym} and \eqref{cont-Fh} are obtained from 
the definition of forms $B_h(\cdot;\cdot,\cdot)$ and $F_h(\cdot)$.

\end{proof}

\subsection{A fixed-point strategy}\label{fixed:point:strategy}
In this subsection we will develop a fixed-point strategy to establish the well-posedness of discrete problem
\eqref{NSE:stream:disc}.  Indeed, for a given 
$\xi_h \in \Vh$, we define the operator 
\begin{equation*}
	\begin{split}
		T^h : \Vh &\longrightarrow \Vh \\
		\xi_h & \longmapsto T^h(\xi_h) = \varphi_h,
	\end{split}
\end{equation*}
where $\varphi_h$ is the solution of the following linear problem: 
find $\varphi_h \in \Vh$, 
such that 
\begin{equation*}
\nu A_h(\varphi_h,\phi_h)+B_h(\xi_h;\varphi_h,\phi_h)=F_{h}(\phi_h) 
\qquad \forall \phi_h\in\Vh.
\end{equation*}

Next, we consider the ball
$Y_h:= \left\{ \phi_h \in \Vh: \|\phi_h\|_{2,\O} \leq \CFh(\wtalpha \nu)^{-1}\|\fb\|_{0,\O}\right\}$. 
Then, we  have the following result for the operator $T^h$.
\begin{lemma}\label{Th:welldefined}
 The operator $T^h$ is well defined. Moreover, if  
\begin{equation}\label{lambdah}
	\lambda_h := \CBh \CFh(\alpha_1\nu)^{-2} \|\fb\|_{0,\O} <1.
\end{equation}	 
 Then, $T^{h}: Y_h  \to Y_h $  is a contraction mapping. 
\end{lemma}
\begin{proof}
The	demonstration follows from the definition of operator $T^h$, 
and Lemma~\ref{lemma:forms} the Lax-Milgram Theorem.

\end{proof}

%The following result establish the well-posedness of discrete problem.
We finish this section with the following result, which establishes that the discrete
problem is well-posed.
\begin{theorem}\label{theodiscr}
If condition \eqref{lambdah} is satisfied, then there exists a unique 
$\psi_h \in \Vh$ solution to problem \eqref{NSE:stream:disc} satisfying the 
following dependence of the data
\begin{equation}\label{depend:disc}
|\psi_h|_{2,h} \leq  \CFh (\wtalpha \nu)^{-1}\|\fb\|_{0,\O}.
\end{equation}  
\end{theorem}   
\begin{proof}
The proof follows from Lemma \ref{Th:welldefined} and the Banach point-fixed Theorem.
  
\end{proof}
%%%%%%%-----------

\section{Error analysis}\label{Error:Analisis}
\setcounter{equation}{0}
In this section we will develop an error analysis for the VEM proposed in~\eqref{NSE:stream:disc}.
By exploiting the naturally skew-symmetry  property of the discrete trilinear form, and the consistence and boundedness properties of discrete bilinear form we write  an abstract convergence result for the nonlinear VE scheme. 
Then, by using the enriching operator, we establish key approximation properties 
involving the bilinear and trilinear forms, together  with the consistency 
errors, which allow the derivation of an optimal error estimate in broken  $H^2$-norm under the 
minimal regularity condition on the weak  solution  (cf. Theorem~\ref{regu-aditional}).
Moreover, by using duality arguments and the enriching operator we also establish optimal error estimates 
in the $H^1$- and $L^2$-norm under the same regularity condition on the stream-function $\psi$ and  
the density force $\fb$.

\subsection{An abstract convergence result}
We start with two technical lemmas involving the continuous
and discrete forms $B(\cdot,\cdot,\cdot)$ and $B_h(\cdot,\cdot,\cdot)$ 
defined in \eqref{B-cont}  and \eqref{disc-globB}, respectively. 
\begin{lemma}\label{cota2-B}
	Let $B(\cdot;\cdot,\cdot)$ be the trilinear form defined in \eqref{B-cont}. Then, 
	for all $\zeta \in H^{2+t}(\O)$, with $t\in (1/2,1]$, and for all 
	$\varphi \in H^2(\O)$ and  $\phi_h \in H^{1}(\CT_h)$, it holds:
	\begin{equation*}
		B(\zeta;\varphi,\phi_h) \leq  C \|\zeta\|_{2+t,\O}\|\varphi\|_{2, \O} |\phi_h|_{1,h}.
	\end{equation*}	
\end{lemma}

\begin{proof}
By using  the H\"{o}lder inequality, for each  $\zeta \in H^{2+t}(\O)$, with $t\in (1/2,1]$, 
for all $\varphi \in H^2(\O)$ and  for all $\phi_h \in H^{1}(\CT_h)$, we have
\begin{equation*}
\begin{split}
B(\zeta;\varphi,\phi_h)
&\leq \Big( \sum_{\E \in  \CT_h} \|\Delta \zeta\|^4_{L^4(\E)}\Big)^{1/4} 
\Big( \sum_{\E \in  \CT_h} \|\nabla \varphi\|^4_{L^4(\E)} \Big)^{1/4} 
\Big( \sum_{\E \in  \CT_h} \|\nabla \phi_h\|^2_{0,\E} \Big)^{1/2}\\
&\leq |\zeta|_{2,4,\O}|\varphi|_{1,4,\O} |\phi_h|_{1,h}.
\end{split} 
	\end{equation*}	
Then, by using the  Sobolev embeddings $ H^{2}(\O) \hookrightarrow  W^{1,4}(\O)$ and $ H^{2+t}(\O) \hookrightarrow  W^{2,4}(\O)$,  with $t\in (1/2,1]$, we obtain
	\begin{equation*}%\label{cota2-B}
		B(\zeta;\varphi, \phi_h) \leq  C \|\zeta\|_{2+t,\O}\|\varphi\|_{2, \O} |\phi_h|_{1,h},
	\end{equation*}	
	where $C$ depends only on $\O$. The proof is complete.
\end{proof}

\begin{remark}\label{remark:boundB2}
Following the above arguments, we can also prove that for all $\zeta \in H^{2+t}(\O)$, with $t\in (1/2,1]$, and for all 
$\varphi_h \in H^{1}(\CT_h)$ and  $\phi \in H^2(\O)$, it holds
\begin{equation*}
B(\zeta; \varphi_h,\phi) \leq  C \|\zeta\|_{2+t,\O}|\varphi_h|_{1,h} |\phi|_{2,\O}.
\end{equation*}	
\end{remark}
\begin{lemma}\label{lemma-tec2}
Let $\varphi \in \Phi$ and $\varphi_h \in \Vh$. 
Then, for each $\phi_h \in \Vh$,  it holds 
\begin{equation*}
|B_h(\varphi;\varphi,\phi)-B_h(\varphi_h;\varphi_h,\phi_h)| \leq 
\CBh \left( |\varphi_h|_{2,h} |\phi_h|_{2,h}
+|\varphi -\varphi_h+ \phi_h|_{2,h}(\|\varphi\|_{2, \O}+ |\varphi_h|_{2,h})\right)|\phi_h|_{2,h}. 		
	\end{equation*}
\end{lemma}	

\begin{proof}
	The proof follows by adding and subtracting adequate terms  together with  
	property \eqref{Bh-skewsym} and  Theorem~\ref{solobev:inclusion}.
	
\end{proof}

In order to derive the abstract error estimate for the nonlinear VE scheme,  
we will introduce the following consistence errors.  Let $\psi \in \Phi$ be the solution 
of continuous problem \eqref{weak:stream:NS}, then we define:
\begin{align}
\calN_h(\psi;\phi_h)&:=\nu A(\psi,\phi_h)+B(\psi;\psi,\phi_h)-F(\phi_h) 
\:\qquad \forall \phi_{h} \in \Vh, \label{igual:nonconf}\\
\calC_h(\psi;\phi_h) &:= B(\psi;\psi,\phi_h)-B_h(\psi;\psi,\phi_h)\qquad\qquad \quad \forall \phi_{h} \in \Vh. \label{constency:B} 
\end{align}

The first  term above measures  to what extent the continuous  solution $\psi$ 
does not satisfy  the nonconforming virtual element formulation~\eqref{NSE:stream:disc} 
and the second term measure of the variational crime perpetrated in the discretization 
of the trilinear form $B(\cdot;\cdot,\cdot)$. In addition, we define the following quantity:
\begin{equation}\label{constency:F}
\|F-F_h\|:= \sup_{\substack{\phi_h \in \Vh \\ \phi_h\neq\, 0}}
\frac{|F(\phi_h)-F_h(\phi_h)|}{|\phi_h|_{2,h}}. 
\end{equation}

In subsection~\ref{approx:results} we will establish approximation properties for the above terms.
Next, we provide the following Strang-type result for our nonlinear VE scheme.
\begin{theorem}[Abstract convergence result]\label{pre:converg}
Let $\psi$ and $\psi_{h}$ be the unique
	solutions to  problems~\eqref{weak:stream:NS} and \eqref{NSE:stream:disc}, respectively.
	There exists a positive constant $C$, independent of $h$, such that
	\begin{equation*} %\small %\label{strang:type}
		|\psi-\psi_h|_{2,h} \leq C\Bigg(\inf_{\phi_h \in \Vh} |\psi-\phi_h|_{2,h} +\inf_{\chi \in \P_2(\CT_h)}|\psi-\chi|_{2,h}
		+ \|F-F_h\|+ \sup_{\substack{\phi_h \in \Vh \\ \phi_h\neq\, 0}} 
		\left( \frac{|\calN_h(\psi;\phi_h)|}{|\phi_h|_{2,h}} +\frac{|\calC_h(\psi;\phi_h)|}{|\phi_h|_{2,h}} \right) \Bigg), 	
	\end{equation*}
	where $\calN_h(\psi;\cdot)$ and $\calC_h(\psi;\cdot) $
	are the consistency errors defined in \eqref{igual:nonconf} and  \eqref{constency:B}.
\end{theorem}

\begin{proof}
Let $\phi_h\in\Vh$ and set $\delta_h:=\phi_h-\psi_h$. Then,  by using triangle inequality we obtain
\begin{equation}\label{trian0}
\vert\psi-\psi_h\vert_{2,h}\leq \vert\psi-\phi_h\vert_{2,h}+\vert \delta_h\vert_{2,h}.
\end{equation}
Now, by using the property \eqref{Ah-ellipt}, the consistence of bilinear forms $A_h^{\E}(\cdot, \cdot)$  
	(cf. \eqref{consis-A}), we have	
	\begin{equation}\label{trian1}
		\begin{split}
			\nu \widetilde{\alpha}|\delta_h|^2_{2,h} & \leq \nu A_{h}(\delta_h,\delta_h) %= \nu A^{h}(\delta_h,\delta_h)
			%= \nu A^{h}(\phi_h-\psi_h,\delta_h)\\
			=\nu A_{h}(\phi_h,\delta_h)-\nu A_{h}(\psi_h,\delta_h)\\
			&=\nu A_{h}(\phi_h,\delta_h)-F_{h}(\delta_h)+B_h(\psi_h;\psi_h,\delta_h)\\
			&=\nu \sum_{\E \in  \CTh} \left( A_h^{\E}(\phi_h-\chi,\delta_h)+A^{\E}(\chi-\psi,\delta_h) \right)
			+\nu \sum_{\E \in  \CTh}A^{\E}(\psi,\delta_h)-F_{h}(\delta_h)+B_h(\psi_h;\psi_h,\delta_h)\\
			&=\nu \sum_{\E \in  \CTh} \left( A_h^{\E}(\phi_h-\chi,\delta_h)+A^{\E}(\chi-\psi,\delta_h) \right) +(\nu A(\psi,\delta_h)-F_{h}(\delta_h)+B_h(\psi_h;\psi_h,\delta_h))\\
			&=\nu \sum_{\E \in  \CTh} \left( A_h^{\E}(\phi_h-\chi,\delta_h)+A^{\E}(\chi-\psi,\delta_h) \right)
			+ \calN_h(\psi;\psi,\delta_h)\\
			& \qquad +[F(\delta_h)-F_{h}(\delta_h)]+[B_h(\psi_h;\psi_h,\delta_h)-B(\psi;\psi,\delta_h)],
		\end{split}
	\end{equation} 
where we have added and subtracted adequate terms and $\chi$ is an arbitrary element of $\P_2(\CT_h)$.
	
From the continuity of bilinear forms $A^{\E}(\cdot, \cdot)$, $A_h^{\E}(\cdot, \cdot)$, 
and  by using the triangular inequality, we have  %Propositions~\ref{approx:poly} and~\ref{approx:virtual}, we have
	\begin{equation*}
		\begin{split}	
			\sum_{\E \in  \CTh} \left( A_h^{\E}(\phi_h-\chi,\delta_h)+A^{\E}(\chi-\psi,\delta_h) \right) 
			& \leq  C (|\phi_h-\psi|_{2,h}+|\psi-\chi|_{2,h})|\delta_h|_{2,h}.
		\end{split} 
	\end{equation*}

Now, we add and subtract the term $B_{h}(\psi;\psi,\delta_h)$, then applying 
Lemma~\ref{lemma-tec2}, we obtain
	\begin{equation}\label{cota:TB}
		\begin{split}
			|B_h(\psi_h;\psi_h,\delta_h)-B(\psi;\psi,\delta_h)|
			& \leq |B_{h}(\psi_h;\psi_h,\delta_h)-B_h(\psi;\psi,\delta_h)|+|B_{h}(\psi;\psi,\delta_h)-B(\psi;\psi,\delta_h)|\\
			&\le \, \CBh\left(|\psi_h|_{2,h}|\delta_h|_{2,h}+ |\psi-\phi_h|_{2,h}\|\psi\|_{2+s,\O}( \|\psi\|_{2,\O}
			+  |\psi_h|_{2,h}) \right)|\delta_h|_{2,h} \\
			& \quad +|\calC_h(\psi;\delta_h)|.
			\end{split}
	\end{equation}
	Therefore, combining  \eqref{trian1}-\eqref{cota:TB}, we get  
	\begin{equation*}
		\begin{split}
			\nu \wtalpha|\delta_h|_{2,h} &\leq  C (|\psi-\phi_h|_{2,h}+|\psi-\chi|_{2,h}) +
			\CBh|\psi_h|_{2,h} +\|F-F_h\|+ |\calN_h(\psi;\delta_h)|+|\calC_h(\psi;\delta_h)|	.
		\end{split}
	\end{equation*}
	
	From the inequality above, we obtain 
	\begin{equation*}
		\begin{split}
			\nu \wtalpha (1-\CBh(\nu \wtalpha)^{-1}|\psi_h|_{2,h}) |\delta_h|_{2,h}
			%\nu \wtalpha (1-\lambda_h) |\delta_h|_{2,h}
			\leq   C \big(|\psi-\phi_h|_{2,h}+|\psi-\chi|_{2,h}+ \|F-F_h\|+ |\calN_h(\psi;\delta_h)|+|\calC_h(\psi;\delta_h)| \big).
		\end{split}
	\end{equation*} 
	
	By using \eqref{depend:disc} and condition \eqref{lambdah} we have that
	$(1-\CBh(\nu \wtalpha)^{-1}|\psi_h|_{2,h})\geq 1- \lambda_h > 0$. 
	Therefore, from above inequality, we have 
	\begin{equation*}
		\begin{split}
			|\delta_h|_{2,h}
			\leq   C \big(|\psi-\phi_h|_{2,h}+|\psi-\chi|_{2,h}+ \|F-F_h\|+ |\calN_h(\psi;\delta_h)|+|\calC_h(\psi;\psi,\delta_h)| \big).
		\end{split}
	\end{equation*}
Finally, the desired result follows from \eqref{trian0} and  the above estimate.
\end{proof}

The next step is to provide approximation properties that can be used in 
Theorem~\ref{pre:converg}. In next subsection we will establish such properties.

\subsection{Approximation results and a priori error estimate} \label{approx:results}
We have the following approximation result
for polynomials on star-shaped domains. 
\begin{proposition}\label{approx:poly}
For every $\phi\in H^{2+t}(\E)$, with $t \in [0,1]$,
there exist $\phi_{\pi}\in\P_{2}(\E)$ and $C>0$, independent of $h$, such that	
\begin{equation*}
\|\phi-\phi_{\pi}\|_{\ell,\E} 
\leq Ch_{\E}^{2+t-\ell}|\phi|_{2+t,\E}, \qquad \ell=0,1,2. 	
\end{equation*}
\end{proposition}

For the virtual space $\Vh$ we  have the following approximation result 
(see \cite{AMV2018,ZZCM2018,LZHC2021,CKP-arxiv2022}).
\begin{proposition}\label{approx:virtual}
For each  $\phi\in H^{2+t}(\O)$, with $t \in [0,1]$, there exist $\phi_{I}\in\Vh$ and $C>0$, independent of $h$, such that
\begin{equation*}
\|\phi-\phi_{I}\|_{\ell,\E} \leq Ch_{\E}^{2+t-\ell}|\phi|_{2+t,\E}, \qquad \ell=0,1,2. 
\end{equation*}
\end{proposition}

Let $E_h: \Vh \to \calW^{\c}_h$ be the restriction of the operator
$\widetilde{E}_h$ to the space $\Vh$, i.e., $E_h :=\widetilde{E}_h|_{\Vh}$. 
We note that this operator satisfies the approximation properties in 
Proposition~\ref{proposition-Eh}. Then, by using the operator $E_h$, we will establish 
an error estimate involving the bilinear form 
$A(\cdot,\cdot)$, which will be useful to obtain an error estimate in broken
$H^2$-norm under minimal regularity condition on the exact stream-function $\psi$ 
(cf. Theorem~\ref{regu-aditional}).   
\begin{lemma}\label{Lemma:A}
Let $\varphi \in H^{2+t}(\O)$, with $t \in [0,1]$. Then, for all $\phi_h \in \Vh$ 
there exists a positive constant $C$, independent of $h$, such that 
\begin{equation*}%\label{eq1:LemmaA}
A(\varphi,\phi_h-E_h \phi_h) \leq Ch^{t}\|\varphi\|_{2+t,\O} |\phi_h|_{2,h}.
\end{equation*}
\end{lemma}
\begin{proof}
The proof has been established in \cite[Lemma 4.10]{AMS2022-CI2MA}.
\end{proof}

The following  result establishes error estimates for the consistence errors  
$\calN_h(\psi;\cdot)$ and $\calC_h(\psi;\cdot)$  defined in \eqref{igual:nonconf}
 and \eqref{constency:B}, respectively. 
\begin{lemma}\label{consist:term}
Let $\psi \in H^{2+\gamma}(\O)\cap \Phi$  be the solution of problem \eqref{weak:stream:NS}.
Then, for all $\phi_h \in \Vh$, there exists a constant $C>0$, independent to $h$, such that  
\begin{equation*}
\begin{split}
|\calN_h(\psi;\phi_h)| &\leq Ch^{\gamma}(\|\psi\|_{2+\gamma,\O}+\|\fb\|_{0,\O})|\phi_h|_{2,h},\\
|\calC_h(\psi;\phi_h)| &\leq Ch^{\gamma}(\|\psi\|_{1+\gamma,\O} + \|\psi\|_{2,\O}) \|\psi\|_{2+\gamma,\O}|\phi_h|_{2,h}.
		\end{split}
\end{equation*}
	\end{lemma}
\begin{proof}
Let $\phi_h \in \Vh$. Then, we can take $E_h \phi_h \in \calW^{\c}_{h} \subset \Phi$ as test function in \eqref{weak:stream:NS} to obtain
\begin{equation}\label{eq:operator}
\nu A(\psi,E_h\phi_h)+B(\psi;\psi,E_h\phi_h) = F(E_h\phi_h).
\end{equation}
Thus, from \eqref{igual:nonconf} and \eqref{eq:operator}, we get 
	\begin{equation}\label{pre-estimation}
		\begin{split}
			\calN_h(\psi;\phi_h)
			&=	\nu A(\psi,\phi_h)+B(\psi;\psi,\phi_h)- F(\phi_h-E_h\phi_h)- F(E_h\phi_h)\\
			&=	\nu A(\psi,\phi_h-E_h \phi_h)+B(\psi;\psi,\phi_h-E_h \phi_h)-F(\phi_h-E_h\phi_h).
		\end{split}
	\end{equation}
	
	By using, identity~\eqref{pre-estimation}, the Cauchy-Schwarz inequality, 
	Lemmas~\ref{cota2-B} and~\ref{Lemma:A}, we get
	\begin{equation*}
		\begin{split}
		|\calN_h(\psi;\phi_h)|
			&\leq C\nu h^{\gamma}\|\psi\|_{2+\gamma,\O}|\phi_h|_{2,h} + C|\psi|_{2+\gamma,\O}|\psi|_{2,\O} |\phi_h-E_h \phi_h|_{1,h} +C_F\|\fb\|_{0,\O}|\phi_h-E_h \phi_h|_{1,h}\\
			& \leq  Ch^{\gamma}(\|\psi\|_{2+\gamma,\O}+\|\fb\|_{0,\O})|\phi_h|_{2,h},
		\end{split}
	\end{equation*}
where $C>0$ is independent of $h$. 

The proof of second property follows by adapting the arguments used  
in \cite[Lemma 4.2]{MS2021-camwa} to the nonconforming case and using 
Theorem~\ref{solobev:inclusion}.
\end{proof}

For the  consistence error in the approximation defined in \eqref{constency:F},
we have the following result.
\begin{lemma}\label{func-bound}
	Let  $\fb \in L^2(\O)^2$,  $F(\cdot)$ and $F_h(\cdot)$ be the functionals defined
	in \eqref{loadcont} and \eqref{load-global}, respectively.
	Then, we have the following estimate:
	\begin{equation*}
		\|F-F_h\| \leq Ch\|\fb\|_{0,\O}.  
	\end{equation*}
\end{lemma}	
\begin{proof}
	The proof follows from the definition of the
	functionals $F(\cdot)$ and $F_h(\cdot)$, together
	with approximation properties of the projector $\PiunoKb$.
\end{proof}

The following result provides the rate of convergence of our virtual element scheme in broken $H^2$-norm.
 \begin{theorem}\label{Theo-conv}
 	Let $\psi \in \Phi \cap H^{2+\gamma}(\O)$ and $\psi_{h} \in \Vh$ be the unique
 	solutions of problem~\eqref{weak:stream:NS} and problem~\eqref{NSE:stream:disc}, respectively.
 	Then, there exists a positive constant $C$, independent of $h$, such that
 	\begin{equation*}
 		|\psi- \psi_{h}|_{2,h} \leq Ch^{\gamma}(\|\psi\|_{2+\gamma,\O}+\|\fb\|_{0,\O}).
 	\end{equation*}	
 \end{theorem}
 \begin{proof}
 The demonstration follows from Theorem~\ref{pre:converg}, Propositions~\ref{approx:poly} and 
 \ref{approx:virtual}, together with Lemmas~\ref{consist:term}  and \ref{func-bound}.
 \end{proof}

\subsection{Error estimates in $H^1$ and $L^2$}
In this section we provide new optimal error estimates in  broken $H^1$- and $L^2$-norms for the
stream-function  by using duality arguments and employing the enriching operator $E_h$,
under same regularity of the weak solution $\psi$ and of the density force $\fb$, 
considered in Theorem~\ref{Theo-conv}.

We start establishing the following key preliminary result involving the 
forms $B(\cdot;\cdot,\cdot)$  and $B_h(\cdot;\cdot,\cdot)$, which will useful 
to provide the error estimates in the weak norms.
This term will take care of the consistency error associate to the trilinear 
form present in the VEM approach and as we will observe, its manipulation is not direct, so it 
will require special attention due to the nonlinearity involved.
\begin{lemma}\label{LemmaTB3}
	Let $\psi \in \Phi \cap H^{2+\gamma}(\O)$ and $\psi_h \in \Vh$
	be the unique solutions of problems \eqref{weak:stream:NS} and 
	\eqref{NSE:stream:disc}, respectively. Assuming that $\fb \in L^2(\O)^2$ and
	let $\varphi \in H^{2+t}(\O)$, with $t \in  (1/2,1]$. Then, it holds
	\begin{equation*}
		\begin{split}
			T_B(\varphi):=B_h(\psi_h;\psi_h,\varphi)-B(\psi_h;\psi_{h},\varphi) 
			&\leq C\left(h^{\gamma+t}+h^{2\gamma}\right)
			\left(\|\fb\|_{0,\O}+\|\psi\|_{2+\gamma,\Omega} \right) \|\varphi\|_{2+t,\O}\\
			& \quad + 2\CR \widetilde{C}^2_{{\rm sob}}\CBD^2\|\fb\|_{0,\O}|\psi-\psi_h|_{1,h}\|\varphi\|_{1+t,\O},
		\end{split}
	\end{equation*}
	where $C>0$ is a constant independent of $h$, and $\widetilde{C}_{{\rm sob}}$, $\CR$ and $\CBD$ are the constants 
	in ~\eqref{sob:cont}, Theorem~\ref{regu-aditional} and~\eqref{bound:proj}, respectively.
\end{lemma}
{\begin{proof}
		By using the definition of trilinear forms $B(\cdot;\cdot,\cdot)$  and $B_h(\cdot;\cdot,\cdot)$,
		adding and subtracting suitable terms and using the orthogonality property of the $L^2$-projections,
		 we have the following identity
		\begin{equation*}%\label{TB2}
			\begin{split}
				T_B(\varphi)&= \sum_{\E \in  \CT_h}  ((\Delta \psi_h-\PimoK \Delta \psi_h)(\curl \psi_h-\curl \psi), \nabla \varphi)_{0,\E}
				+ (\PimoK(\Delta \psi_h-\Delta \psi)(\curl \psi_h-\PiunoKb \curl\psi_h), \nabla \varphi)_{0,\E} \\
				& \qquad +  (\PimoK(\Delta(\psi_h-\psi))\PiunoKb\curl \psi_h, 
				\nabla \varphi- \PiunoKb \nabla \varphi)_{0,\E} 
				+ (\PimoK\Delta \psi\PiunoKb (\curl(\psi_h-\psi)), \nabla \varphi- \PiunoKb \nabla \varphi)_{0,\E}\\
				& \qquad + (\PimoK\Delta \psi(\curl \psi_h-\PiunoKb\curl \psi_h), \nabla \varphi)_{0,\E}
				+((\Delta \psi_h-\PimoK \Delta \psi_h)\curl \psi, \nabla \varphi)_{0,\E} \\
				& \qquad  +  (\PimoK\Delta \psi\:\PiunoKb \curl\psi , \nabla \varphi- \PiunoKb \nabla \varphi)_{0,\E}\\
				&=: T_1 + T_2 +T_3+T_4+T_5+T_6+T_7.
			\end{split}	
		\end{equation*}
		In what follows, we will establish estimates  for each terms on the right hand side of the previous identity. 
		For the term $T_1$ we use the H\"older and triangle inequalities,
		along with approximations properties of $\PimoK$, to obtain
		\begin{equation*}%\label{I1}
			\begin{split}
			T_1	& \leq  \sum_{\E \in  \CT_h} \|\Delta \psi_h-\PimoK \Delta \psi_h\|_{0,\E} \|\curl \psi_h-\curl \psi\|_{L^4(\E)} \|\nabla \varphi\|_{L^4(\E)} \\
				& \leq  \sum_{\E \in  \CT_h} (2\|\Delta \psi_h-\Delta \psi\|_{0,\E}+\|\Delta \psi-\PimoK \Delta \psi\|_{0,\E}) \|\curl (\psi_h- \psi)\|_{L^4(\E)} \|\nabla \varphi\|_{L^4(\E)} \\	
				%& \leq  \sum_{\E \in  \CT_h} (2\|\Delta (\psi-\psi_h)\|_{0,\E}+Ch^{\min\{s,k-1\}}\|\Delta \psi\|_{s,\E}) \|\curl (\psi_h- \psi)\|_{\L^4(\E)} \|\nabla \varphi\|_{\L^4(\E)} \\
				&\leq C(|\psi-\psi_h|_{2,h}+h^{\gamma}\|\psi\|_{2+\gamma,\O})|\psi- \psi_h|_{1,4,h} \|\nabla \varphi\|_{L^4(\O)}\\
				&\leq  Ch^{2\gamma} (\|\fb\|_{0,\O}+ \|\psi\|_{2+\gamma,\O}) \| \varphi\|_{2+t,\O},
			\end{split}
		\end{equation*}
		where we have used the H\"older inequality (for sequences), continuous Sobolev inclusion, along with Theorems~\ref{solobev:inclusion} and~\ref{Theo-conv}.
		
		Now, for $T_2$ we follow similar arguments to obtain 
		\begin{equation*}
			\begin{split}
				T_2 \leq Ch^{2\gamma}(\|\fb\|_{0,\O}+ \|\psi\|_{2+\gamma,\O}) \| \varphi\|_{2+t,\O}.
			\end{split}
		\end{equation*}  
		For the term $T_3$ we employ again the H\"older inequality, the continuity of the projector $\PiunoKb$, along with Theorems~\ref{solobev:inclusion} and~\ref{Theo-conv}, to obtain:
		\begin{equation*}%\label{I3}
			\begin{split}
				T_3 & \leq  \sum_{\E \in  \CT_h} \|\PimoK (\Delta \psi_h-\Delta \psi)\|_{0,\E} \|\PiunoKb\curl \psi_h\|_{L^4(\E)} \|\nabla \varphi-\PiunoKb \nabla \varphi\|_{L^4(\E)} \\
				& \leq C|\psi-\psi_h|_{2,h} \|\curl \psi_h\|_{1,4,h}h^{t} |\nabla \varphi|_{W^{t}_4(\O)} \\
				&\leq C h^{\gamma+t}(\|\fb\|_{0,\O}+ \|\psi\|_{2+\gamma,\O}) \| \varphi\|_{2+t,\O}.
			\end{split}
		\end{equation*} 
		For the term $T_4$, we follow similar steps to those used above, to get  
		\begin{equation*}
			\begin{split}
				T_4 & \leq  \sum_{\E \in  \CT_h} \|\PimoK \Delta \psi\|_{0,\E} \|\PiunoKb\curl (\psi_h-\psi)\|_{L^4(\E)} \|\nabla \varphi-\PiunoKb \nabla \varphi\|_{L^4(\E)} \\
				&\leq C h^{\gamma+t}(\|\fb\|_{0,\O}+ \|\psi\|_{2+\gamma,\O}) \| \varphi\|_{2+t,\O}.
			\end{split}
		\end{equation*}  
		Now, for the term $T_5$, we add and subtract suitable terms, use the  H\"older inequality, properties of the $L^2$-projections $\PiunoKb$ and $\PimoK$, together with continuous Sobolev embeddings to obtain 
		\begin{equation*}
			\begin{split}
				T_5 & \leq \sum_{\E \in  \CT_h} \|\PimoK\Delta \psi\|_{L^4(\E)}\|\curl \psi_h-\PiunoKb\curl \psi_h\|_{0,\E} \|\nabla \varphi \|_{L^4(\E)}\\
				& \leq \left( 2|\psi-\psi_h|_{1,h} +Ch^{1+\gamma}\|\psi\|_{2+\gamma,\O}\right) (\CBD \|\Delta \psi\|_{L^4(\O)} \CBD\|\nabla \varphi\|_{L^4(\O)})\\
				&\leq   2\CR \widetilde{C}^2_{{\rm sob}}\CBD^2 \|\fb\|_{0,\O} |\psi-\psi_h|_{1,h} \| \varphi\|_{1+t,\O} 
				+Ch^{\gamma+t}\|\psi\|_{2+\gamma,\O} \| \varphi\|_{2+t,\O},\\	
			\end{split}	
		\end{equation*}
		
Repeating the same arguments, we obtain the following bounds for the terms $T_6$ and $T_7$:
	\begin{equation}	
			T_6+T_7\leq C h^{\gamma+t}(\|\fb\|_{0,\O}+ \|\psi\|_{2+\gamma,\O}) \| \varphi\|_{2+t,\O}. 
		\end{equation}
Finally, by combining the above bounds  we obtain the desired result.
\end{proof}}

Moreover, for the bilinear form $A(\cdot,\cdot)$ we have the following auxiliary result~\cite[Lemma 4.11]{AMS2022-CI2MA}. 
\begin{lemma}\label{Theorem:AD2}
	For $\varphi \in H^{2+t}(\O)$ and $ \phi \in \Phi \cap  H^{2+t}(\O)$,  with $t \in [0,1]$, it holds:
	\begin{equation*}
		A(\varphi,\phi-\phi_I) \leq Ch^{2t}\|\varphi\|_{2+t,\O}\|\phi\|_{2+t,\O},
	\end{equation*}
	where $\phi_I \in \Vh$ is the interpolant of $\phi$ in the virtual space $\Vh$ (cf. Proposition~\ref{approx:virtual}). 
\end{lemma}

In order to establish the desired error estimates we consider the following assumption: 
\begin{equation}\label{assup:additional}
2\CR \widetilde{C}^2_{{\rm sob}}\CBD^2\|\fb\|_{0,\O} < 1,
\end{equation} 
where $\widetilde{C}_{{\rm sob}}$, and $\CR$ and $\CBD$ are the constants 
in  \eqref{sob:cont}, Theorem~\ref{regu-aditional} and \eqref{bound:proj}, respectively.

The next theorem establish the main result of this subsection.
\begin{theorem}\label{Theo:conv:H1:L2}
	Let $\psi \in \Phi \cap H^{2+\gamma}(\O)$ and $\psi_{h} \in \Vh$ be the unique solutions of 
	problems~\eqref{weak:stream:NS} and \eqref{NSE:stream:disc}, respectively.
	Then, under assumption \eqref{assup:additional} there exists 
	a positive constant $C$, independent of $h$, such that	
	\begin{equation}\label{L2:H1:norms}
		\|\psi- \psi_{h}\|_{0,\O}+|\psi- \psi_{h}|_{1,h} \leq Ch^{2\gamma}(\|\psi\|_{2+\gamma,\O}+\|\fb\|_{0,\O}).
	\end{equation}	
\end{theorem}
{\begin{proof}
First we will prove the $H^1$ estimate in \eqref{L2:H1:norms}. To this propose,  
let $\psi_I\in\Vh$ be the interpolant of $\psi$ such that Proposition~\ref{approx:virtual}
holds true. We set $\delta_h:=(\psi_h-\psi_I) \in\Vh$. Then, we write
\begin{equation*}
\psi_h- \psi = (\psi_h-\psi_I) + (\psi_I-\psi)=(\psi_I-\psi) + (\delta_h-E_h \delta_h)+E_h \delta_h.
\end{equation*}
Thus, by using the triangle inequality together with Remark~\ref{remark:boundB2}, Proposition~\ref{approx:virtual}, Lemma \ref{proposition-Eh} and Theorem \ref{Theo-conv}, we obtain  
\begin{equation}\label{trian00}
\begin{split}
\vert\psi-\psi_h\vert_{1,h}&\leq \vert\psi-\psi_I\vert_{1,h}
+\vert \delta_h-E_h \delta_h\vert_{1,h} +\vert E_h\delta_h\vert_{1,h}
\leq Ch^{2\gamma}\|\psi\|_{2+s,\O}  + \|\nabla E_h \delta_h\|_{0,\O}.
\end{split}	
\end{equation}			
Now, the goal is to estimate the term $\|\nabla E_h \delta_h\|_{0,\O}$. 
To do that,	we consider the following dual problem: given $\psi \in \Phi$ 
(the unique solution of the formulation~\eqref{weak:stream:NS}), 
find $\phi \in \Phi$, such that
\begin{equation}\label{dual:problem}
\calA^{DP}(\psi; \varphi, \phi):=	\nu A (\varphi,\phi)+ B(\psi;\varphi,\phi)+ B(\varphi;\psi,\phi)
=  (\nabla(E_h \delta_h),  \nabla \varphi)_{0,\O}  \qquad \forall \varphi \in \Phi,
\end{equation}
where $A(\cdot, \cdot) $ and $B(\cdot;\cdot,\cdot)$ are the  continuous forms defined 
in \eqref{A-cont} and \eqref{B-cont}, respectively. 
Following the same arguments in \cite{KPK-CMAME2021} we have that 
problem~\eqref{dual:problem} is well-posed and from Theorem~\ref{regu-aditional}, 
we obtain that $\phi \in \Phi \cap H^{2+\gamma}(\O)$  and
\begin{equation}\label{reg:axi}
\|\phi\|_{2+\gamma,\O} \leq  C  \|\nabla E_h \delta_h\|_{0,\O},
\end{equation}
where  $C>0$ is a constant independent of $h$.  Taking $\varphi=E_h \delta_h  \in \calW^{\c}_h \subset \Phi$ as test function, 
adding and subtracting $\delta_h$ in problem~\eqref{dual:problem}, we get
\begin{equation}\label{trian01}
\begin{split}
\|\nabla E_h \delta_h\|^2_{0,\O} = \calA^{DP}(\psi; E_h \delta_h, \phi) =
\calA^{DP}(\psi; E_h \delta_h- \delta_h, \phi) +\calA^{DP}(\psi;\delta_h, \phi)
=:I_1+I_2.  
\end{split}
\end{equation} 

Now, we will obtain bounds for the terms $I_1$ and $I_2$ in the above identity. 
For $I_1$, 	we apply Lemma~\ref{Lemma:A} and Proposition~\ref{approx:virtual} to obtain 
\begin{equation}\label{pre:I1}
\begin{split}
I_1 &:= \calA^{DP}(\psi; E_h \delta_h- \delta_h, \phi)\\
&=\nu A(E_h \delta_h-\delta_h,\phi)+B( \psi; E_h \delta_h-\delta_h,\phi)+ B(E_h \delta_h-\delta_h;\psi,\phi)\\	
& \leq C\nu h^{\gamma}|\delta_h|_{2,h}\|\phi\|_{2+\gamma,\O}+ C\|\psi\|_{2+\gamma,\O} 
|E_h \delta_h-\delta_h|_{1,h}\|\phi\|_{2,\O} + B(E_h \delta_h-\delta_h;\psi,\phi)\\
& \leq C\nu h^{2\gamma}\|\psi\|_{2+\gamma,\O}\|\phi\|_{2+s,\O}+ Ch^{2\gamma}\|\psi\|_{2+\gamma,\O} \|\phi\|_{2,\O} + B(E_h \delta_h-\delta_h;\psi,\phi).
\end{split}
\end{equation}
To estimate the  term  $B(E_h \delta_h-\delta_h;\psi,\phi)$ we start recalling that 
 $\psi, \phi \in H^{2+\gamma}(\O)$, with $\gamma \in (1/2,1]$, then by using the Sobolev
inclusion $H^{2+\gamma}(\O) \hookrightarrow W^{1,4}(\O)$, we have
\begin{equation*}
\begin{split}
|\curl \psi \cdot \nabla \phi|_{1,\O} &\leq \|\curl \psi\|_{1,4,\O} \|\nabla \phi\|_{1,4,\O}
\leq \CS^2\|\psi\|_{2+\gamma,\O} \| \phi\|_{2+\gamma,\O} < + \infty. 
\end{split}
\end{equation*}
Therefore, $\curl \psi \cdot \nabla \phi \in H^{1}(\O)$ (hence belongs to $H^{1}(\E)$ for each $\E \in \CT_h$).
Thus, by using the definition of $B(\cdot;\cdot,\cdot)$ we have 
\begin{equation*}
\begin{split}
B(E_h \delta_h-\delta_h;\psi,\phi) 
= \sum_{\E \in  \CT_h}  (\Delta (E_h \delta_h-\delta_h),\curl \psi \cdot \nabla \phi)_{0,\E} 
\leq \sum_{\E \in  \CT_h} \|\Delta (E_h \delta_h-\delta_h)\|_{-1,\E} \|\curl \psi \cdot \nabla \phi\|_{1,\E}.
\end{split}
\end{equation*}
Now, by using the  definition of the dual norm  and an integration by part, we obtain
\begin{equation*}
\begin{split}
\|\Delta (E_h \delta_h-\delta_h)\|_{-1,\E} &= \sup_{\varphi \in H_{0}^1(\E)} \frac{ (\Delta (E_h \delta_h-\delta_h),\varphi)_{0,\E} }{|\varphi|_{1,\E}}  = \sup_{\varphi \in H_{0}^1(\E)} \frac{ (\nabla(E_h \delta_h-\delta_h),\nabla \varphi)_{0,\E} }{|\varphi|_{1,\E}} \\
& \leq |E_h \delta_h-\delta_h|_{1,\E}.
\end{split}
\end{equation*}
From the two  estimates above, the H\"{o}lder inequality for sequences and \eqref{reg:axi},  we have
\begin{equation*}
\begin{split}
B(E_h \delta_h-\delta_h;\psi,\phi) &\leq 
\sum_{\E \in  \CT_h} |E_h \delta_h-\delta_h|_{1,\E} \|\curl \psi \cdot \nabla \phi\|_{1,\E} \leq 
|E_h \delta_h-\delta_h|_{1,h} \|\curl \psi \cdot \nabla \phi\|_{1,\O} \\
&\leq C h^{2\gamma}\|\psi\|_{2+\gamma,\O} \| \phi\|_{2+\gamma,\O}\leq  C h^{2\gamma}\|\psi\|_{2+\gamma,\O} \|\nabla E_h \delta_h\|_{0,\O}.
\end{split}
\end{equation*}
Consequently, inserting  the above inequality in \eqref{pre:I1}, we arrive  to
		\begin{equation}\label{term:I1}
			I_1 \leq C h^{2\gamma}\|\psi\|_{2+\gamma,\O} \|\nabla E_h \delta_h\|_{0,\O}.
		\end{equation}
		Now, we will estimate the remaining term $I_2$. Indeed, we split again $\delta_h:=(\psi_h-\psi)+(\psi-\psi_I)$, then 
		\begin{equation}\label{pre:I2}
			I_2 = 	-\calA^{DP}(\psi; \psi-\psi_h, \phi) + \calA^{DP}(\psi; \psi-\psi_I, \phi) =: -I_{21} +I_{22}.
		\end{equation}
		By using analogous arguments those employed to bound the term $I_1$ and applying Proposition~\ref{approx:virtual} and Lemma~\ref{Theorem:AD2}, we can obtain
		\begin{equation}\label{term:I22}
			I_{22} \leq  C h^{2\gamma}\|\psi\|_{2+\gamma,\O} \|\nabla E_h \delta_h\|_{0,\O}.
		\end{equation}
		Next, adding and subtracting  $\phi_I$, $B(\psi;\psi,\phi_I)$ and other suitable terms
		together with the definition of the continuous and discrete problems (cf. \eqref{weak:stream:NS} and 
		\eqref{NSE:stream:disc}, respectively),  we obtain
		\begin{equation}\label{term:I21}
			\begin{split}
				I_{21} &=\nu A (\psi- \psi_{h},\phi)+ B(\psi;\psi- \psi_{h},\phi)+B(\psi- \psi_{h};\psi,\phi)\\
				&=\nu A(\psi-\psi_h, \phi- \phi_I)+\nu A(\psi-\psi_h, \phi_I)
				+ B(\psi;\psi- \psi_{h},\phi)+ B(\psi- \psi_{h};\psi,\phi)\\
				&= \nu A(\psi-\psi_h, \phi- \phi_I)+F(\phi_I)-F_h(\phi_I) + \nu A_h(\psi_h,\phi_I)
				+B_h(\psi_h;\psi_h,\phi_I) \\
				& \qquad  -B(\psi;\psi,\phi_I)-\nu A(\psi_h, \phi_I) + B(\psi;\psi- \psi_{h},\phi)+ 
				B(\psi- \psi_{h};\psi,\phi)\\
				&= \nu A(\psi-\psi_h, \phi- \phi_I)+  \nu[A_h(\psi_h,\phi_I)- A(\psi_h, \phi_I)] + [F(\phi_I)- F_h(\phi_I)]\\
				& \qquad  +[B_h(\psi_h;\psi_h,\phi_I-\phi) -B(\psi;\psi,\phi_I-\phi)] \\
				& \qquad + B(\psi-\psi_h;\psi- \psi_{h},\phi) +[B_h(\psi_h;\psi_h,\phi)-B(\psi_h;\psi_{h},\phi)]\\ 
				&  =: T_{A1}+T_{A2}+T_{F} +T_{B1}+T_{B2}+ T_{B3},
			\end{split}
		\end{equation}
		where also we have used also the identity
		\begin{equation*}
			\begin{split}
				B(\psi;\psi- \psi_{h},\phi)&+B(\psi- \psi_{h};\psi,\phi)+ B_h(\psi_h;\psi_h,\phi)-B(\psi;\psi,\phi)\\
				&=B(\psi-\psi_h;\psi- \psi_{h},\phi) +[B_h(\psi_h;\psi_h,\phi)-B(\psi_h;\psi_{h},\phi)].
			\end{split}
		\end{equation*}
		By using standard arguments and \eqref{reg:axi} we obtain that 
		\begin{equation}\label{many:terms}
			T_{A1}+T_{A2}+T_{F} +T_{B2} \leq C h^{2\gamma }
			\left(\|\fb\|_{0,\O}+\|\psi\|_{2+\gamma,\Omega} \right)  \|\nabla E_h\delta_h\|_{0,\O}.
		\end{equation}
		For the  remaining term $T_{B1}$,  we employ Lemmas~\ref{lemma-tec2} and~\ref{consist:term},  to obtain
		\begin{equation}\label{TB1}
			\begin{split}
				|T_{B1}| & \leq |B(\psi;\psi,\phi_I-\phi) -B(\psi;\psi,\phi_I-\phi)|
				+|B_h(\psi;\psi,\phi_I-\phi) -B_h(\psi_h;\psi_h,\phi_I-\phi)|\\
				& \leq  C  h^{\gamma} (\|\psi\|_{1+\gamma,\O}+\|\psi\|_{2,\O})\|\psi\|_{2+\gamma,\O}|\phi_I-\phi|_{2,h}\\
				& \quad+ \CBh \left( |\psi_h|_{2,h} |\phi_I-\phi|_{2,h}+|(\psi-\psi_h)+(\phi_I-\phi)|_{2,h}  
				(|\psi|_{2,h} + \|\psi_h\|_{2,\O}) \right)|\phi_I-\phi|_{2,h} \\
				& \leq  Ch^{2\gamma } (\|\psi\|_{1+\gamma,\O}+\|\psi\|_{2,\O})\|\psi\|_{2+\gamma,\O}\|\nabla E_h\delta_h\|_{0,\O}\\
				&\ \quad + Ch^{2\gamma }  (\|\psi\|_{2+\gamma,\O} +\|\fb\|_{0,\O})(\|\psi\|_{2,\O} + |\psi_h|_{2,h})\|\nabla E_h\delta_h\|_{0,\O}\\  
				& \: \quad+ Ch^{2\gamma}(|\psi_h|_{2,h} + \|\psi\|_{2,\O})\|\nabla E_h\delta_h\|_{0,\O},
			\end{split}
		\end{equation} 
	where we have used Theorem~\ref{Theo-conv} and~\eqref{reg:axi}.	For the term $T_{B3}$, we observe that $T_{B3}=T_B(\phi)$, then by  using Lemma~\ref{LemmaTB3} and 
		\eqref{reg:axi} we get
		\begin{equation}\label{TB3}
			\begin{split}
				T_{B3}	&\leq Ch^{2\gamma } (\|\psi\|_{2+\gamma,\O} +\|\fb\|_{0,\O}) \|\phi\|_{2+\gamma,\O}
				+2\CR \widetilde{C}^2_{{\rm sob}}\CBD^2\|\fb\|_{0,\O}|\psi-\psi_h|_{1,h}\|\phi\|_{2+\gamma,\O}\\
				&\leq Ch^{2\gamma } (\|\psi\|_{2+\gamma,\O} +\|\fb\|_{0,\O}) \|\nabla E_h\delta_h\|_{0,\O}
				+2\CR \widetilde{C}^2_{{\rm sob}}\CBD^2\|\fb\|_{0,\O}|\psi-\psi_h|_{1,h}\|\nabla E_h\delta_h\|_{0,\O}.
			\end{split}
		\end{equation}
		Combining \eqref{pre:I2}-\eqref{TB3}, we have
		\begin{equation}\label{term:I2}
			\begin{split}
				|I_2| \leq C(\|\psi\|_{2+s,\O} +\|\fb\|_{0,\O}) \|\nabla E_h\delta_h\|_{0,\O}
				+2\CR \widetilde{C}^2_{{\rm sob}}\CBD^2\|\fb\|_{0,\O}|\psi-\psi_h|_{1,h}\|\nabla E_h\delta_h\|_{0,\O}.
			\end{split}
		\end{equation}
		The desired result follows by combining the estimates \eqref{trian00}, \eqref{trian01}, \eqref{term:I1} and \eqref{term:I2} together with the fact that $(1-2\CR \widetilde{C}^2_{{\rm sob}}\CBD^2\|\fb\|_{0,\O} )>0$ (see assumption \eqref{assup:additional}).
		
		Finally, the $L^2$ estimate in \eqref{L2:H1:norms} is obtained from the triangle inequality, Proposition \ref{approx:virtual}, Lemma \ref{Theorem:AD2} and 
		Theorem \ref{Theo-conv} as follow:
		\begin{equation*}
			\begin{split}
				\Vert\psi-\psi_h\Vert_{0,\O}&\leq \Vert\psi-\psi_I\Vert_{0,\O}+\Vert \delta_h-E_h \delta_h\Vert_{0,\O} 
				+\Vert E_h\delta_h\Vert_{0,\O}\\
				& \leq Ch^{2+\gamma}\|\psi\|_{2+\gamma,\O} + Ch^2(\vert \psi_h-\psi \vert_{2,h}+\vert \psi-\psi_I \vert_{2,h}) + C\vert E_h\delta_h\vert_{1,\O}\\			
				& \leq Ch^{2\gamma}(\|\psi\|_{2+\gamma,\O}+\|\fb\|_{0,\O}),
			\end{split}	
		\end{equation*}		
		where we have used  norm equivalence in $\Phi$.
		The proof is complete.
\end{proof}}

We finish this section establishing the following remark.
\begin{remark}\label{remark-1}
	If $\fb$ is a smooth function, then applying an integration
	by parts and the boundary conditions in \eqref{loadcont}, we have
	that $(\fb, \curl \phi )_{0,\O}=(\rot  \fb,  \phi )_{0,\O} \quad \forall \phi \in \Phi$.
	Inspired by this identity, we can consider an alternative
	right hand side as follows: 
	\begin{equation}\label{load-alter-global}
		\widetilde{F}_h(\phi_h):=\sum_{\E\in\CT_h}(\rot \fb, \PioK \phi_h)_{0,\E}
		\qquad \forall \phi_h\in\Vh.
	\end{equation}
	We note that  $\widetilde{F}_h(\cdot)$ is fully computable using the degrees of freedom 
	$\DVu-\DVd$, since $\PioK$ is computable (cf. Lemma~\eqref{lemm-PioK}). 
	
	For the VE scheme \eqref{NSE:stream:disc} considering the alternative load term~\eqref{load-alter-global},
	we can provide an analogous analysis as the one develop in the  above sections. 
	Therefore, we can obtain the rate of convergences as in Theorems~\ref{Theo-conv} and \ref{Theo:conv:H1:L2}.
	We will present a numerical test to confirm the error estimates in this case (cf. Subsection~\ref{test:small:nu}).
	Moreover, we observe that if the density force is irrotational, i.e., $\fb= \nabla \varphi$ (for some $ \varphi$), it is possible improve the error estimate in Theorem~\ref{Theo-conv} by removing the dependence of the error by the load term $\fb$.

\end{remark}

\section{Postprocessing of further fields of interest}
\label{rec:fields}
\setcounter{equation}{0}
In this section we propose post-processing techniques that allow obtain 
approximations of the velocity, vorticity and pressure fields  from the discrete stream-function $\psi_h$.

\subsection{Postprocessing the velocity field}\label{rec:velocity}
In order to propose an approximation for the velocity field, we recall that  
if $\psi\in \Phi$ the unique solution of continuous problem~\eqref{weak:stream:NS}, 
then 
\begin{equation}\label{continuous:velocity}
\ub = \curl \psi.
\end{equation}
At the discrete level, we define a piecewise linear approximation 
of the velocity field  $\bu$ as
\begin{equation}\label{discrete:velocity}
\widetilde{\ub}_h|_{\E} := \PiunoKb \curl \psi_h,
\end{equation}
where $\psi_h \in \Vh$ is discrete virtual solution delivered by solving problem \eqref{NSE:stream:disc} and the 
operator $\PiunoKb$ is defined by the vectorial version of \eqref{proymdos}.

We have the following result for velocity vector $\widetilde{\ub}_h$.
\begin{theorem}\label{theorem:velocity}
The discrete velocity field $\widetilde{\ub}_h$ defined by the 
relation~\eqref{discrete:velocity} is computable  from  the 
degrees of freedom $\DVu-\DVd$. Moreover,	
under the hypotheses of Theorem~\ref{Theo-conv},
there exists a positive constant $C$, 
independent of $h$, such that
\begin{equation*}
\|\ub- \widetilde{\ub}_h\|_{0,\O}+h^{\gamma}|\ub- \widetilde{\ub}_h|_{1,h}  \leq Ch^{2\gamma}(\|\psi\|_{2+\gamma,\O}+\|\fb\|_{0,\O}).
\end{equation*}
%where $\widetilde{\calK}$ is a suitable function independent of $h$.
\end{theorem}

\begin{proof}
From Lemma~\ref{lemm-PioK} we have immediately the computability of $\widetilde{\ub}_h$ 
by using $\DVu-\DVd$. On the other hand, the error estimate, follow from \eqref{continuous:velocity},
\eqref{discrete:velocity}, the triangular inequality, stability  property  of $\PiunoKb$, together with Theorems~\ref{Theo-conv} and \ref{Theo:conv:H1:L2}. 

\end{proof}

\subsection{Postprocessing the vorticity field}\label{rec:vorticity}    

Due its importance and applications  in fluid mechanics, different works have been 
devoted to approximate the vorticity field of the incompressible 
Navier-Stokes equations; see for instance \cite{AB1999_M2AN,GGS82,ABMCR2015} 
and the references therein. By solving the nonconforming discrete problem~\eqref{NSE:stream:disc}, 
we only obtain an approximation for the stream-function. Nevertheless, 
in this subsection we propose an approximation for the
 vorticity field $\omega$ via postprocessing formula through the discrete 
 stream-function $\psi_h$ and the projection $\PimoK$ defined by the relation \eqref{proymdos}.

First, we recall that $\omega=\rot\ub$, then using the identity
$\ub=\curl\psi$, we have obtain $\omega = \rot \ub =\rot (\curl \psi) = -\Delta \psi. $
Then, at discrete level we define the following approximation for the vorticity:
\begin{equation}\label{discrete:vorticity}
	\widetilde{\omega}_h|_{\E} := -\PimoK(\Delta \psi_h), 
\end{equation}
where $\psi_{h}\in\Vh$ is the unique solution of problem
\eqref{NSE:stream:disc} and $\PimoK$ is defined in \eqref{proymdos}.
 
We have  the following result for the discrete vorticity.

\begin{theorem}\label{teo-conv-voti}
The discrete vorticity field $\widetilde{\omega}_h$ defined by the 
relation~\eqref{discrete:vorticity} is computable  from  the 
degrees of freedom $\DVu-\DVd$. Moreover,	
under the hypotheses of Theorem~\ref{Theo-conv},
there exists a positive constant $C$, 
independent of $h$, such that
\begin{equation*}
\|\omega- \widetilde{\omega}_h\|_{0,\O}  \leq Ch^{\gamma}(\|\psi\|_{2+\gamma,\O}+\|\fb\|_{0,\O}).
\end{equation*}
\end{theorem}
\begin{proof}
The proof follows by using the same arguments in Theorem~\ref{theorem:velocity}.

\end{proof}

\subsection{Postprocessing the pressure field}\label{rec:pressure} 
This subsection is devoted to developing a strategy to recover the pressure variable form the 
discrete stream-function solution $\psi_h$ of problem~\eqref{NSE:stream:disc}, which is based on the 
algorithm presented in~\cite{CN89} and extended to the nonconforming VEM approach. 

We start by recalling that if $\psi \in \Phi$ is the unique solution 
of the weak formulation~\eqref{weak:stream:NS}, then the velocity field is given 
by $\ub=\curl \psi$. Thus, we can write 
\begin{equation}\label{Pressure:eqn1}
	\begin{split}
		b(\vb,p):= (p,\div\vb)_{0,\Omega}
		&= \nu (\nablab \ub,\nablab \vb)_{0,\O}+((\nablab\bu)\ub,\vb)_{0,\O}
		-(\fb,\vb)_{0,\O}\\
		&= \nu(\nablab\curl\psi, \nablab\vb)_{0,\Omega}
		+  ((\nablab\curl\psi)\curl\psi,\vb)_{0,\Omega}
		-(\fb,\vb)_{0,\O} \qquad \forall \vb \in \H.
	\end{split}
\end{equation}    

Now, we consider the functional
$\mathcal{F}(\psi,\fb)(\cdot):\H\to\mathbb{R}$ given by
\begin{equation}\label{pressure:load}
\mathcal{F}(\psi,\fb)(\vb)
	:= \nu(\nablab\curl\psi, \nablab\vb)_{0,\Omega}
	+  ((\nablab\curl\psi)\curl\psi,\vb)_{0,\Omega}
	-(\fb,\vb)_{0,\O} \qquad \forall \vb \in \H.
\end{equation}

By using \eqref{Pressure:eqn1} and \eqref{pressure:load}, we
reformulate~\eqref{standard:weak:NS} as a variational problem for the
pressure variable: given $\psi\in\Phi$ the unique solution of
problem~\eqref{weak:stream:NS} and  $\fb\in	L^2(\Omega)^2$, find $p\in Q$ such that
\begin{equation}\label{pressure:recovery}
b(\vb,p) = \mathcal{F}(\psi,\fb)(\vb) \qquad \forall\vb\in\H,
\end{equation}
where $\H$ and $Q$ are the spaces defined in~\eqref{spaces:H1:L20}. From an equivalence of problems and the LBB theory we have that problem~\eqref{pressure:recovery} has a unique solution $p\in Q$ (see \cite{GR}). 

The difficulties to discretize directly problem~\eqref{pressure:recovery} have been  discussed in \cite[Section 9]{CN86}. Thus, inspired in this work we consider 
the following equivalent problem: find $(\wb,p)\in \H\times Q$, such that
\begin{equation}\label{pressure:mixed}
\begin{split}
a(\wb,\vb) + b(\vb, p) &= \mathcal{F}(\psi,\fb)(\vb)  \qquad\forall\vb\in\H \\ %,\label{eq:pressure:redefine:A}\\
b(\wb, q)              &= 0 \qquad \qquad \qquad  \:\forall q\in Q,%\label{eq:pressure:redefine:B}
\end{split}
\end{equation}	
where $a(\widetilde{\vb},\vb):=(\nablab\widetilde{\vb},\nablab\vb)_{0,\Omega} \quad \forall \widetilde{\vb},\vb \in \H$. 
We have that this Stokes-like problem is well-posed. Moreover, $\wb=0$.  Now the goal is to discretize the problem~\eqref{pressure:mixed}.

\subsubsection{Nonconforming Crouzeix-Raviart-type VE discretization}
In this subsection we will present  a VE scheme to solve problem~\eqref{pressure:mixed}.
First, we recall that the Morley-type VE space $\Vh$ is in a Stokes-complex relation 
with the Crouzeix-Raviart type VE space $\CRS_h$, defined in \eqref{global:space:fnc} 
and \eqref{CR_Glob}, respectively. Apart from the previously mentioned spaces, 
we introduce the space for pressure approximation as 
\begin{equation}
	\label{Pressure:Fld}
	Q_h:=\{q_h\in Q: q_h|_{\E}\in \P_0(\E) \quad \forall \E \in \CT_h \}.
\end{equation}
 At last, we introduce the auxiliary space
\begin{equation}\label{eq:CR:SbSp}
	\widehat{\CRS}_h := \Big\{\vb_h\in \CRS_h :
	\sum_{\E\in\CT_h}  (q_h,\div\vb_h)_{0,\E} =0
	\quad\forall q_h\in Q_h	\Big\},
	\end{equation}
where $\CRS_h$ is the Crouzeix-Raviart-type VE space defined in~\eqref{CR_Glob}.
\begin{lemma}\label{lemma:stokes:complex}
	Let $\Vh$ and  $\widehat{\CRS}_h$  be the spaces defined in~\eqref{global:space:fnc} 
	and  in \eqref{eq:CR:SbSp}, respectively. Then, it holds that
	\begin{align*}
		\curl\,\Vh = \widehat{\CRS}_h,
	\end{align*}
\end{lemma}
\begin{proof}
	The proof can be followed from \cite[Lemma~6.1]{AM2022}.
\end{proof}

By employing the projection operator $\Pivec$ defined in \eqref{proj:H1:vec}, we discretize the
bilinear form $a(\cdot,\cdot)$ through the bilinear form
$a_h:\CRS_h\times\CRS_h\to\mathbb{R}$, which is such
that
\begin{align*}
	a_h(\wb_h,\vb_h)
	:= \sum_{\E\in\CT_h}a_h^K(\wb_h,\vb_h)
	= \sum_{\E\in\CT_h} \Big(
	a^{\E}\big(\Pivec\ub_h,\Pivec\vb_h\big)
	+ S_{\nablab}^{\E}\big((\Ib-\Pivec)\bu_h,(\Ib-\Pivec)\vb_h\big)
	\Big),
\end{align*}

where $S_{\nablab}^{\E}(\cdot,\cdot)$ is a symmetric, positive-definite bilinear
form satisfying the stability condition
\begin{align*}
	c_{\#} a^{\E}(\vb_h,\vb_h) \leq S_{\nablab}^{\E}(\vb_h,\vb_h)\leq c^{\#} a^{\E}(\vb_h,\vb_h)
	\qquad\forall\vb_h\in {\rm Ker}(\Pivec),
\end{align*}
for some pair of strictly positive, real constants $c_{\#}$ and $c^{\#}$, independent of  $h$.

Then, we define the bilinear form $b_h:\CRS_h\times Q_h\to\mathbb{R}$ as

\begin{align}
	b_h(\vb_h,q_h) := \sum_{\E\in\CT_h}  (q_h,\div\vb_h)_{0,\E}.
	\label{eq:bh:def}
\end{align}

The next step is the construction of a discrete version of the lineal functional defined in~\eqref{pressure:load}. 
To do that, first we consider the constant vector field $\boldsymbol{\Pi}_{\E}^0:\CRS_h(\E)\to\P_0(\E)^2$, 
defined on $\CRS_h(\E)$. Then, we consider  the following discrete functional
$\mathcal{F}_h(\psi_h,\fb)(\cdot):\CRS_h\to\mathbb{R}$
\begin{equation}\label{load:pressure:Dis}
	\begin{split}
		\mathcal{F}_h(\psi_h,\fb)(\vb_h)
		:= \sum_{\E\in\CT_h} \Big( a^{\E}\big(\Pivec\curl\psi_h,\Pivec\vb_h\big)
		& + ((\nablab\PiunoKb\curl\psi_h)\PiunoKb \curl \psi_h-\fb,\PizeroK\vb_h)_{0,\E})	\Big).
	\end{split}
\end{equation}
From the stability properties of projectors $\Pivec$, $\PiunoKb$ and $\PizeroK$, we have that the $\mathcal{F}_h(\psi_h,\fb)(\cdot)$  is continuous. Moreover, the projection $\PizeroK$ is computable by using the degrees of freedom $\DU$. Then, from  this fact and Lemma~\ref{lemma:prop:Morley}, we conclude that this functional is fully computable.

Now, we present the virtual element discretization of the Stokes
problem~\eqref{pressure:mixed} that reads as:
find $(\wb_h,p_h) \in \CRS_h \times Q_h$ such that
\begin{equation}\label{pressure:discrete}
	\begin{split}
a_h(\wb_h,\vb_h) + b_h(\vb_h,p_h) &= \mathcal{F}_h(\psi_h,\fb)(\vb_h)  \qquad\forall\vb_h\in\CRS_h,\\[0.5em]
b_h(\wb_h,q_h)                   &= 0 \qquad \qquad \qquad  \qquad \:\forall q_h\in Q_h,		
	\end{split}
\end{equation}	
where $Q_h$ is the space defined in~\eqref{Pressure:Fld}.

The scheme~\eqref{pressure:discrete} is well-posed since 
$a_h(\cdot,\cdot)$ is coercive and continuous, the bilinear 
form $b_h(\cdot,\cdot)$ is continuous and satisfies a discrete 
inf-sup condition on the pair of functional spaces $\CRS_h$-$Q_h$ (see \cite{ZZMC19})
and $\curl\Vh = \widehat{\CRS}_h$. We summarize this fact in the following result. 
\begin{theorem}\label{lemma:dis:inf_sup}
	Let $b_h(\cdot,\cdot)$ be the discrete bilinear form defined in
	\eqref{eq:bh:def}.
	Then, there exists a strictly positive, real constant $C_b>0$ such
	that
	\begin{equation*}
		\underset{\vb_h\in \:\CRS_h\setminus\{\mathbf{0}\}}{\sup} \frac{b_h(\vb_h,q_h)}{|\vb_h|_{1,h}}
		\geq \beta\|q_h\|_{0,\O}	\qquad\forall q_h\in Q_h.
		\end{equation*}
	Moreover, there exist a unique $(\wb_h,p_h) \in \CRS_h \times Q_h$, solution of problem~\eqref{pressure:discrete}.
\end{theorem}
\subsubsection{Error estimate for the pressure scheme}
In this subsection we develop an abstract error result for the virtual scheme presented above. 
Moreover, we provide  error estimates involving some consistent errors. Finally, by combining 
these results  we derive an optimal error estimate for the pressure field.

First, we focus on deriving a bound on the difference between the
functional~\eqref{load:pressure:Dis} applied to the stream-function
$\psi$ solving the continuous variational
formulation~\eqref{weak:stream:NS} and its virtual
element approximation solving~\eqref{NSE:stream:disc}.

\begin{lemma}\label{lemma:difference:load}
Let $\psi\in\Phi$ and $\psi_h\in\Vh$ be the solution to problems~\eqref{weak:stream:NS} and~\eqref{NSE:stream:disc}, respectively. Moreover, let $\mathcal{F}_h(\psi,\fb)(\cdot)$ and $\mathcal{F}_h(\psi_h,\fb)(\cdot)$ be the functionals defined in~\eqref{load:pressure:Dis} (applied to $\psi$
and $\psi_h$, respectively). Then, there exists a real, positive constant $C_{\mathcal{F}_h}>0$,
independent of $h$, such that
\begin{equation*}
\big| \mathcal{F}_h(\psi,\fb)(\vb_h)-\mathcal{F}_h(\psi_h,\fb)(\vb_h) \big|
\leq C_{\mathcal{F}_h}\vert\psi-\psi_h\vert_{2,h}\vert\vb_h\vert_{1,h}.
\end{equation*}
\end{lemma}
\begin{proof}
Upon employing the definition~\eqref{load:pressure:Dis}, we obtain
\begin{align*}
	& \big| \mathcal{F}_h(\psi,\fb)(\vb_h)-\mathcal{F}_h(\psi_h,\fb)(\vb_h) \big|
	\leq \nu\sum_{\E \in \CT_h} \big| a^{\E} (\Pivec\curl(\psi-\psi_h),\Pivec\vb_h) \big| \nonumber\\ 
	& +\sum_{\E \in \CT_h} \big|((\nablab\PiunoKb\curl\psi)\PiunoKb\curl \psi,\PizeroK\vb_h)_{0,\E}-((\nablab\PiunoKb\curl\psi_h)\PiunoKb\curl \psi_h,\PizeroK\vb_h)_{0,\E} \big|. 
	\end{align*}
Since $\Pivec$ is a continuous operator with respect to the
$H^{1}$-inner product, we bound the first term as follows
\begin{equation*}
	\nu\sum_{\E\in \CT_h}\big| a^{\E} (\Pivec\curl(\psi-\psi_h),\Pivec\vb_h) \big|
	\leq C\nu\vert\psi_h-\psi\vert_{2,h}\vert\vb_h\vert_{1,h}.
\end{equation*}
By adding and subtracting the term $((\nablab\PiunoKb\curl\psi_h)\PiunoKb\curl \psi,\PizeroK\vb_h)_{0,\E}$, applying the H\"{o}lder inequality and Theorem~\ref{solobev:inclusion},  along with stability properties of projectors $\Pivec$, $\PiunoKb$ and $\PizeroK$,, we obtain
\begin{equation*}
	\begin{split}
		&\sum_{\E \in \CT_h} \big| ((\nablab\PiunoKb\curl\psi)\PiunoKb\curl \psi,\PizeroK\vb_h)_{0,\E}-((\nablab\PiunoKb\curl\psi_h)\PiunoKb\curl \psi_h,\PizeroK\vb_h)_{0,\E} \big| \\
		&=\sum_{\E \in \CT_h} \big|  ((\nablab\PiunoKb\curl(\psi-\psi_h))\PiunoKb\curl \psi,\PizeroK\vb_h)_{0,\E} \big|+ \big| ((\PiunoKb\curl\psi_h)\PiunoKb\curl( \psi-\psi_h),\PizeroK\vb_h))_{0,\E} \big|\\
		&\leq C \vert \psi-\psi_h \vert_{2,h} \vert \psi \vert_{2,\O} \vert \vb_h \vert_{1,h}+C \vert\psi_h \vert_{2,h} \vert \psi-\psi_h \vert_{2,h} \vert \vb_h \vert_{1,h}.
	\end{split}
\end{equation*}
The  result follows by combining the above estimates.
\end{proof}

In continuation, we define the consistency error $\Theta_h(\cdot,\cdot)$ as follows: 
given $\psi \in \Phi$ the solution of problem~\eqref{weak:stream:NS}, we consider 
\begin{equation}\label{error:consis:pressure}
\Theta_h(\psi,\vb_h) :=\mathcal{F}_h(\psi,\fb)(\vb_h)-b_h(\vb_h,p)
	\qquad\forall\vb_h \in \CRS_h.
\end{equation}
We have the following abstract error estimate for the pressure recovery scheme.
\begin{theorem}\label{pre:error:pressure}
Let $\psi\in \Phi \cap H^{2+\gamma}(\Omega)$, with $\gamma\in(1/2,1]$ and $\psi_h \in \Vh$ 
be the solutions of problems~\eqref{weak:stream:NS} and \eqref{NSE:stream:disc}, respectively. 
Moreover, let $(\wb, p)\in \H\times Q$  and $(\wb_h, p_h)\in \CRS_h \times Q_{h}$ be the solutions of
problems~\eqref{pressure:mixed} and~\eqref{pressure:discrete}.		
	Then, there exists a strictly positive, real constant $C>0$,
	independent of $h$, such that
	\begin{equation}\label{abst:press:esti}
	\|p-p_h\|_{0,\O} \leq C \Big( \inf_{q_h \in Q_h} \|p-q_h\|_{0,\O} + \sup_{\substack{\vb_h \in \CRS_h \\ \vb_h\neq\, \0}} \frac{|\Theta(\psi,\vb_h)|}{|\vb_h|_{1,h}} + |\psi-\psi_h|_{2,h} \Big),
	\end{equation}
where $\Theta(\psi,\cdot)$ is the consistency error defined in \eqref{error:consis:pressure}. 
\end{theorem} 

\begin{proof}
Adding and subtracting adequate terms in \eqref{pressure:discrete}, for each  $\vb_h \in \CRS_h$ we have 
\begin{equation}\label{equa:pressure}
\begin{split}
a_h(\wb_h,\vb_h )&= \mathcal{F}_h(\psi_h,\fb)(\vb_h)-b_h(\vb_h,p_h) \\
 &=\mathcal{F}_h(\psi_h,\fb)(\vb_h)- \mathcal{F}_h(\psi,\fb)(\vb_h)
 +\mathcal{F}_h(\psi,\fb)(\vb_h)- b_h(\vb_h, p) + b_h(\vb_h, p-p_h)\\
&=(\mathcal{F}_h(\psi_h,\fb)(\vb_h)- \mathcal{F}_h(\psi,\fb)(\vb_h))
+\Theta(\psi,\vb_h) + b_h(\vb_h, p-p_h).
\end{split}
\end{equation}
Taking $\vb_h = \wb_h$ in \eqref{equa:pressure}, then by using the fact that  
$b_h(\wb_h, q_h) = b_h(\wb_h, p_h) = 0 \quad \forall q_h \in Q_h$, the continuity of $b_h(\cdot,\cdot)$ and Lemma~\ref{lemma:difference:load}, we get 
\begin{equation}\label{ineq:wh}
|\wb_h|_{1,h} \leq C \Big(|\psi-\psi_h|_{2,h} + \|p-q_h\|_{0,\O} + \sup_{\substack{\vb_h \in \CRS_h \\ \vb_h\neq\, \0}} \frac{|\Theta(\psi,\vb_h)|}{|\vb_h|_{1,h}} \Big). 	
\end{equation}
By using again \eqref{equa:pressure} and the linearity of $b_h(\cdot,\cdot)$, for all $q_h \in Q_h$
 we have 
\begin{equation*}
\begin{split}
b_h(\vb_h, q_h - p_h) &= b_h(\vb_h, q_h- p) + b_h(\vb_h, p- p_h)\\ 
&=  b_h(\vb_h, q_h- p)+a_h(\wb_h,\vb_h )- (\mathcal{F}_h(\psi_h,\fb)(\vb_h)- \mathcal{F}_h(\psi,\fb)(\vb_h))
-\Theta(\psi,\vb_h).
\end{split}
\end{equation*}
Thus, by using the two last estimate above,  Lemma~\ref{lemma:difference:load}, the inf-sup condition in Lemma~\ref{lemma:dis:inf_sup}, we obtain
\begin{equation*}
\beta\| q_h - p_h\|_{0,\O}	\leq  C \Big( \|p-q_h\|_{0,\O} + |\wb_h|_{1,h}  +\sup_{\substack{\vb_h \in \CRS_h \\ \vb_h\neq\, \0}} \frac{|\Theta(\psi,\vb_h)|}{|\vb_h|_{1,h}}  \Big).
\end{equation*}
The desired result follows from the triangle inequality, the above estimate and \eqref{ineq:wh}. 
\end{proof}

\begin{lemma}\label{lemma:Pressure:VC:bound}
Let $\psi\in \Phi \cap H^{2+\gamma}(\Omega)$, $\gamma\in(1/2,1]$, be the solution of 
	problem~\eqref{weak:stream:NS}. Then, there exists a strictly positive, real constant $C>0$,
	independent of $h$, such that
	\begin{equation*}%\label{ConsistErrEst}
		|\Theta_h(\psi,\vb_h)| \leq Ch^{\gamma}(\|p\|_{\gamma,\Omega} + \|\psi\|_{2+\gamma,\Omega} 
		+ \|\fb\|_{0,\Omega})|\vb_h|_{1,h}
		\qquad\forall\vb_h \in\CRS_h.
	\end{equation*}
\end{lemma}
\begin{proof}
By using the definition of the consistency term $\Theta_h(\psi,\cdot)$ (cf.~\eqref{error:consis:pressure}),
the weak continuity of the discrete function of the Crouzeix-Raviart space on edges, and employing standard arguments as \cite[Theorem~13]{ZZMC19}, together with the real method of interpolation, we can obtain the required result.
\end{proof}

Finally, the next result  provide the rate of convergent for our pressure VE scheme.
\begin{theorem}
Under same assumptions of Theorem~\ref{pre:error:pressure}, 
for $p\in  Q\cap H^{\gamma}(\Omega)$, there exists $C>0$, independent of $h$, such that
\begin{equation*}
\|p-p_h\|_{0,\Omega} \leq Ch^{\gamma}( \|p\|_{\gamma,\Omega} + \|\psi\|_{2+\gamma,\Omega} 
+ \|\fb\|_{0,\Omega}).
\end{equation*}
\end{theorem}
\begin{proof}
The demonstration follows from \eqref{abst:press:esti}, taking $q_h =\Pi_\E^{0} p$ in  Theorem~\ref{pre:error:pressure}, Lemma~\ref{lemma:Pressure:VC:bound} and Theorem~\ref{Theo-conv}. 

\end{proof}
%%%%%%%%%%------------------------------------
\begin{remark}
We recall that if we are interesting to approximate only the main unknown of problem~\eqref{weak:stream:NS}, 
we can consider the Morley-type VE introduced in \cite[Subsection 3.2]{AMS2022-CI2MA}, 
avoiding the construction of the Stokes complex sequence. Moreover, we are able to recover 
the velocity and vorticity fields  by using the postprocessing of subsections~\ref{rec:velocity} 
and \ref{rec:vorticity}, and obtain the theoretical analysis presented here. 
However, the pressure recovery would not be available. Thus, we point out
that the main advantage to used Stokes complex sequence associated to $\Vh$  and $\CRS_h$ is that 
we can additionally compute  the pressure field from the discrete stream-function, 
with optimal rate of convergence, making the suitable setting.
\end{remark}

\section{Numerical result}\label{numeric_Result}
\setcounter{equation}{0}
In this section, we present four numerical experiments to test the practical performance 
of the proposed virtual element discretization~\eqref{NSE:stream:disc} and assess the 
theoretical predictions as estimated in Sections~\ref{Error:Analisis} and~\ref{rec:fields}. 
We first approximate the discrete stream-function $\psi$ by employing Morley-type VE space 
\eqref{global:space:fnc}, and then we recovered other fields of interest such as velocity, 
and vorticity by employing suitable projection operators. Further, we recover the pressure 
variable by solving a saddle point problem, where the velocity space are in Stokes complex
relationship with the stream-function space (cf. Section~\ref{rec:pressure}).
In each test to solve the nonlinear system resulting from~\eqref{NSE:stream:disc}, we apply 
the Newton method, with a fixed tolerance of ${\rm Tol}=10^{-8}$ and the initial guess
is given by $\psi_h^{{\rm in}}=0$.
  
We have tested the method by using different families of meshes such as: 
\begin{itemize}
\item $\CT^1_h$: Square meshes; 
\item $\CT^2_h$: Triangular meshes; 
\item
 $\CT^3_h$: Sequence of CVT (Centroidal Voronoi Tessellation); 
 \item  $\CT^4_h$: Trapezoidal meshes,
\end{itemize} 
which are posted in Figure~\ref{mesh_PolyGon}. We quantify the errors by employing the projection operators:
$\PiK$, $\PiunoKb$, and $\PimoK$. The following formulations are used for the computation of  experimental 
errors:
\begin{equation}\label{errors}
\begin{split}
&\mathcal{E}_{i}(\psi):= \Big( \sum_{\E\in \CTh}|\psi-\PiK\psi_h|_{i,\E}^2 \Big)^{1/2} 
\quad \forall i \in \{ 0,1,2 \},\quad \mathcal{E}_{j}(\ub):= 
\Big( \sum_{\E\in \CTh}|\ub-\PiunoKb \curl \psi_h|_{j,\E}^2 \Big)^{1/2} \quad  \forall j\in \{ 0,1\}; \\
& \quad \mathcal{E}_0(\omega):= \Big( \sum_{\E\in \CTh}\|\omega-\PimoK \Delta \psi_h\|_{0,\E}^2 \Big)^{1/2}, \quad \mathcal{E}_0(p):= \Big( \sum_{\E\in \CTh}\|p-p_h \|_{0,\E}^2 \Big)^{1/2}.
\end{split}
\end{equation}
Furthermore, we let $\mathcal{R}_i(\chi)$, where $\chi\in \{ \ub,\psi, \omega \}$, and $i \in \{0,1,2\}$ denotes the rates of convergence of the approximate solutions in $H^2$-, $H^1$- and $L^2$-norms.

\begin{figure}[h!]
	\begin{center}
\includegraphics[height=3.5cm, width=3.5cm]{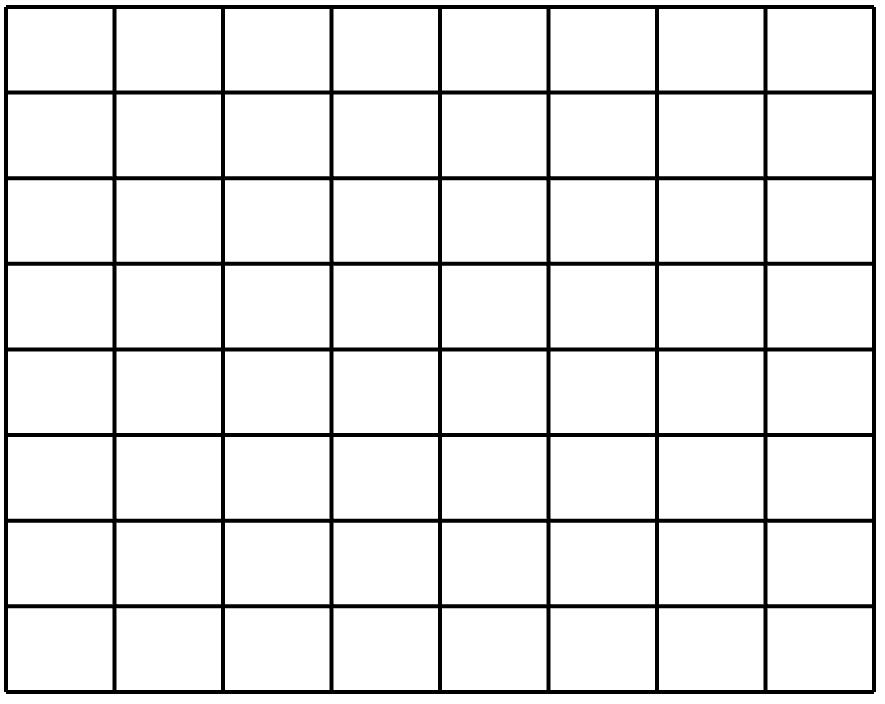} 
\includegraphics[height=3.5cm, width=3.5cm]{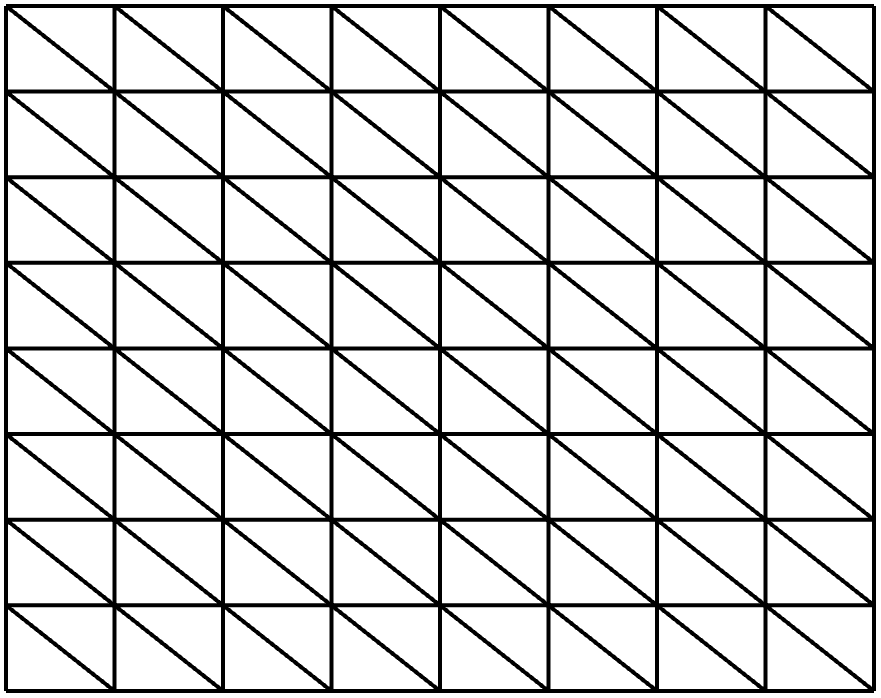}  
\includegraphics[height=3.5cm, width=3.5cm]{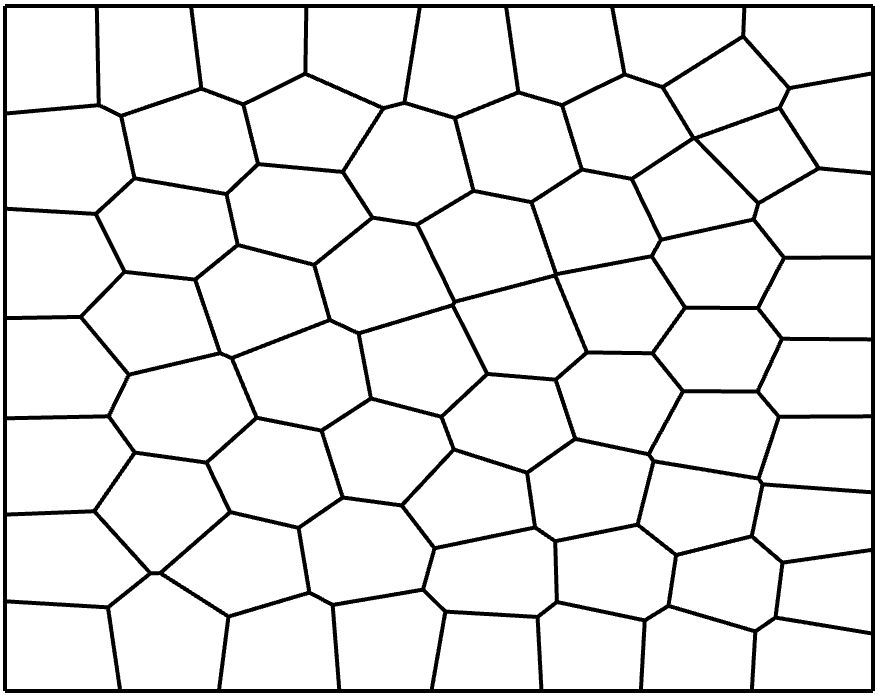} 
\includegraphics[height=3.5cm, width=3.5cm]{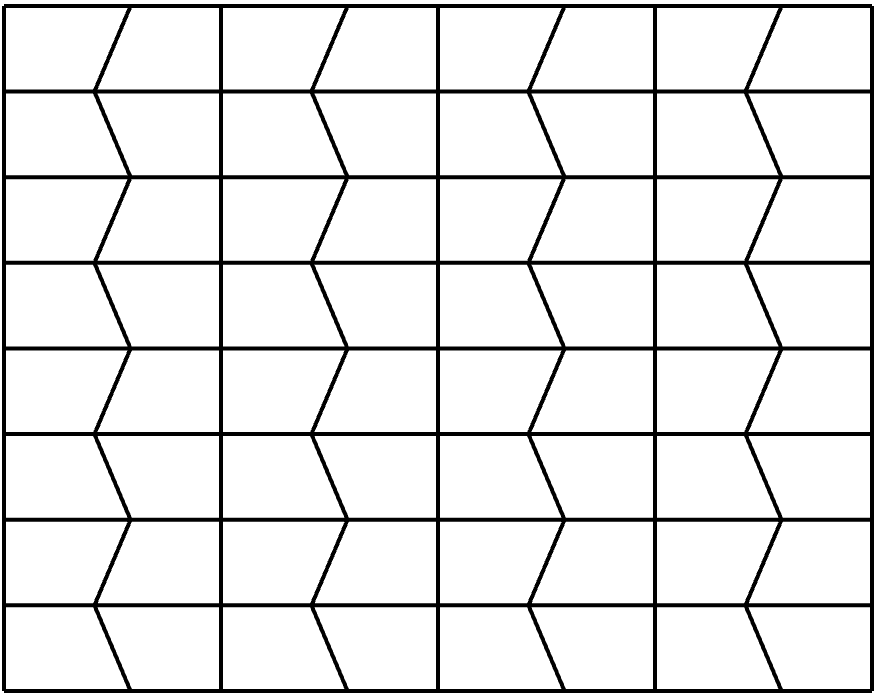}
	\end{center}
	\caption{Sample meshes. $\CT^1_h$, $\CT^2_h$, $\CT^3_h$  and $\CT^4_h$ (from left to right).}
	\label{mesh_PolyGon} 
\end{figure}

\subsection{Test 1. Kovasznay flow} \label{test:1}
In this numerical test, we solve the Navier-Stokes 
problem~\eqref{NSE:velocity:pressure} on the square domain $\overline{\O} := [0,1]^2$.
We take the load term $\fb$ and boundary conditions
in such a way that the analytical solution is given by the Kovasznay solution:
\begin{equation*}
\begin{split}
&\bu(x,y)=
\begin{pmatrix}
1-\exp(\lambda x) \cos(2 \pi y)\\
\frac{\lambda}{2 \pi} \exp(\lambda x) \sin(2 \pi y)
\end{pmatrix}, \quad 
 \psi(x,y)=y-\frac{1}{2 \pi} \exp(\lambda x) \sin(2 \pi y),\\
& p(x,y)=-(1/2) \exp(2 \lambda x)+\bar{p}, \quad \quad \quad  \omega= \Big( \frac{\lambda^2-4\pi^2}{2 \pi}\Big) \exp(\lambda x) \sin(2 \pi y),
\end{split}
\end{equation*}
where  $\lambda=\frac{Re}{2}-\Big( \frac{Re^2}{4}+4 \pi^2\Big)^{1/2}$, and $Re=\nu^{-1}$. We have computed the discrete stream-function for different choice of viscosity coefficients, e.g., $\nu=1, 0.01$, and errors for the stream-function
(cf.~\eqref{errors}) are posted in Figure~\ref{fig:Exampele1:psi:nu1}, and Figure~\ref{fig:Exampele1:psi:nu.01}, respectively. Further, by employing the formulas \eqref{discrete:velocity} and \eqref{discrete:vorticity}, we have recovered discrete velocity and vorticity fields for $\nu=1,0.01$. The error curves of the velocity and vorticity are posted in Figure~\ref{fig:Exampele1:velvort:nu1}, and Figure~\ref{fig:Exampele1:velvort:nu0.01}, while the  error  curves for the pressure are posted in Figure~\ref{fig:Example1:p} for both values of  $\nu$. 
Besides, for all the meshes the maximum number of iterations that are required to achieve the tolerance in the Newton method is $4$ for $\nu=1$ and $6$ for $\nu=0.01$.
 
In Figure~\ref{example1_psi_Vel_Vort_p}, we have posted the discrete stream-function and pressure fields for $\nu=1$, using the mesh $\CT^1_h$, with $h=1/32$.  

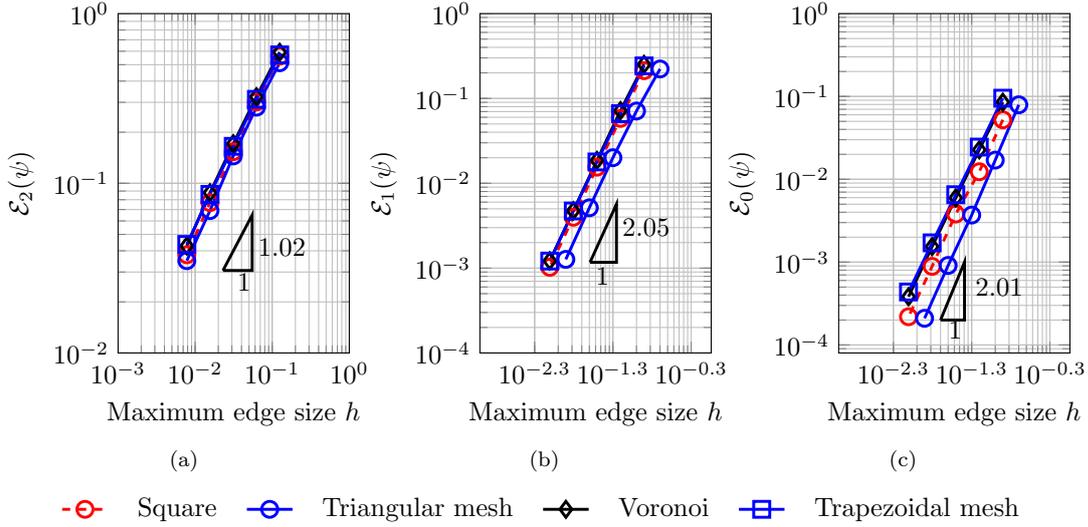
\begin{figure}[h!]%[htpb!]
	\centering
	\newlength\figureheight 
	\newlength\figurewidth
	\setlength\figureheight{4.5cm} %4.8cm
	\setlength\figurewidth{3.2cm} % 3.4cm
	\subfigure[]{% This file was created by matlab2tikz.
% Minimal pgfplots version: 1.3
%  http://www.mathworks.com/matlabcentral/fileexchange/22022-matlab2tikz
\begin{tikzpicture}

\begin{axis}[%
width=0.95092\figurewidth,
height=\figureheight,
at={(0\figurewidth,0\figureheight)},
scale only axis,
xmode=log,
xmin=0.001,
xmax=1,
xtick={0.001,  0.01,    0.1, 1},
xminorticks=true,
xmajorgrids,
xminorgrids,
xlabel={Maximum edge size $h$},
ymode=log,
ymin=0.01,
ymax=1,
ytick={ 1e-05, 0.0001,  0.001,   0.01,    0.1,      1},
yminorticks=true,
ymajorgrids,
yminorgrids,
ylabel={$\mathcal{E}_2(\psi)$}, %  $H^1$ Error for $\psi$
legend style={legend cell align=left,align=left,draw=white!15!black},
legend pos = south east
]
\addplot [color=red,dashed,line width=1.1pt,mark=o,mark size=3pt,mark options={solid}]
  table[row sep=crcr]{%
0.1250 .56138285 \\  
0.0625  .301331189\\
0.0313 .1534408722\\
0.0156  .07686947\\
.007812 .037905597 \\
};
%\addlegendentry{Square};

%\addplot [color=black,dashed,line width=1.1pt, mark=triangle,mark size=6pt]
%  table[row sep=crcr]{%
%0.235702260395987	0.063616650758363\\
%0.128564869306516	0.036208342254826\\
%0.083189033081103	0.020011507568844\\
%0.061487546190749	0.013450590739499\\
%0.050507627227206	0.010752695027151\\
%};
%\addlegendentry{Distorted square};

\addplot [color=blue,solid,mark=o,mark options={solid},line width=1.1pt,mark size=3pt]
  table[row sep=crcr]{%
0.1250  0.5130439566 \\
.0625  0.28081771661\\
.0313 0.1449009136\\
0.0156  0.0690190564\\
.007812 .034991263\\
};
%\addlegendentry{Triangular mesh};

\addplot [color=black,solid,mark=diamond,mark options={solid},line width=1.1pt,mark size=3pt]
  table[row sep=crcr]{%
0.1250 .59140219\\
0.0625 0.3235838457\\
0.0313 0.171016981277\\
0.0156 0.087912461522\\
.007812 .0427542453\\
};
%\addlegendentry{voronoi};

\addplot [color=blue,solid, mark=square, mark options={solid},line width=1.1pt,mark size=3pt]
  table[row sep=crcr]{%
0.1250  0.57091832\\
0.0625 0.312376160\\
0.0313 0.1650936186\\
0.0156 0.086052222\\
.007812 0.04362673333\\
};
%\addlegendentry{Trapizoidal};

\addplot [color=black,solid,forget plot,line width=1.1pt]
  table[row sep=crcr]{%
0.023	0.030626\\
0.055	0.030626\\
0.055	0.06500\\
0.023	0.030626\\
};
\node[right, align=left, inner sep=0mm, text=black]
at (axis cs:0.035000,0.026,0) {$1$};
\node[right, align=left, inner sep=0mm, text=black]
at (axis cs:0.06500,.042,0) {$1.02$};
\end{axis}
\end{tikzpicture}%
} \hspace*{-4.5mm}
	\subfigure[]{% This file was created by matlab2tikz.
% Minimal pgfplots version: 1.3
%  http://www.mathworks.com/matlabcentral/fileexchange/22022-matlab2tikz
\begin{tikzpicture}

\begin{axis}[%
width=0.95092\figurewidth,
height=\figureheight,
at={(0\figurewidth,0\figureheight)},
scale only axis,
xmode=log,
xmin=0.001,
xmax=0.9,
xtick={0.005,  0.05,    0.5},
xminorticks=true,
xmajorgrids,
xminorgrids,
xlabel={Maximum edge size $h$},
ymode=log,
ymin=0.0001,
ymax=1,
ytick={ 1e-05, 0.0001,  0.001,   0.01,    0.1,      1},
yminorticks=true,
ymajorgrids,
yminorgrids,
ylabel={$\mathcal{E}_1(\psi)$}, %  $H^1$ Error for $\psi$
legend style={legend cell align=left,align=left,draw=white!15!black},
legend pos = south east
]
\addplot [color=red,dashed,line width=1.1pt,mark=o,mark size=3pt,mark options={solid}]
  table[row sep=crcr]{%
0.1250 0.2123022 \\  
0.0625  0.058146642\\
0.0313 0.015477051\\
0.0156   0.003978042\\
0.0078  0.001022469481\\
};
%\addlegendentry{Square};

%\addplot [color=black,dashed,line width=1.1pt, mark=triangle,mark size=6pt]
%  table[row sep=crcr]{%
%0.235702260395987	0.063616650758363\\
%0.128564869306516	0.036208342254826\\
%0.083189033081103	0.020011507568844\\
%0.061487546190749	0.013450590739499\\
%0.050507627227206	0.010752695027151\\
%};
%\addlegendentry{Distorted square};

\addplot [color=blue,solid,mark=o,mark options={solid},line width=1.1pt,mark size=3pt]
  table[row sep=crcr]{%
0.2000  0.222342960 \\
.1000  0.071002123\\
.0500 0.0200112382\\
0.0250  0.0051434576\\
0.0125 0.0012681615\\	
};
%\addlegendentry{Triangular mesh};

\addplot [color=black,solid,mark=diamond,mark options={solid},line width=1.1pt,mark size=3pt]
  table[row sep=crcr]{%
0.1250 0.25189015\\
0.0625 0.0718367858\\
0.0313 0.0185925261\\
0.0156 0.00474579\\
0.0078 0.00121980530\\
};
%\addlegendentry{Voronoi};

\addplot [color=blue,solid,mark=square,mark options={solid},line width=1.1pt,mark size=3pt]
  table[row sep=crcr]{%
0.1250  0.24171301\\
0.0625 0.066126256\\
0.0313 0.0178413306\\
0.0156 0.00471465\\
0.0078 0.0012117997473 \\
};
%\addlegendentry{Trapizoidal mesh};

\addplot [color=black,solid,forget plot,line width=1.1pt]
  table[row sep=crcr]{%
0.0255	0.00117\\
0.0555	0.00117\\
0.0555	0.00567\\
0.0255	0.00117\\
};
\node[right, align=left, inner sep=0mm, text=black]
at (axis cs:0.0300,0.00080,0) {$1$};
\node[right, align=left, inner sep=0mm, text=black]
at (axis cs:0.0650,.003000,0) {$2.05$};
\end{axis}
\end{tikzpicture}%
} \hspace*{-4.5mm}
	\subfigure[]{% This file was created by matlab2tikz.
% Minimal pgfplots version: 1.3
%  http://www.mathworks.com/matlabcentral/fileexchange/22022-matlab2tikz
\begin{tikzpicture}

\begin{axis}[%
width=0.95092\figurewidth,
height=\figureheight,
at={(0\figurewidth,0\figureheight)},
scale only axis,
xmode=log,
xmin=0.001,
xmax=0.9,
xtick={0.005,  0.05,    0.5},
xminorticks=true,
xmajorgrids,
xminorgrids,
xlabel={Maximum edge size $h$},
ymode=log,
ymin=0.00008,
ymax=1,
ytick={ 1e-05, 0.0001,  0.001,   0.01,    0.1,      1},
yminorticks=true,
ymajorgrids,
yminorgrids,
ylabel={$\mathcal{E}_0(\psi)$}, % \mathcal{E}_i(\psi_h)
legend style={legend cell align=left,align=left,draw=white!15!black},
legend pos = south east
]
\addplot [color=red,dashed,line width=1.1pt,mark=o,mark size=3pt,mark options={solid}]
  table[row sep=crcr]{%
0.1250 .0522152119 \\  
0.0625  0.0123186358\\
0.0313 0.0038076383\\
0.0156   8.8816303e-4\\
0.0078 2.18983858e-04\\
};
%\addlegendentry{Square};

%\addplot [color=black,dashed,line width=1.1pt, mark=triangle,mark size=6pt]
%  table[row sep=crcr]{%
%0.235702260395987	0.063616650758363\\
%0.128564869306516	0.036208342254826\\
%0.083189033081103	0.020011507568844\\
%0.061487546190749	0.013450590739499\\
%0.050507627227206	0.010752695027151\\
%};
%\addlegendentry{Distorted square};

\addplot [color=blue,solid,mark=o,mark options={solid},line width=1.1pt,mark size=3pt]
  table[row sep=crcr]{%
0.2000  0.0792090291 \\
.1000  0.017027108\\
.0500 0.0036986012\\
0.0250  9.1192007020e-4\\
0.0125  2.09784396e-4\\	
};
%\addlegendentry{Triangular mesh};

\addplot [color=black,solid,mark=diamond,mark options={solid},line width=1.1pt,mark size=3pt]
  table[row sep=crcr]{%
0.1250 0.0851713\\
0.0625 0.022663448\\
0.0313 0.0059475498\\
0.0156 0.00153932244\\
0.0078 3.8217239e-4\\
};
%\addlegendentry{Voronoi};

\addplot [color=blue,solid,mark=square,mark options={solid},line width=1.1pt,mark size=3pt]
  table[row sep=crcr]{%
0.1250  0.0951887\\
0.0625 0.0244662046\\
0.0313 0.0065102746\\
0.0156 0.0017084859\\
0.0078 4.360962270e-4 \\
};
%\addlegendentry{Trapoizal mesh};

\addplot [color=black,solid,forget plot,line width=1.1pt]
  table[row sep=crcr]{%
0.0200	0.0002\\
0.0405	0.0002\\
0.0405	0.0010\\
0.0200	0.0002\\
};
\node[right, align=left, inner sep=0mm, text=black]
at (axis cs:0.025,0.00015,0) {$1$};
\node[right, align=left, inner sep=0mm, text=black]
at (axis cs:0.05500,.000500,0) {$2.01$};
\end{axis}
\end{tikzpicture}%
} 
	\begin{tikzpicture}
		\begin{customlegend}[legend columns=5,legend style={align=center,draw=none,column sep=2ex},
			legend entries={Square, Triangular mesh,  Voronoi, Trapezoidal mesh}]
			\addlegendimage{color=red,dashed,line width=1.1pt,mark=o,mark size=3pt,mark options={solid}}
			%\addlegendimage{color=black,dashed,line width=1.1pt, mark=triangle,mark size=6pt}
			\addlegendimage{color=blue,solid,mark=o,mark options={solid},line width=1.1pt,mark size=3pt}   
			\addlegendimage{color=black,solid,mark=diamond,mark options={solid},line width=1.1pt,mark size=3pt}
			\addlegendimage{color=blue,solid,mark=square,mark options={solid},line width=1.1pt,mark size=3pt}
		\end{customlegend}
	\end{tikzpicture}
\caption{Test 1: Convergence of the stream-function $\psi$ in broken $H^2$-, $H^1$- and $L^2$-norms
with mesh refinement for different types of meshes, using  $\nu=1$.}
\label{fig:Exampele1:psi:nu1}
\end{figure} 

\begin{figure}[h!]
	\centering
	\setlength\figureheight{4.5cm} 
	\setlength\figurewidth{3.2cm}
	\subfigure[]{% This file was created by matlab2tikz.
% Minimal pgfplots version: 1.3
%  http://www.mathworks.com/matlabcentral/fileexchange/22022-matlab2tikz
\begin{tikzpicture}

\begin{axis}[%
width=0.95092\figurewidth,
height=\figureheight,
at={(0\figurewidth,0\figureheight)},
scale only axis,
xmode=log,
xmin=0.001,
xmax=1,
xtick={0.001,  0.01,    0.1, 1},
xminorticks=true,
xmajorgrids,
xminorgrids,
xlabel={Maximum edge size $h$},
ymode=log,
ymin=0.01,
ymax=1,
ytick={ 1e-05, 0.0001,  0.001,   0.01,    0.1,      1},
yminorticks=true,
ymajorgrids,
yminorgrids,
ylabel={$\mathcal{E}_1(\ub)$}, %  $H^1$ Error for $\psi$
legend style={legend cell align=left,align=left,draw=white!15!black},
legend pos = south east
]
\addplot [color=red,dashed,line width=1.1pt,mark=o,mark size=3pt,mark options={solid}]
  table[row sep=crcr]{%
0.1250 .56138285 \\  
0.0625  .301331189\\
0.0313 .1534408722\\
0.0156  .07686947\\
.007812 .037905597 \\
};
%\addlegendentry{Square};

%\addplot [color=black,dashed,line width=1.1pt, mark=triangle,mark size=6pt]
%  table[row sep=crcr]{%
%0.235702260395987	0.063616650758363\\
%0.128564869306516	0.036208342254826\\
%0.083189033081103	0.020011507568844\\
%0.061487546190749	0.013450590739499\\
%0.050507627227206	0.010752695027151\\
%};
%\addlegendentry{Distorted square};

\addplot [color=blue,solid,mark=o,mark options={solid},line width=1.1pt,mark size=3pt]
  table[row sep=crcr]{%
0.2000  0.5130439566 \\
.1000  0.28081771661\\
.0500 0.1449009136\\
0.0250  0.0690190564\\
.012500 .034991263\\
};
%\addlegendentry{Triangular mesh};

\addplot [color=black,solid,mark=diamond,mark options={solid},line width=1.1pt,mark size=3pt]
  table[row sep=crcr]{%
0.1250 .59140219\\
0.0625 0.3235838457\\
0.0313 0.171016981277\\
0.0156 0.087912461522\\
.007812 .0427542453\\
};
%\addlegendentry{voronoi};

\addplot [color=blue,solid, mark=square, mark options={solid},line width=1.1pt,mark size=3pt]
  table[row sep=crcr]{%
0.1250  0.57091832\\
0.0625 0.312376160\\
0.0313 0.1650936186\\
0.0156 0.086052222\\
.007812 0.04362673333\\
};
%\addlegendentry{Trapizoidal};

\addplot [color=black,solid,forget plot,line width=1.1pt]
  table[row sep=crcr]{%
0.023	0.030626\\
0.055	0.030626\\
0.055	0.06500\\
0.023	0.030626\\
};
\node[right, align=left, inner sep=0mm, text=black]
at (axis cs:0.0290,0.024,0) {$1$};
\node[right, align=left, inner sep=0mm, text=black]
at (axis cs:0.06600,.042,0) {$1.04$};
\end{axis}
\end{tikzpicture}%
} \hspace*{-5.0mm}
	\subfigure[]{% This file was created by matlab2tikz.
% Minimal pgfplots version: 1.3
%  http://www.mathworks.com/matlabcentral/fileexchange/22022-matlab2tikz
\begin{tikzpicture}

\begin{axis}[%
width=0.95092\figurewidth,
height=\figureheight,
at={(0\figurewidth,0\figureheight)},
scale only axis,
xmode=log,
xmin=0.001,
xmax=0.9,
xtick={0.005,  0.05,    0.5},
xminorticks=true,
xmajorgrids,
xminorgrids,
xlabel={Maximum edge size $h$},
ymode=log,
ymin=0.0001,
ymax=1,
ytick={ 1e-05, 0.0001,  0.001,   0.01,    0.1,      1},
yminorticks=true,
ymajorgrids,
yminorgrids,
ylabel={$\mathcal{E}_0(\ub)$}, %  $H^1$ Error for $\psi$
legend style={legend cell align=left,align=left,draw=white!15!black},
legend pos = south east
]
\addplot [color=red,dashed,line width=1.1pt,mark=o,mark size=3pt,mark options={solid}]
  table[row sep=crcr]{%
0.1250 0.2123022 \\  
0.0625  0.058146642\\
0.0313 0.015477051\\
0.0156   0.003978042\\
0.0078  0.001022469481\\
};
%\addlegendentry{Square};

%\addplot [color=black,dashed,line width=1.1pt, mark=triangle,mark size=6pt]
%  table[row sep=crcr]{%
%0.235702260395987	0.063616650758363\\
%0.128564869306516	0.036208342254826\\
%0.083189033081103	0.020011507568844\\
%0.061487546190749	0.013450590739499\\
%0.050507627227206	0.010752695027151\\
%};
%\addlegendentry{Distorted square};

\addplot [color=blue,solid,mark=o,mark options={solid},line width=1.1pt,mark size=3pt]
  table[row sep=crcr]{%
0.2000  0.222342960 \\
.1000  0.071002123\\
.0500 0.0200112382\\
0.0250  0.0051434576\\
0.0125 0.0012681615\\	
};
%\addlegendentry{Triangular mesh};

\addplot [color=black,solid,mark=diamond,mark options={solid},line width=1.1pt,mark size=3pt]
  table[row sep=crcr]{%
0.1250 0.25189015\\
0.0625 0.0718367858\\
0.0313 0.0185925261\\
0.0156 0.00474579\\
0.0078 0.00121980530\\
};
%\addlegendentry{Voronoi};

\addplot [color=blue,solid,mark=square,mark options={solid},line width=1.1pt,mark size=3pt]
  table[row sep=crcr]{%
0.1250  0.24171301\\
0.0625 0.066126256\\
0.0313 0.0178413306\\
0.0156 0.00471465\\
0.0078 0.0012117997473 \\
};
%\addlegendentry{Trapizoidal mesh};

\addplot [color=black,solid,forget plot,line width=1.1pt]
  table[row sep=crcr]{%
0.0255	0.00117\\
0.0555	0.00117\\
0.0555	0.00567\\
0.0255	0.00117\\
};
\node[right, align=left, inner sep=0mm, text=black]
at (axis cs:0.0300,0.00080,0) {$1$};
\node[right, align=left, inner sep=0mm, text=black]
at (axis cs:0.0620,.002500,0) {$2.03$};
\end{axis}
\end{tikzpicture}%
} \hspace*{-5.0mm}
	\subfigure[]{% This file was created by matlab2tikz.
% Minimal pgfplots version: 1.3
%  http://www.mathworks.com/matlabcentral/fileexchange/22022-matlab2tikz
\begin{tikzpicture}

\begin{axis}[%
width=0.95092\figurewidth,
height=\figureheight,
at={(0\figurewidth,0\figureheight)},
scale only axis,
xmode=log,
xmin=0.001,
xmax=0.9,
xtick={0.005,  0.05,    0.5},
xminorticks=true,
xmajorgrids,
xminorgrids,
xlabel={Maximum edge size $h$},
ymode=log,
ymin=0.01,
ymax=1,
ytick={ 1e-05, 0.0001,  0.001,   0.01,    0.1,      1},
yminorticks=true,
ymajorgrids,
yminorgrids,
ylabel={$\mathcal{E}_0(\omega)$ },  %$L^2$ Error for $\omega$,
legend style={legend cell align=left,align=left,draw=white!15!black},
legend pos = south east
]
\addplot [color=red,dashed,line width=1.1pt,mark=o,mark size=3pt,mark options={solid}]
  table[row sep=crcr]{%
0.1250 .288901 \\  
0.0625  0.190561\\
0.0313 0.095812\\
0.0156   0.048218\\
0.0078 0.0237192\\
};
%\addlegendentry{Square};

%\addplot [color=black,dashed,line width=1.1pt, mark=triangle,mark size=6pt]
%  table[row sep=crcr]{%
%0.235702260395987	0.063616650758363\\
%0.128564869306516	0.036208342254826\\
%0.083189033081103	0.020011507568844\\
%0.061487546190749	0.013450590739499\\
%0.050507627227206	0.010752695027151\\
%};
%\addlegendentry{Distorted square};

\addplot [color=blue,solid,mark=o,mark options={solid},line width=1.1pt,mark size=3pt]
  table[row sep=crcr]{%
0.2000  0.295354009428 \\
.1000  0.193279765005\\
.0500 0.107963881495\\
0.0250  0.055499578\\
0.0125 0.0283328715161\\	
};
%\addlegendentry{Trianglar mesh};

\addplot [color=black,solid,mark=diamond,mark options={solid},line width=1.1pt,mark size=3pt]
  table[row sep=crcr]{%
0.1250 0.31083155\\
0.0625 0.1712534375\\
0.0313 0.0911385262\\
0.0156 0.0471762597\\
0.0078 .024083766\\
};
%\addlegendentry{Voronoi};

\addplot [color=blue,solid,mark=square,mark options={solid},line width=1.1pt,mark size=3pt]
  table[row sep=crcr]{%
0.1250  0.3191755\\
0.0625 0.17342957\\
0.0313 0.091659092\\
0.0156 0.047445721\\
0.0078 0.0240541976 \\
};
%\addlegendentry{Trapoizoidal mesh};

\addplot [color=black,solid,forget plot,line width=1.1pt]
  table[row sep=crcr]{%
0.020	0.0300\\
0.055	0.0300\\
0.055	0.0800\\
0.020	0.0300\\
};
\node[right, align=left, inner sep=0mm, text=black]
at (axis cs:0.0250,0.025,0) {$1$};
\node[right, align=left, inner sep=0mm, text=black]
at (axis cs:0.060,.0450,0) {$1.10$};
\end{axis}
\end{tikzpicture}%
} 
	\begin{tikzpicture}
		\begin{customlegend}[legend columns=5,legend style={align=center,draw=none,column sep=2ex},
			legend entries={Square, Triangular mesh,  Voronoi, Trapezoidal mesh}]
			\addlegendimage{color=red,dashed,line width=1.1pt,mark=o,mark size=3pt,mark options={solid}}
			%\addlegendimage{color=black,dashed,line width=1.1pt, mark=triangle,mark size=6pt}
			\addlegendimage{color=blue,solid,mark=o,mark options={solid},line width=1.1pt,mark size=3pt}   
			\addlegendimage{color=black,solid,mark=diamond,mark options={solid},line width=1.1pt,mark size=3pt}
			\addlegendimage{color=blue,solid,mark=square,mark options={solid},line width=1.1pt,mark size=3pt}
		\end{customlegend}
	\end{tikzpicture}
\caption{Test 1: Convergence of the velocity field $\ub$ in broken $H^1$- and $L^2$-norms, and vorticity field $\omega$ in $L^2$-norm with mesh refinement for different types of meshes, using $\nu=1$. Left panel shows errors curve for velocity in broken $H^1$-norm, and middle panel shows error curve for velocity in $L^2$-norms, and right panel shows error curves for vorticity in $L^2$-norm.}
\label{fig:Exampele1:velvort:nu1}
\end{figure}
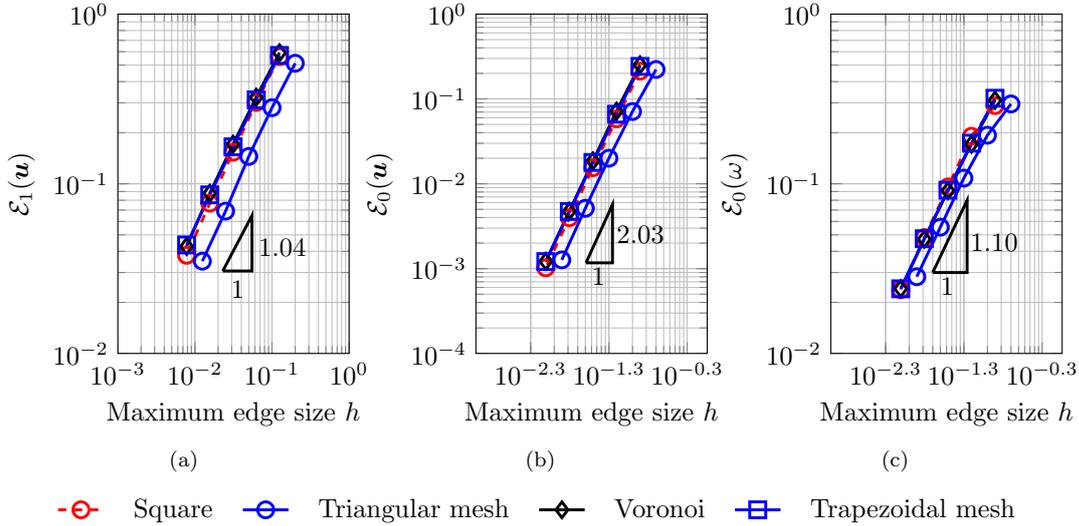 
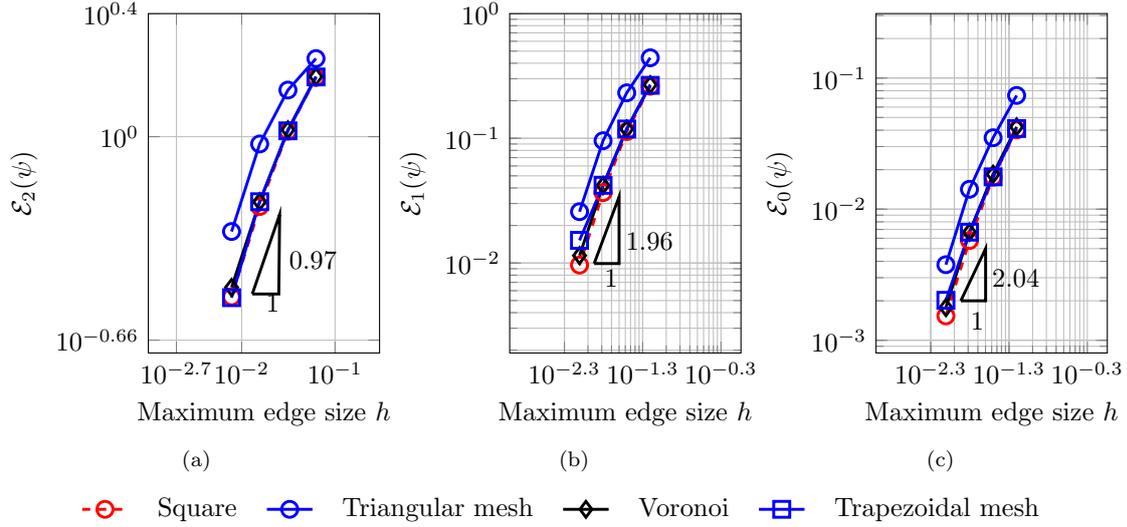
\begin{figure}[h!]
	\centering
	%	\vspace*{-15.5mm}
	\setlength\figureheight{4.5cm} 
	\setlength\figurewidth{3.2cm}
	\subfigure[]{% This file was created by matlab2tikz.
% Minimal pgfplots version: 1.3
%  http://www.mathworks.com/matlabcentral/fileexchange/22022-matlab2tikz
\begin{tikzpicture}
\begin{axis}[%
width=0.95092\figurewidth,
height=\figureheight,
at={(0\figurewidth,0\figureheight)},
scale only axis,
xmode=log,
xmin=0.001,
xmax=.3,
xtick={0.0005,0.002,.01,    .1},
xminorticks=true,
xmajorgrids,
xminorgrids,
xlabel={Maximum edge size $h$},
ymode=log,
ymin=0.20,
ymax=2.5,
ytick={ 1e-05, 0.0001,  0.001,   0.01,    0.22,.1, 1, 2.5},
yminorticks=true,
ymajorgrids,
yminorgrids,
ylabel={$\mathcal{E}_2(\psi)$}, %  $H^2$ Error for $\psi$
legend style={legend cell align=left,align=left,draw=white!15!black}
]
\addplot [color=red,dashed,line width=1.1pt,mark=o,mark size=3pt,mark options={solid}]
  table[row sep=crcr]{%
.0625 1.5539\\ 
.0313  1.04153\\
.0156  .59471660 \\
.0078 0.30360642\\
};
%\addlegendentry{Square};

%\addplot [color=black,dashed,line width=1.1pt, mark=triangle,mark size=6pt]
%  table[row sep=crcr]{%
%0.235702260395987	0.051580195958885\\
%0.128564869306516	0.031608860255526\\
%0.083189033081103	0.019265892106755\\
%0.061487546190749	0.013519908483795\\
%0.050507627227206	0.010518394991391\\
%};
%\addlegendentry{Distorted square};

\addplot [color=blue,solid,mark=o,mark options={solid},line width=1.1pt,mark size=3pt]
  table[row sep=crcr]{%
.0625 1.789138\\
.0313 1.4149\\
.0156  0.947411\\
0.0078 0.493822197\\
};
%\addlegendentry{Triangular mesh};

\addplot [color=black,solid,mark=diamond,mark options={solid},line width=1.1pt,mark size=3pt]
  table[row sep=crcr]{%
0.0625 1.563643\\
0.0313 1.05242326\\
0.0156 0.616484200\\
0.0078 0.32581746\\
};
%\addlegendentry{Voronoi};

\addplot [color=blue,solid,mark=square,mark options={solid},line width=1.1pt,mark size=3pt]
  table[row sep=crcr]{%
0.0625 1.56\\
0.0313 1.044702098\\
0.0156 .616484\\
0.0078 .301801\\
};
%\addlegendentry{Trapizal mesh};

\addplot [color=black,solid,forget plot,line width=1.1pt]
  table[row sep=crcr]{%
0.013 .3102\\
.0250  .3102\\
.0255  .5601\\
0.013	0.3102\\
};
\node[right, align=left, inner sep=0mm, text=black]
at (axis cs:0.0180,0.290,0) {$1$};
\node[right, align=left, inner sep=0mm, text=black]
at (axis cs:0.0320,0.400,0) {$0.97$};
\end{axis}
\end{tikzpicture}%
} \hspace*{-3.5mm}
	\subfigure[]{% This file was created by matlab2tikz.
% Minimal pgfplots version: 1.3
%  http://www.mathworks.com/matlabcentral/fileexchange/22022-matlab2tikz
\begin{tikzpicture}

\begin{axis}[%
width=0.95092\figurewidth,
height=\figureheight,
at={(0\figurewidth,0\figureheight)},
scale only axis,
xmode=log,
xmin=0.001,
xmax=0.9,
xtick={0.005,  0.05,    0.5},
xminorticks=true,
xmajorgrids,
xminorgrids,
xlabel={Maximum edge size $h$},
ymode=log,
ymin=0.0019,
ymax=1,
ytick={ 1e-05, 0.0001,  0.001,   0.01,    0.1,      1},
yminorticks=true,
ymajorgrids,
yminorgrids,
ylabel={$\mathcal{E}_1(\psi)$}, %  $H^1$ Error for $\psi$
legend style={legend cell align=left,align=left,draw=white!15!black},
legend pos = south east
]
\addplot [color=red,dashed,line width=1.1pt,mark=o,mark size=3pt,mark options={solid}]
  table[row sep=crcr]{%
.0625 .261377\\ 
.0313  0.1132248\\
.0156  0.036684 \\
.0078 0.009627\\
};
%\addlegendentry{Square};

%\addplot [color=black,dashed,line width=1.1pt, mark=triangle,mark size=6pt]
%  table[row sep=crcr]{%
%0.235702260395987	0.051580195958885\\
%0.128564869306516	0.031608860255526\\
%0.083189033081103	0.019265892106755\\
%0.061487546190749	0.013519908483795\\
%0.050507627227206	0.010518394991391\\
%};
%\addlegendentry{Distorted square};

\addplot [color=blue,solid,mark=o,mark options={solid},line width=1.1pt,mark size=3pt]
  table[row sep=crcr]{%
.0625	0.44212305\\
.0313 0.231506289\\
.0156 0.0956459\\
.0078  0.0258059\\
};
%\addlegendentry{Triangular mesh};

\addplot [color=black,solid,mark=diamond,mark options={solid},line width=1.1pt,mark size=3pt]
  table[row sep=crcr]{%
0.0625 0.2688127106\\
0.0313 0.118704867\\
0.0156 0.041958\\
0.0078 0.011478593\\
};
%\addlegendentry{Voronoi};

\addplot [color=blue,solid,mark=square,mark options={solid},line width=1.1pt,mark size=3pt]
  table[row sep=crcr]{%
0.0625 0.265949\\
0.0313 0.1187048\\
0.0156 0.041958\\
0.0078 0.0151981\\
};
%\addlegendentry{Trapoizoidal mesh};

\addplot [color=black,solid,forget plot,line width=1.1pt]
  table[row sep=crcr]{%
0.012 .0099\\
.025  .0099\\
.025  .034\\
0.012	0.0099\\
};
\node[right, align=left, inner sep=0mm, text=black]
at (axis cs:0.016,0.0075,0) {$1$};
\node[right, align=left, inner sep=0mm, text=black]
at (axis cs:0.0300,.01500,0) {$1.96$};
\end{axis}
\end{tikzpicture}%
} \hspace*{-3.5mm}
	\subfigure[]{% This file was created by matlab2tikz.
% Minimal pgfplots version: 1.3
%  http://www.mathworks.com/matlabcentral/fileexchange/22022-matlab2tikz
\begin{tikzpicture}

\begin{axis}[%
width=0.95092\figurewidth,
height=\figureheight,
at={(0\figurewidth,0\figureheight)},
scale only axis,
xmode=log,
xmin=0.001,
xmax=0.9,
xtick={0.005,  0.05,    0.5},
xminorticks=true,
xmajorgrids,
xminorgrids,
xlabel={Maximum edge size $h$},
ymode=log,
ymin=0.0008,
ymax=.31,
ytick={ 1e-05, 0.0001,  0.001,   0.01,    0.1,      1},
yminorticks=true,
ymajorgrids,
yminorgrids,
ylabel={$\mathcal{E}_0(\psi)$}, % \mathcal{E}_i(\psi_h)
legend style={legend cell align=left,align=left,draw=white!15!black},
legend pos = south east
]
\addplot [color=red,dashed,line width=1.1pt,mark=o,mark size=3pt,mark options={solid}]
  table[row sep=crcr]{%
0.0625  0.04015\\
0.0313 0.017506\\
0.0156   0.00576744\\
0.0078 0.00153467\\
};
%\addlegendentry{Square};

%\addplot [color=black,dashed,line width=1.1pt, mark=triangle,mark size=6pt]
%  table[row sep=crcr]{%
%0.235702260395987	0.063616650758363\\
%0.128564869306516	0.036208342254826\\
%0.083189033081103	0.020011507568844\\
%0.061487546190749	0.013450590739499\\
%0.050507627227206	0.010752695027151\\
%};
%\addlegendentry{Distorted square};

\addplot [color=blue,solid,mark=o,mark options={solid},line width=1.1pt,mark size=3pt]
  table[row sep=crcr]{%
0.0625  0.07363936 \\
.0313  0.035030122\\
.0156 0.01417123\\
0.0078  0.00377085\\
};
%\addlegendentry{Triangular mesh};

\addplot [color=black,solid,mark=diamond,mark options={solid},line width=1.1pt,mark size=3pt]
  table[row sep=crcr]{%
0.0625 0.042387791\\
0.0313 0.0184680855\\
0.0156 0.006670\\
0.0078 0.00177483\\
};
%\addlegendentry{Voronoi};

\addplot [color=blue,solid,mark=square,mark options={solid},line width=1.1pt,mark size=3pt]
  table[row sep=crcr]{%
0.0625  0.0411737386\\
0.0313 0.0176386282\\
0.0156  0.006670\\
0.0078 0.0020174\\
};
%\addlegendentry{Trapizoidal mesh};

\addplot [color=black,solid,forget plot,line width=1.1pt]
  table[row sep=crcr]{%
0.0121	0.0020\\
0.0251	0.0020\\
0.0251	0.0050\\
0.0121	0.0020\\
};
\node[right, align=left, inner sep=0mm, text=black]
at (axis cs:0.016,0.0014,0) {$1$};
\node[right, align=left, inner sep=0mm, text=black]
at (axis cs:0.0300,.00300,0) {$2.04$};
\end{axis}
\end{tikzpicture}%
} 
	\begin{tikzpicture}
		\begin{customlegend}[legend columns=5,legend style={align=center,draw=none,column sep=2ex},
			legend entries={Square, Triangular mesh,  Voronoi, Trapezoidal mesh}]
			\addlegendimage{color=red,dashed,line width=1.1pt,mark=o,mark size=3pt,mark options={solid}}
			%\addlegendimage{color=black,dashed,line width=1.1pt, mark=triangle,mark size=6pt}
			\addlegendimage{color=blue,solid,mark=o,mark options={solid},line width=1.1pt,mark size=3pt}   
			\addlegendimage{color=black,solid,mark=diamond,mark options={solid},line width=1.1pt,mark size=3pt}
			\addlegendimage{color=blue,solid,mark=square,mark options={solid},line width=1.1pt,mark size=3pt}
		\end{customlegend}
	\end{tikzpicture}
\caption{Test 1: Convergence of the stream-function in broken $H^2$, $H^1$- $L^2$-norms 
with mesh refinement for different types of meshes, using  $\nu=0.01$.}
\label{fig:Exampele1:psi:nu.01}
\end{figure} 

\begin{figure}[h!]
	\centering
	\vspace*{-5.5mm}
	\setlength\figureheight{4.5cm} 
	\setlength\figurewidth{3.2cm}
	\subfigure[]{% This file was created by matlab2tikz.
% Minimal pgfplots version: 1.3
%  http://www.mathworks.com/matlabcentral/fileexchange/22022-matlab2tikz
\begin{tikzpicture}
\begin{axis}[%
width=0.95092\figurewidth,
height=\figureheight,
at={(0\figurewidth,0\figureheight)},
scale only axis,
xmode=log,
xmin=0.005,
xmax=.25,
xtick={0.005,  0.05,    .5},
xminorticks=true,
xmajorgrids,
xminorgrids,
xlabel={Maximum edge size $h$},
ymode=log,
ymin=0.018,
ymax=1,
ytick={ 1e-05, 0.0001,  0.001,   0.01,    0.1,      1},
yminorticks=true,
ymajorgrids,
yminorgrids,
ylabel={$\mathcal{E}_1(\ub)$}, %  $H^1$ Error for $\bu$
legend style={legend cell align=left,align=left,draw=white!15!black}
]
\addplot [color=red,dashed,line width=1.1pt,mark=o,mark size=3pt,mark options={solid}]
  table[row sep=crcr]{%
.0625 .656980005\\ 
.0313  0.32365678\\
.0156  0.1563155 \\
.0078 0.07924881\\
};
%\addlegendentry{Square};

%\addplot [color=black,dashed,line width=1.1pt, mark=triangle,mark size=6pt]
%  table[row sep=crcr]{%
%0.235702260395987	0.051580195958885\\
%0.128564869306516	0.031608860255526\\
%0.083189033081103	0.019265892106755\\
%0.061487546190749	0.013519908483795\\
%0.050507627227206	0.010518394991391\\
%};
%\addlegendentry{Distorted square};

\addplot [color=blue,solid,mark=o,mark options={solid},line width=1.1pt,mark size=3pt]
  table[row sep=crcr]{%
.0625	.733781\\
.0313  .40086\\
.0156  .21607711\\
.0078  .09529671\\
};
%\addlegendentry{Triangular mesh};

\addplot [color=black,solid,mark=diamond,mark options={solid},line width=1.1pt,mark size=3pt]
  table[row sep=crcr]{%
0.0625 0.6806647\\
0.0313 0.35334043\\
0.0156 0.143425\\
0.0078 0.07632864\\
};
%\addlegendentry{Voronoi};

\addplot [color=blue,solid,mark=square,mark options={solid},line width=1.1pt,mark size=3pt]
  table[row sep=crcr]{%
0.0625 0.676540\\
0.0313 0.33514767\\
0.0156 0.122561\\
0.0078 0.055131\\
};
%\addlegendentry{Trapoizoidal mesh};

\addplot [color=black,solid,forget plot,line width=1.1pt]
  table[row sep=crcr]{%
0.0125 .060\\
.0255  .060\\
.0255 .124\\
0.0125	0.060\\
};
\node[right, align=left, inner sep=0mm, text=black]
at (axis cs:0.0200,0.05,0) {$1$};
\node[right, align=left, inner sep=0mm, text=black]
at (axis cs:0.0300,0.090,0) {$0.95$};
\end{axis}
\end{tikzpicture}%
} \hspace*{-3.5mm}
	\subfigure[]{% This file was created by matlab2tikz.
% Minimal pgfplots version: 1.3
%  http://www.mathworks.com/matlabcentral/fileexchange/22022-matlab2tikz
\begin{tikzpicture}

\begin{axis}[%
width=0.95092\figurewidth,
height=\figureheight,
at={(0\figurewidth,0\figureheight)},
scale only axis,
xmode=log,
xmin=0.001,
xmax=0.9,
xtick={0.005,  0.05,    0.5},
xminorticks=true,
xmajorgrids,
xminorgrids,
xlabel={Maximum edge size $h$},
ymode=log,
ymin=0.002,
ymax=1,
ytick={ 1e-05, 0.0001,  0.001,   0.01,    0.1,      1},
yminorticks=true,
ymajorgrids,
yminorgrids,
ylabel={$\mathcal{E}_0(\ub)$}, % $L^2$ Error for $\bu$
legend style={legend cell align=left,align=left,draw=white!15!black},
legend pos = south east
]
\addplot [color=red,dashed,line width=1.1pt,mark=o,mark size=3pt,mark options={solid}]
  table[row sep=crcr]{%  
0.0625  0.26137759\\
0.0313 0.1132248\\
0.0156   .038400042\\
0.0078 .01007729\\
};
%\addlegendentry{Square};

%\addplot [color=black,dashed,line width=1.1pt, mark=triangle,mark size=6pt]
%  table[row sep=crcr]{%
%0.235702260395987	0.063616650758363\\
%0.128564869306516	0.036208342254826\\
%0.083189033081103	0.020011507568844\\
%0.061487546190749	0.013450590739499\\
%0.050507627227206	0.010752695027151\\
%};
%\addlegendentry{Distorted square};

\addplot [color=blue,solid,mark=o,mark options={solid},line width=1.1pt,mark size=3pt]
  table[row sep=crcr]{%
0.0625  0.44212305 \\
.0313  0.2315062\\
.0156 0.09564595\\
0.0078  .02545067\\	
};
%\addlegendentry{Triangular mesh};

\addplot [color=black,solid,mark=diamond,mark options={solid},line width=1.1pt,mark size=3pt]
  table[row sep=crcr]{%
0.0625 0.2688127\\
0.0313 0.11870486\\
0.0156 0.04195851\\
0.0078 0.011087725\\
};
%\addlegendentry{Voronoi};

\addplot [color=blue,solid,mark=square,mark options={solid},line width=1.1pt,mark size=3pt]
  table[row sep=crcr]{%
0.0625 0.265949687\\
0.0313 0.11419547\\
0.0156 0.031290\\
0.0078 0.0099321 \\
};
%\addlegendentry{non-convex mesh};

\addplot [color=black,solid,forget plot,line width=1.1pt]
  table[row sep=crcr]{%
0.0155	0.009\\
0.0455	0.009\\
0.0455	0.03\\
0.0155	0.009\\
};
\node[right, align=left, inner sep=0mm, text=black]
at (axis cs:0.0250,0.006,0) {$1$};
\node[right, align=left, inner sep=0mm, text=black]
at (axis cs:0.0600,.01500,0) {$1.96$};
\end{axis}
\end{tikzpicture}%
} \hspace*{-3.5mm}
	\subfigure[]{% This file was created by matlab2tikz.
% Minimal pgfplots version: 1.3
%  http://www.mathworks.com/matlabcentral/fileexchange/22022-matlab2tikz
\begin{tikzpicture}

\begin{axis}[%
width=0.95092\figurewidth,
height=\figureheight,
at={(0\figurewidth,0\figureheight)},
scale only axis,
xmode=log,
xmin=0.001,
xmax=0.9,
xtick={0.005,  0.05,    0.5},
xminorticks=true,
xmajorgrids,
xminorgrids,
xlabel={Maximum edge size $h$},
ymode=log,
ymin=0.05,
ymax=1.5,
ytick={ 1e-05, 0.0001,  0.001,   0.01,    0.1,  .5,  1,   2},
yminorticks=true,
ymajorgrids,
yminorgrids,
ylabel={$\mathcal{E}_0(\omega)$ },  %$L^2$ Error for $\omega$,
legend style={legend cell align=left,align=left,draw=white!15!black},
legend pos = south east
]
\addplot [color=red,dashed,line width=1.1pt,mark=o,mark size=3pt,mark options={solid}]
  table[row sep=crcr]{%  
0.0625  0.58893225\\
0.0313 0.31584914\\
0.0156   0.16692911\\
0.0078 0.08521832\\
};
%\addlegendentry{Square};

%\addplot [color=black,dashed,line width=1.1pt, mark=triangle,mark size=6pt]
%  table[row sep=crcr]{%
%0.235702260395987	0.063616650758363\\
%0.128564869306516	0.036208342254826\\
%0.083189033081103	0.020011507568844\\
%0.061487546190749	0.013450590739499\\
%0.050507627227206	0.010752695027151\\
%};
%\addlegendentry{Distorted square};

\addplot [color=blue,solid,mark=o,mark options={solid},line width=1.1pt,mark size=3pt]
  table[row sep=crcr]{%
0.0625  0.946171889 \\
.0313  0.73812342\\
.0156  0.45072455\\
0.0078  0.238211990\\	
};
%\addlegendentry{Triangular mesh};

\addplot [color=black,solid,mark=diamond,mark options={solid},line width=1.1pt,mark size=3pt]
  table[row sep=crcr]{%
0.0625 0.634874\\
0.0313 0.3370085\\
0.0156 0.1305656\\
0.0078 .0680550\\
};
%\addlegendentry{Voronoi mesh};

\addplot [color=blue,solid,mark=square,mark options={solid},line width=1.1pt,mark size=3pt]
  table[row sep=crcr]{%
0.0625 0.59669750\\
0.0313 0.31771191\\
0.0156 0.191549\\
0.0078 0.0791224 \\
};
%\addlegendentry{Trapizoidal mesh};

\addplot [color=black,solid,forget plot,line width=1.1pt]
  table[row sep=crcr]{%
0.0155	0.0700\\
0.0355	0.0700\\
0.0355	0.1500\\
0.0155	0.0700\\
};
\node[right, align=left, inner sep=0mm, text=black]
at (axis cs:0.0200,0.0603,0) {$1$};
\node[right, align=left, inner sep=0mm, text=black]
at (axis cs:0.0450,.100,0) {$0.98$};
\end{axis}
\end{tikzpicture}%
} 
	\begin{tikzpicture}
		\begin{customlegend}[legend columns=5,legend style={align=center,draw=none,column sep=2ex},
			legend entries={Square, Triangular mesh,  Voronoi, Trapezoidal mesh}]
			\addlegendimage{color=red,dashed,line width=1.1pt,mark=o,mark size=3pt,mark options={solid}}
			%\addlegendimage{color=black,dashed,line width=1.1pt, mark=triangle,mark size=6pt}
			\addlegendimage{color=blue,solid,mark=o,mark options={solid},line width=1.1pt,mark size=3pt}   
			\addlegendimage{color=black,solid,mark=diamond,mark options={solid},line width=1.1pt,mark size=3pt}
			\addlegendimage{color=blue,solid,mark=square,mark options={solid},line width=1.1pt,mark size=3pt}
		\end{customlegend}
	\end{tikzpicture}
	\caption{Test 1: Convergence of the velocity field $\ub$ in broken $H^1$- and $L^2$-norms, and vorticity field $\omega$ in $L^2$-norm with mesh refinement for different types  of meshes, using  $\nu=0.01$. Left panel shows error curves for velocity in discrete $H^1$-norm, and middle panel shows error curves for velocity in $L^2$-norms, and right panel shows error curves for vorticity in $L^2$-norm.}
	\label{fig:Exampele1:velvort:nu0.01}
\end{figure}
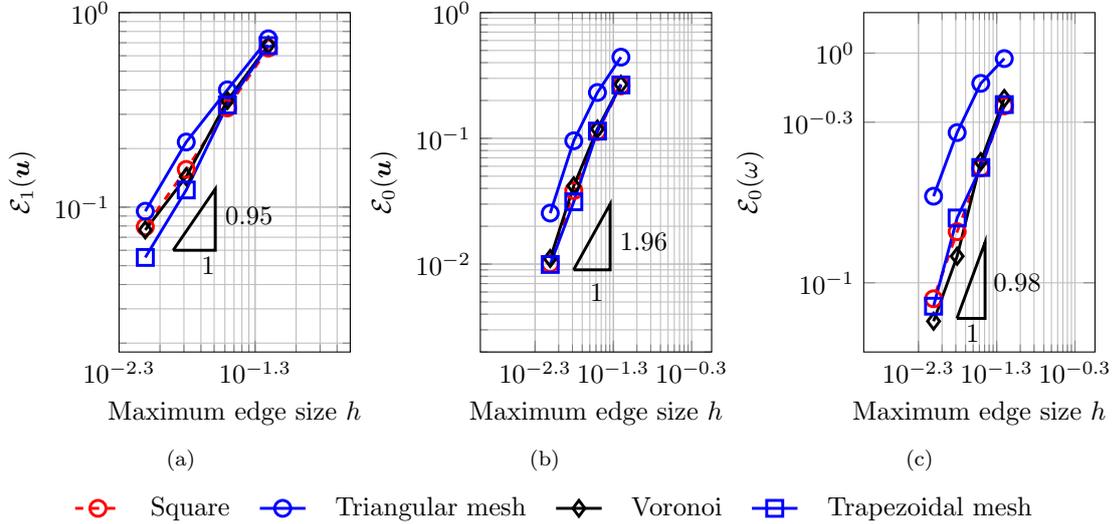

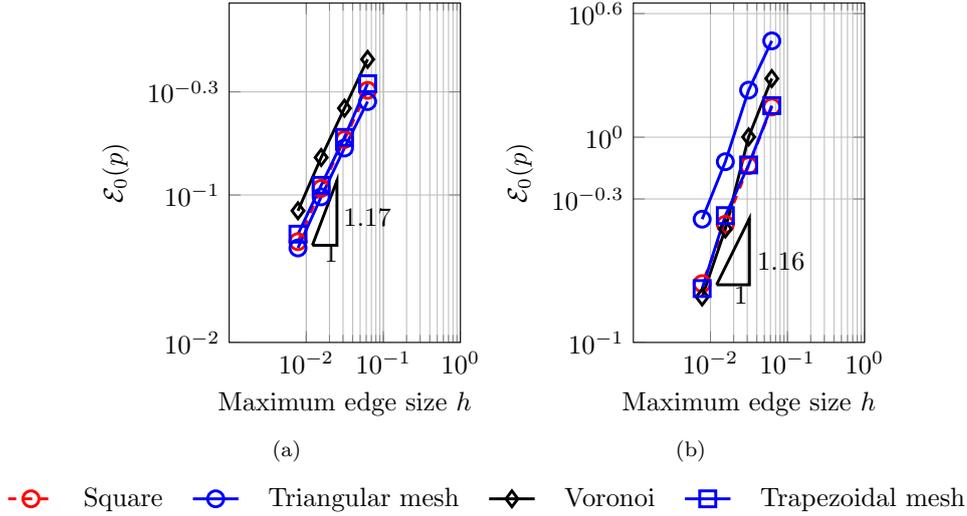
\begin{figure}[h!]%[htpb!]
	\centering
	\setlength\figureheight{4.5cm} 
	\setlength\figurewidth{3.2cm}
	\subfigure[]{% This file was created by matlab2tikz.
% Minimal pgfplots version: 1.3
%  http://www.mathworks.com/matlabcentral/fileexchange/22022-matlab2tikz
\begin{tikzpicture}

\begin{axis}[%
width=0.95092\figurewidth,
height=\figureheight,
at={(0\figurewidth,0\figureheight)},
scale only axis,
xmode=log,
xmin=0.001,
xmax=1,
xtick={0.01,  0.1,    1},
xminorticks=true,
xmajorgrids,
xminorgrids,
xlabel={Maximum edge size $h$},
ymode=log,
ymin=1e-02,
ymax=2,
ytick={ 1e-05, 0.0001,  0.001,   0.01,    0.1,.5},
yminorticks=true,
ymajorgrids,
yminorgrids,
ylabel={ $\mathcal{E}_0(p)$},
legend style={legend cell align=left,align=left,draw=white!15!black},
legend pos=south east
]
\addplot [color=red,dashed,line width=1.1pt,mark=o,mark size=3pt,mark options={solid}]
  table[row sep=crcr]{%
.0625	0.512154\\
.0313   0.237739\\
.0156   0.110835\\
.0078   0.048243735\\
};
%\addlegendentry{Square};

\addplot [color=blue,solid,mark=o,mark options={solid},line width=1.1pt,mark size=3pt]
  table[row sep=crcr]{%
.0625  0.42967649\\ 
.0313  0.208096655\\
.0156  0.097080522\\
.0078  .043746932 \\
};
%\addlegendentry{Triangular mesh};

\addplot [color=black,solid,mark=diamond,mark options={solid},line width=1.1pt,mark size=3pt]
  table[row sep=crcr]{%
.0625  .83077136\\ 
.0313  0.38756854\\
.0156  0.179558193\\
.0078  .0781572433 \\
};
%\addlegendentry{ Voronoi};

\addplot [color=blue,solid,mark=square,mark options={solid},line width=1.1pt,mark size=3pt]
table[row sep=crcr]{%
0.0625 0.568612435\\
0.0313 0.245793310\\
0.0156 0.116267320\\
0.0078 0.054240622 \\ 
};
%\addlegendentry{Trapizoidal};

\addplot [color=black,solid,forget plot,line width=1.1pt]
  table[row sep=crcr]{%
0.01200	  .0456\\
0.0250	  .0456\\
0.0250 	 .125\\
0.01200 .0456\\
};
\node[right, align=left, inner sep=0mm, text=black]
at (axis cs:0.017,.04000,0) {$1$};
\node[right, align=left, inner sep=0mm, text=black]
at (axis cs:0.030,0.07,0) {$1.17$};
\end{axis}
\end{tikzpicture}%
} %\hspace*{-4.5mm}
	\subfigure[]{% This file was created by matlab2tikz.
% Minimal pgfplots version: 1.3
%  http://www.mathworks.com/matlabcentral/fileexchange/22022-matlab2tikz
\begin{tikzpicture}

\begin{axis}[%
width=0.95092\figurewidth,
height=\figureheight,
at={(0\figurewidth,0\figureheight)},
scale only axis,
xmode=log,
xmin=0.001,
xmax=1,
xtick={0.01,  0.1,    1},
xminorticks=true,
xmajorgrids,
xminorgrids,
xlabel={Maximum edge size $h$},
ymode=log,
ymin=1e-01,
ymax=4.5,
ytick={ 1e-05, 0.0001,  0.001,   0.01,    0.1,.5,1,4},
yminorticks=true,
ymajorgrids,
yminorgrids,
ylabel={ $\mathcal{E}_0(p)$},
legend style={legend cell align=left,align=left,draw=white!15!black},
legend pos=south east
]
\addplot [color=red,dashed,line width=1.1pt,mark=o,mark size=3pt,mark options={solid}]
  table[row sep=crcr]{%
.0625	1.403266\\
.0313   0.72712134\\
.0156   0.376381433\\
.0078   0.19348147\\
};
%\addlegendentry{Square};

\addplot [color=blue,solid,mark=o,mark options={solid},line width=1.1pt,mark size=3pt]
  table[row sep=crcr]{%
.0625  2.945075\\ 
.0313  1.69440\\
.0156  0.759657\\
.0078  .398712313 \\
};
%\addlegendentry{Triangular mesh};

\addplot [color=black,solid,mark=diamond,mark options={solid},line width=1.1pt,mark size=3pt]
  table[row sep=crcr]{%
.0625  1.92943\\ 
.0313  1.002329\\
.0156  0.35707989\\
.0078  .16658365 \\
};
%\addlegendentry{ Voronoi};

\addplot [color=blue,solid,mark=square,mark options={solid},line width=1.1pt,mark size=3pt]
table[row sep=crcr]{%
0.0625 1.424056\\
0.0313 0.732697\\
0.0156 0.41503070\\
0.0078 0.183124 \\ 
};
%\addlegendentry{Trapizoidal};

\addplot [color=black,solid,forget plot,line width=1.1pt]
  table[row sep=crcr]{%
0.012	  .191\\
0.032	  .191\\
0.032 	 .401\\
0.012	 .191\\
};
\node[right, align=left, inner sep=0mm, text=black]
at (axis cs:0.020,.17,0) {$1$};
\node[right, align=left, inner sep=0mm, text=black]
at (axis cs:0.04,0.25,0) {$1.16$};
\end{axis}
\end{tikzpicture}%
}
	\begin{tikzpicture}
		\begin{customlegend}[legend columns=5,legend style={align=center,draw=none,column sep=2ex},
			legend entries={Square, Triangular mesh,  Voronoi, Trapezoidal mesh}]
			\addlegendimage{color=red,dashed,line width=1.1pt,mark=o,mark size=3pt,mark options={solid}}
				\addlegendimage{color=blue,solid,mark=o,mark options={solid},line width=1.1pt,mark size=3pt}   
			\addlegendimage{color=black,solid,mark=diamond,mark options={solid},line width=1.1pt,mark size=3pt}
			\addlegendimage{color=blue,solid,mark=square,mark options={solid},line width=1.1pt,mark size=3pt}
		\end{customlegend}
	\end{tikzpicture}
\caption{Test 1: Convergence of the pressure ($p$) in $L^2$-norm with mesh refinement for different types of meshes, using   $\nu=1$ and $\nu=0.01$. Left panel shows the errors curve of $p$ for $\nu=1$, and right panel shows the errors curve of $p$ for $\nu=0.01$.  }
\label{fig:Example1:p}
\end{figure}

\begin{figure}[h!]
	\centering
	\overfullrule = 0pt
	\begin{tabular}{cc}
		\includegraphics[scale=0.056]{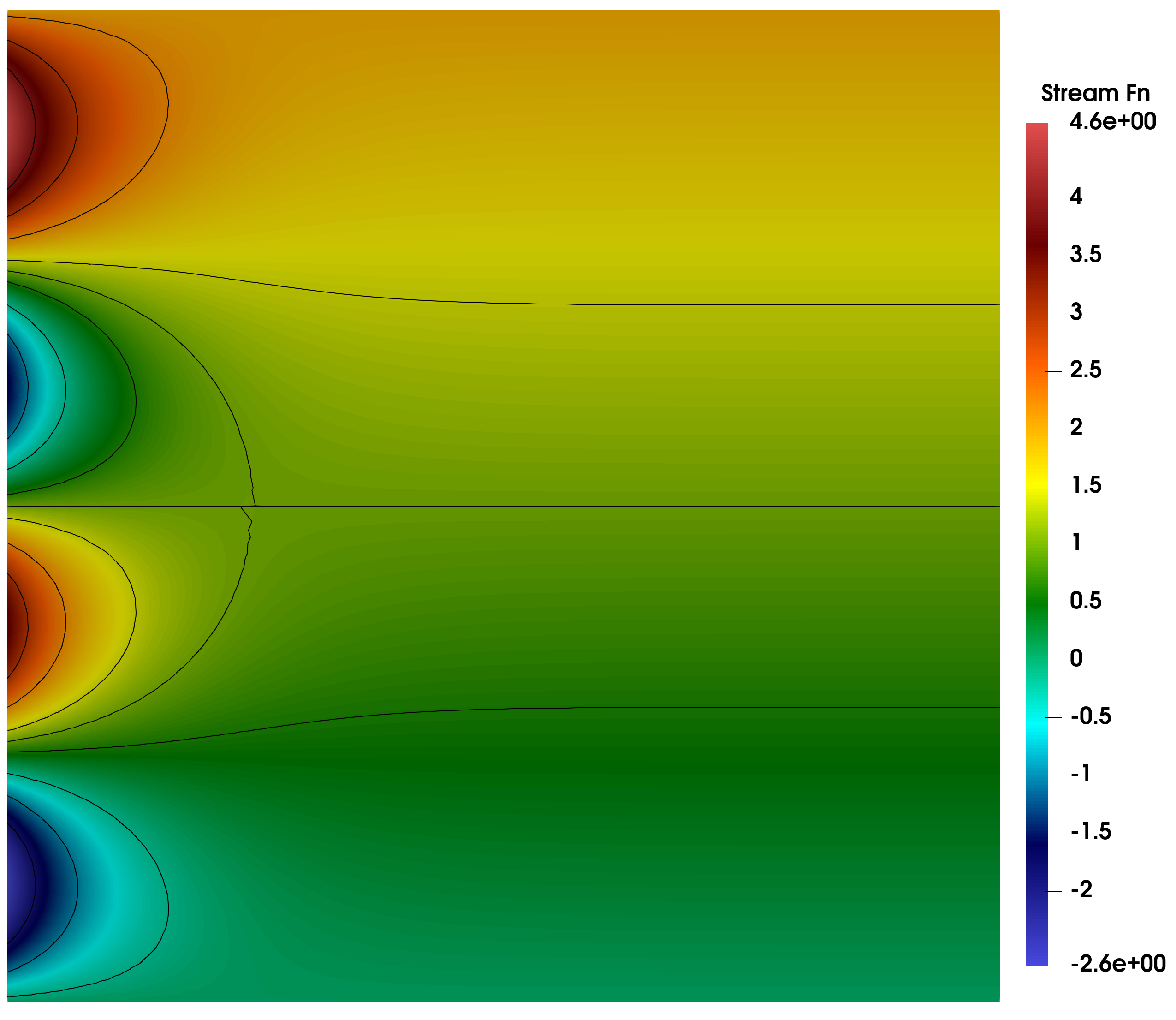}    & \hspace*{5.0mm}
			\includegraphics[scale=0.056]{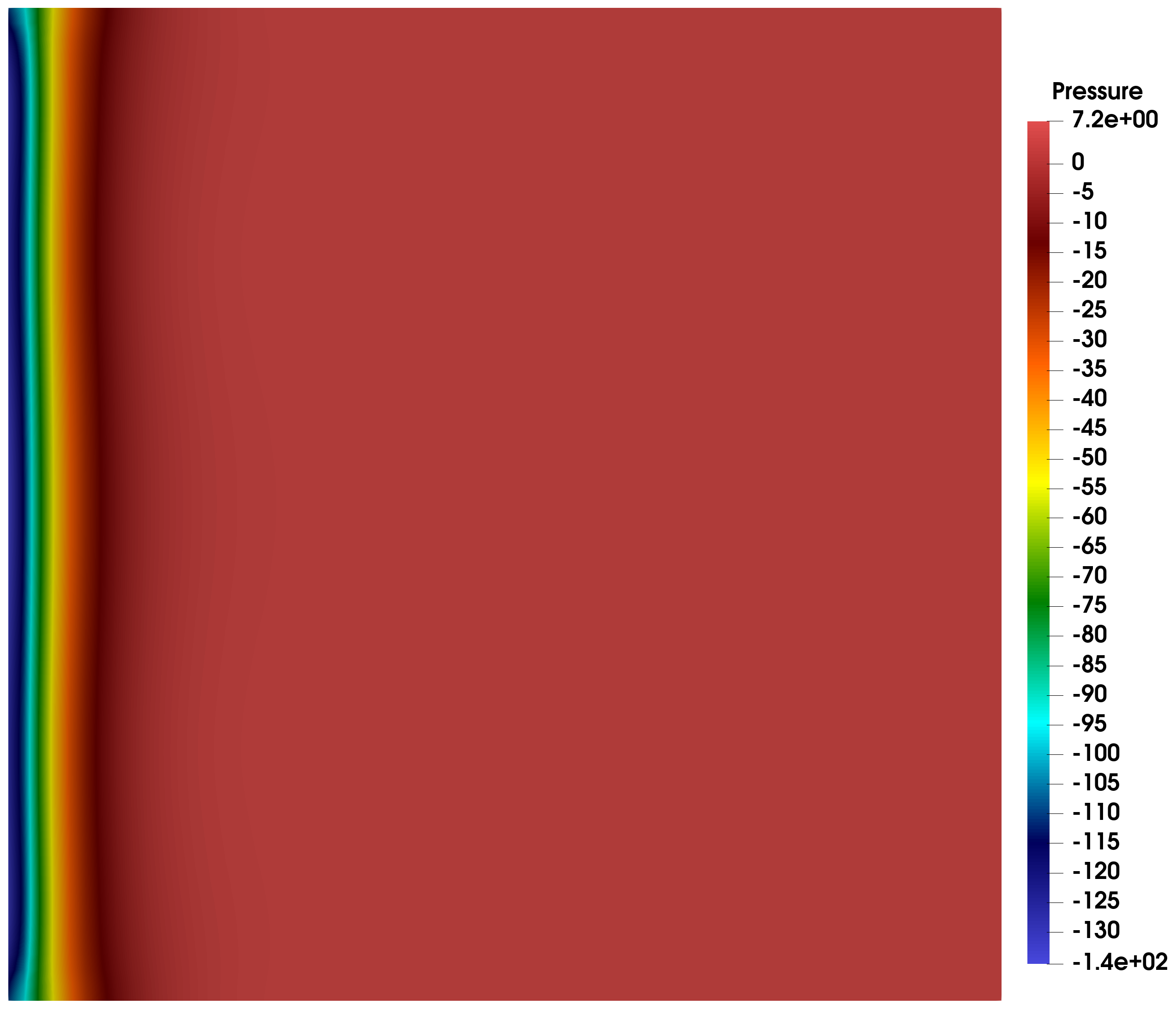}  \\[0.5em]
		$(a)$ Discrete stream-function &
		$(b)$ Discrete pressure  \\[1em]
	\end{tabular}
	\caption{Test 1: ``Snapshots'' of the approximate stream-function and pressure, using $\nu=1$ and the mesh $\CT^1_h$, $h=1/32$.}
	\vspace{-0.75\baselineskip}
	\label{example1_psi_Vel_Vort_p}
\end{figure}

\subsection{Test 2. L-shaped  domain}

In this example, we would like to focus to examine the rate of convergences of the discrete stream-function,
 velocity and vorticity fields on a nonconvex $L$-shaped domain, where the exact solution $\psi$ has less regularity. 
 For the computational domain, we considered $\overline{\Omega}=[-1,1] \times [-1,1] \setminus  (0,1) \times (-1,0)$. The exact solution is given by $\psi(r,\theta)=r^{5/3} \sin(\frac{5 \theta}{3})$, where $r=(x^2+y^2)^{1/2}$, and $\theta$ is the angle with the vertical axis. Since $\frac{\partial \psi}{\partial r}$ is unbounded near the origin, then the solution $\psi$ has weak regularity near the origin. 
 The rate of convergence of  stream-function velocity and vorticity solutions are posted in Table~\ref{table-stream-lshp}  for viscosity  $\nu=1$, and using the mesh $\CT^2_h$. From the posted results, we observed that the rates of convergence are in accordance with the theoretical prediction for all the variables. Further, we have chosen exact pressure as $p:=\sin(x)-\sin(y)-\overline{p}$, where $\overline{p}$ is a constant that is set to satisfy zero mean condition, i.e., $ (p,1)_{0,\Omega}=0$. The convergence behavior of the pressure field is posted in Table~\ref{table-pressure-lshp}. It is observed that initially the rate of convergence is slightly higher than the predicted order as in Theorem~\ref{lemma:Pressure:VC:bound}. However, for finer mesh we observe expected order of convergence, i.e., $\mathcal{O}(h^{2/3})$. Further, we report that the presence of singularity of the stream-function  at re-entrant corner affects the convergence order of pressure field as proven in Theorem~\ref{lemma:Pressure:VC:bound}.  In this example, the number of iterations that are required for the Newton method is $4$.
\begin{table}[h!]
 \setlength{\tabcolsep}{2.0pt}
\begin{center}
{\small \begin{tabular}{rrcccccccccccccccccccc}
\hline
\hline\noalign{\smallskip}
&$h$ &$\calE_2(\psi)$ &$\calR_2(\psi)$ & $\calE_1(\psi)$ & $\calR_1(\psi)$& $\calE_0(\psi)$ &$\calR_0(\psi)$& $\calE_1(\ub)$ &$\calR_1(\ub)$ & $\calE_0(\ub)$ & $\calR_0(\ub)$&  $\calE_0(\omega)$ & $\calR_0(\omega)$ \\
\hline  
\hline
&1/4 &5.7631e-2&--- &7.6316e-3&--- &3.3797e-3 &--- &1.0336e-1 &--- &7.5336e-3 &--- &6.1773e-2 &---\\
&1/8 &3.8328e-2&0.59&2.9766e-3&1.34&1.2923e-3 &1.38&6.7243e-2 &0.62 &2.8964e-3&1.37&4.2442e-2 &0.54 \\
&1/16 &2.4854e-2&0.62&1.1634e-3&1.35&5.5365e-4 &1.22&4.3160e-2 &0.64 &1.1236e-3&1.36&2.7923e-2 &0.60\\
&1/32&1.5907e-2&0.64&4.6577e-4&1.32&2.3946e-4 &1.20&2.7492e-2 &0.65 &4.5976e-4&1.29&1.7999e-2 &0.63\\
&1/64&1.0032e-2&0.66&1.9139e-4&1.28&1.0326e-4 &1.21&1.7435e-2 &0.66 &1.8729e-4&1.29&1.1483e-2&0.65\\  
\hline
\hline
\end{tabular}}
\end{center}
\caption{Test~2. Errors for the stream-function, and the post-processed velocity, vorticity fields in  broken $H^2$-, $H^1$- and $L^2$-norms for $\nu=1$, using the mesh $\CT^2_h$.}
\label{table-stream-lshp}
	\end{table}
	\begin{table}[h!]
 \setlength{\tabcolsep}{2.2pt}
\begin{center}
{\small \begin{tabular}{rrcccccccccccccccccccc}
\hline
\hline\noalign{\smallskip}
&$h$ &$\calE_0(p)$ &$\calR_0(p)$ \\
\hline  
\hline
&1/4 &3.3613e-1&--- \\
&1/8 &1.7549e-1&0.93 \\
&1/16 &9.3685e-2&0.90\\
&1/32&5.2943e-2&0.82\\
&1/64&3.1274e-2&0.75\\  
\hline
\hline
\end{tabular}}
\end{center}
\caption{Test~2. Errors for the pressure variable in $L^2$-norm for $\nu=1$, using the mesh $\CT^2_h$.}
		\label{table-pressure-lshp}
	\end{table}

\subsection{Test 3. The lid-driven cavity problem}
In the third example, we assess the nature of the fluid for the lid-driven cavity flow. This is a benchmark test to validate the numerical schemes for different values of viscosity $\nu$. The computational domain is unit square with upper horizontal lid is moving with uniform velocity $\ub:=(1,0)$, and fixed boundary condition, i.e., $\ub:=(0,0)$ is applied to other static walls. In stream-function formulation, we have imposed the following Dirichlet boundary conditions: $\psi=\psi_x=0$, and $\psi_y=1$ on moving lid, and $\psi=\frac{\partial\psi}{\partial \nb}=0$ on all other static walls. In Figure~\ref{fig:LidDriven}, we posted the discrete stream-function and pressure field for $\nu=0.01$ and using the mesh $\CT^3_h$, with $h=1/64$. The small values of $\nu$ exhibits singularities near $x=0$, and $x=1$ \cite{GGS82,MSB2016}, which increases for smaller values of $\nu$. Such behaviors are noticed in other methods \cite{GGS82}, and persists also for finer grid. Further, we observed that the vortex center has moved towards the direction of velocity  for small values of $\nu$. Such characteristic of fluids with small viscosity coefficient are well observed in literature. Additionally, we report that our scheme preserves the property of the fluids with low viscosity coefficients on general shaped polygonal meshes.
For this numerical experiment, the number of iterations that are required for the Newton method is $5$.

\begin{figure}[!h]
\centering 
\subfigure[Discrete stream-function]{\includegraphics[scale=0.055]{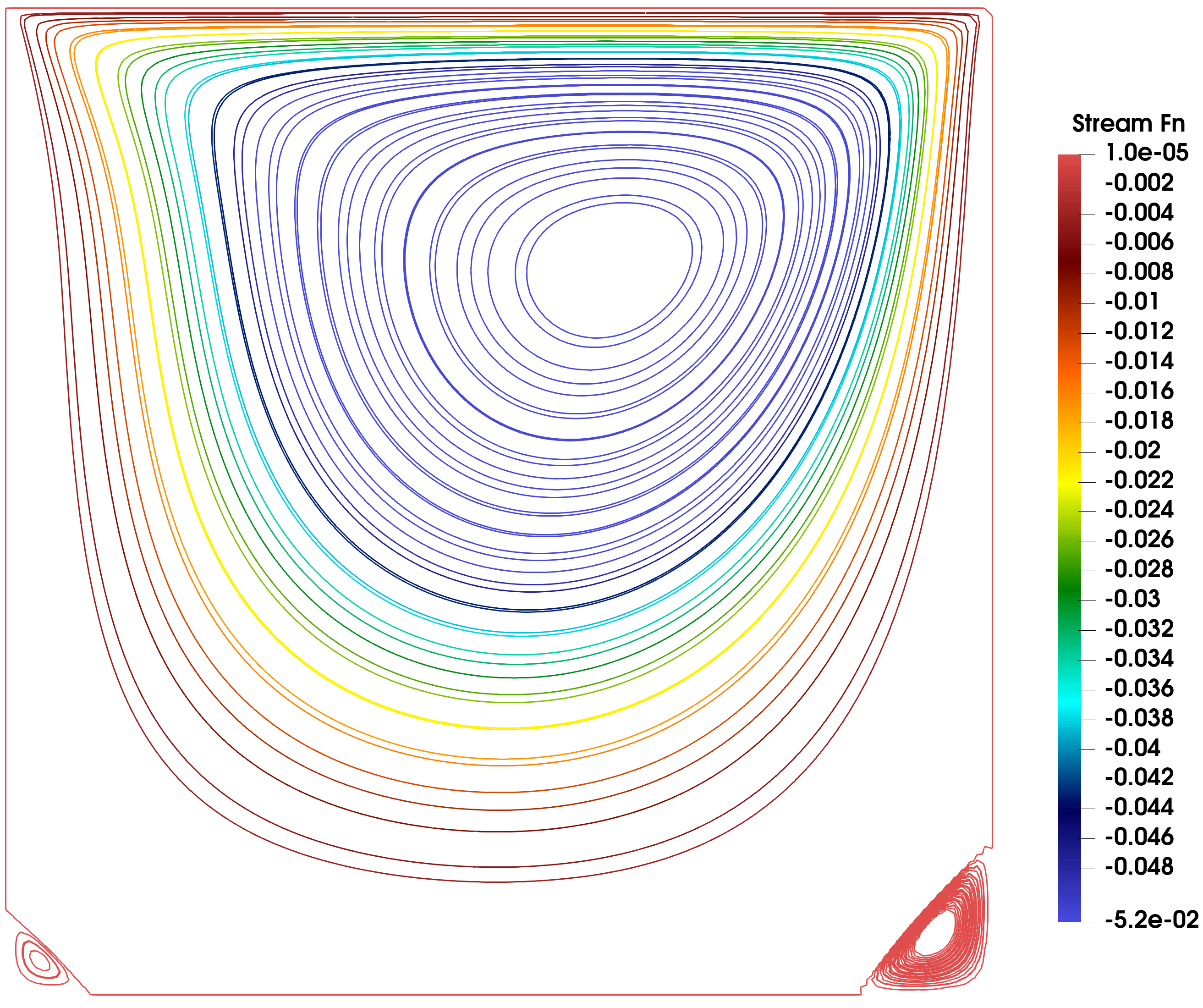}}\hspace{0.3 cm}
\subfigure[Discrete pressure]{\includegraphics[scale=0.055]{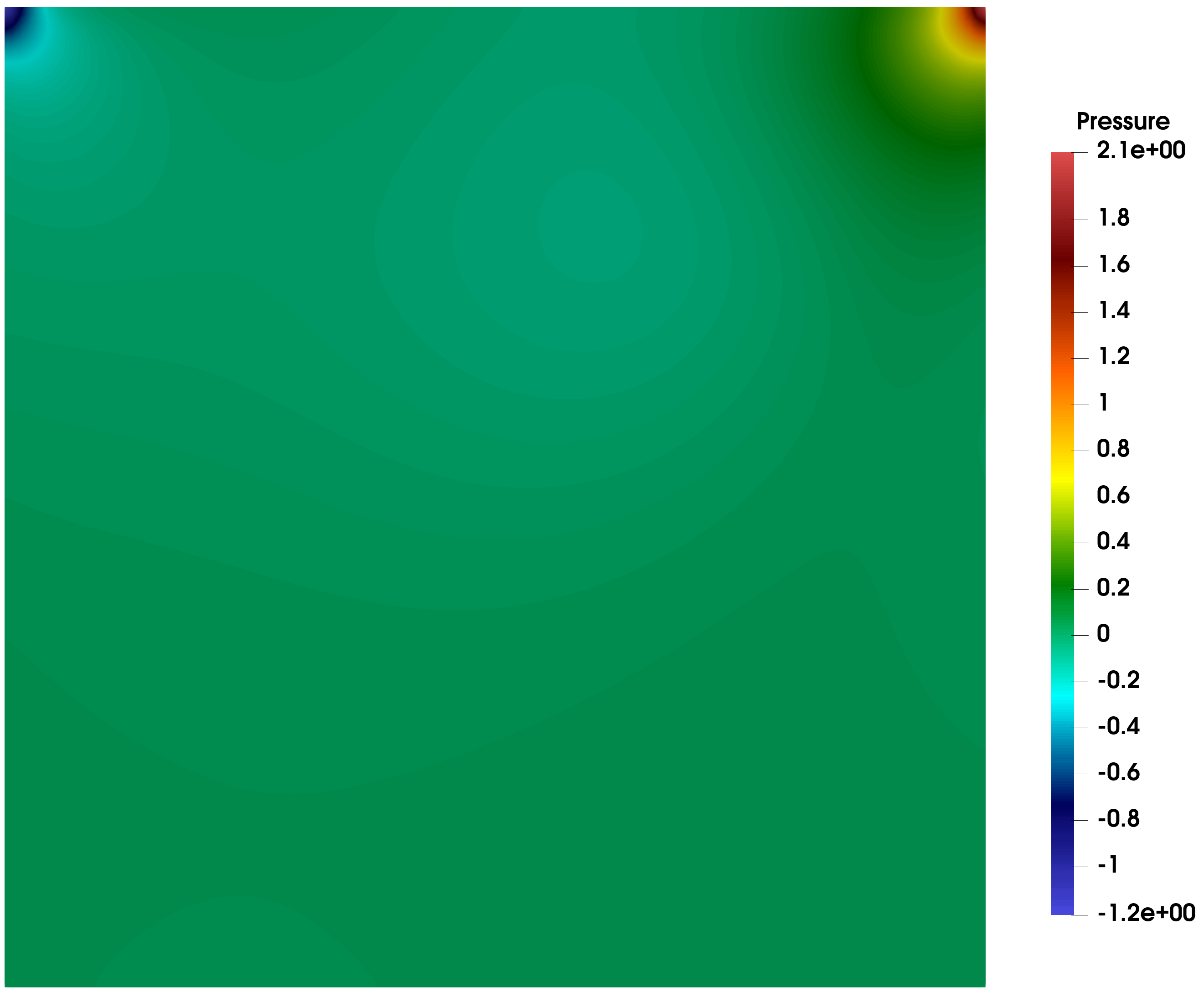}}
\caption{ Test 3: ``Snapshots'' of the approximate stream-function 
    and pressure for $\nu=0.01$, using the mesh $\CT^3_h$, with $h=1/64$.}
\label{fig:LidDriven}
\end{figure}

\subsection{Test 4. Performance of the scheme for small viscosity}\label{test:small:nu}
In this example, we mainly focus to discuss the performance  of the scheme for  small values of viscosity coefficients. We consider the exact   stream-function, velocity and pressure solutions as
\begin{equation*}
	\psi(x,y) = x^2y^2(1-x)^2(1-y)^2, \quad
	\bu(x,y):=\begin{pmatrix} \: x^2(1-x)^2(2y-6y^2+4y^3) \\
		 -y^2(1-y)^2(2x-6x^2+4x^3)
	\end{pmatrix}, \quad p(x,y):=x^3y^3-\frac{1}{6}.
\end{equation*}

The numerical approximations of the stream-functions are computed by employing the scheme \eqref{NSE:stream:disc}, with the alternative  load term given by \eqref{load-alter-global}. The computational domain is considered as $\overline{\Omega}:=[0,1]\times[0,1]$. Further, we discretize the domain with square elements with different mesh sizes, 
and computed the errors for stream-function in broken $H^2$-norm for different values of $\nu$, which are  posted in Figure~\ref{fig:robust}. We observed that the errors are accurate when the parameter $\nu$ within the range $\nu \in [10^{-3},10^{0}]$ and the errors increase for $\nu =10^{-4}$. 
We claim that these results are in accordance with the general behaviour of the exactly divergence-free Galerkin schemes are more robust with respect to small viscosity parameters, see for instance \cite{BLV-NS18} in the VEM approach. Finally, we report that  the maximum number of iterations that are required to achieve the tolerance 
in the Newton method is $7$. %, for each mesh size and for all values of $\nu$.

\begin{figure}[h!]
\begin{center}
{\includegraphics[height=4.3cm, width=8.7cm]{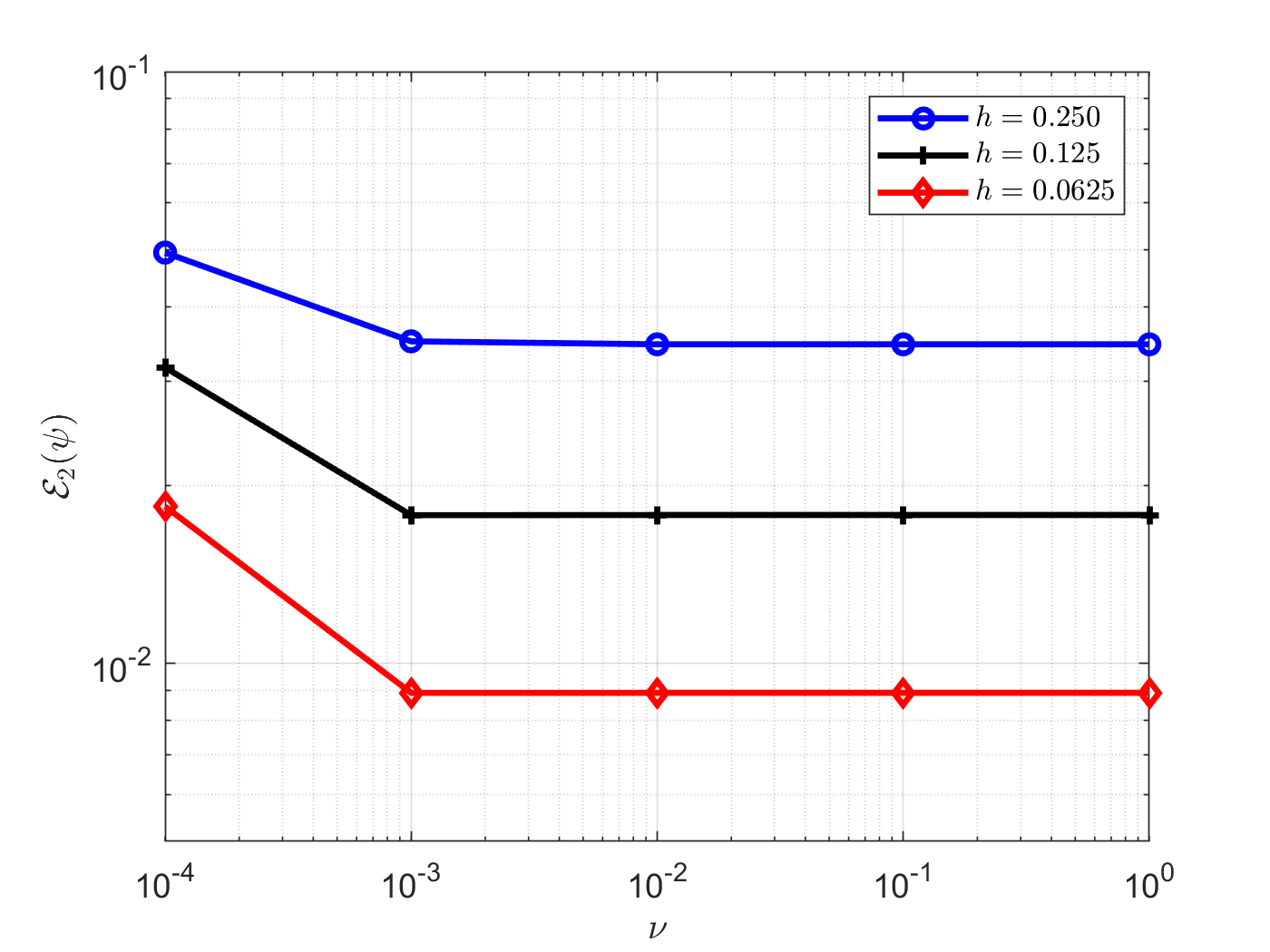}}
	\end{center}
	\caption{Test 4. Errors  of the stream-function $\mathcal{E}_2(\psi)$, using the VE scheme \eqref{NSE:stream:disc} with the alternative load term \eqref{load-alter-global}, for different values of $\nu$ and the mesh $\CT^1_h$.}
	\label{fig:robust} 
\end{figure}

\small
\subsection*{Acknowledgements} 	
The first author was partially supported by the National Agency for Research and Development,
ANID-Chile through FONDECYT Postdoctorado project 3200242.
The second author was partially supported by the National Agency
for Research and Development, ANID-Chile through FONDECYT project 1220881,
by project {\sc Anillo of Computational Mathematics for Desalination Processes} ACT210087,
and by project Centro de Modelamiento Matem\'atico (CMM), FB210005,
BASAL funds for centers of excellence.
The third author was supported by the National Agency for Research and
Development, ANID-Chile, Scholarship Program, Doctorado Becas Chile 2020, 21201910.
%-----------------------------------------------------------------------

\end{document}